\documentclass[12pt]{amsart}
\usepackage[margin=0.9in]{geometry} 
\usepackage{esint,amsthm,amsmath,amssymb,amsfonts,graphicx,color,comment,enumerate,caption,bbm,bm,hyperref,enumitem,mathrsfs}
\usepackage{mathtools}

\usepackage{cleveref}
\allowdisplaybreaks

\DeclareMathOperator*{\esssup}{ess\,sup}
\DeclareMathOperator\supp{supp}

\newtheorem{lemma}{Lemma}[section]
\newtheorem{remark}{Remark}[section]
\numberwithin{equation}{section}
\newtheorem{theorem}{Theorem}[section]
\newtheorem{proposition}[theorem]{Proposition}
\newtheorem{definition}[theorem]{Definition}

\title[Schr\"odinger pseudo-multiplier]{On Schr\"odinger Pseudo-Multipliers and their Commutators}

\author[S. Bagchi, R. Basak, J. Singh, M. N. Vempati]
{Sayan Bagchi \and Riju Basak \and Joydwip Singh \and Manasa N. Vempati} 

\address[S. Bagchi]{Department of Mathematics and Statistics, Indian Institute of Science Education and Research Kolkata, Mohanpur--741246, West Bengal, India.}
\email{sayan.bagchi@iiserkol.ac.in}

\address[R. Basak]{Department of Mathematics, National Taiwan Normal University, Taipei City, 116059, Taiwan.}
\email{rijubasak52@ntnu.edu.tw}

\address[J. Singh]{Department of Mathematics and Statistics, Indian Institute of Science Education and Research Kolkata, Mohanpur--741246, West Bengal, India.}
\email{js20rs078@iiserkol.ac.in}

\address[M. N. Vempati]{Department of Mathematics, Louisiana State University, Baton Rouge, 70803, USA.}
\email{nvempati@lsu.edu}



\begin{document}

\begin{abstract}
In this article, we establish the unweighted and weighted $L^p$-boundedness of pseudo-multipliers associated with a class of Schr\"odinger operators, this generalizes the result of our first author and Thangavelu [Bagchi \& Thangavelu, J. Funct. Anal. 2015] for Hermite pseudo-multipliers. The weight classes we consider are tailored to this framework and strictly contain the classical Muckenhoupt $A_p$-classes. To establish the weighted boundedness, we prove a quantitative version of reverse H\"older's inequality and quantitative weighted estimates for general sparse operators, which are of independent interest. We also study commutators of Schr\"odinger pseudo-multipliers, establishing their boundedness and compactness results on these weighted $L^p$-spaces. 

\end{abstract}

\maketitle

\section{Introduction and Main results}
The theory of classical Fourier multipliers, originating with M. Riesz and later refined in the Mihlin–H\"ormander framework \cite{Hormander_Translation_invariant, Mihlin_Fourier_Multiplier}, characterizes translation-invariant linear operators through multiplication in the Fourier domain. Fourier multipliers find extensive applications in elliptic and dispersive partial differential equations, fluid dynamics, and signal processing. Pseudo-differential operators generalize this framework by permitting multipliers that vary with both space and frequency, allowing for a more precise analysis of singularities. This generalization has become a cornerstone of modern analysis, connecting classical Fourier multiplier techniques with broader operator-theoretic approaches \cite{Grafakos_Classical_Fourier, Stein_Weiss_Fourier_Analysis, Hormander_Pseudo_Differential_operator, Stein-book, Taylor_Pseudo_differential, Ruzhansky_Turunen_Book}. Recently, pseudo-multiplier operators have also been studied in the context of various self-adjoint operators, like Hermite operators, Grushin operators, see for example \cite{BagchiThangaveluHermitePseudo}, \cite{Ly_L2_boundedness_JFA}, \cite{Bui_Hermite_Pseudo_2020}, \cite{Bagchi_Basak_Garg_Ghosh_Grushin_2023}, \cite{Bernicot_Frey_Pseudo_semigroup_2014}. In this paper, our main results extend these ideas to pseudo-multipliers associated with a class of Schr\"odinger operators, which generalizes the result of Hermite pseudo-multipliers \cite{BagchiThangaveluHermitePseudo}. Another contribution of this paper is the boundedness and compactness result for Schr\"odinger pseudo-multipliers, which are new in the non-Euclidean setting. To present our results, we first review the existing literature on pseudo-differential operators in Euclidean spaces and their extensions to more general settings.

\medskip
Given a bounded measurable function $\sigma$ on $\mathbb{R}^n \times \mathbb{R}^n$, the corresponding pseudo-differential operator
 $\sigma(x, D)$ is defined by
\begin{align*}
    \sigma(x,D)f(x) = \frac{1}{(2\pi)^{n/2}} \int_{\mathbb{R}^n} \sigma(x, \xi) \widehat{f}(\xi)\, e^{i x \cdot \xi} \ d\xi ,
\end{align*}
for functions $f$ in Schwartz class, where $\widehat{f}$ denotes the Fourier transform of $f$ given as
\begin{align*}
    \widehat{f}(\xi) &= \frac{1}{(2\pi)^{n/2}} \int_{\mathbb{R}^n} f(x) e^{-i x \cdot \xi} \ dx .
\end{align*}

The function $\sigma$ is referred to as the symbol of the pseudo-differential operator $\sigma(x,D)$. When the symbol $\sigma$ is independent of $x$, such operators reduce to the classical Fourier multipliers $\sigma(D)$. In this case, the $L^2$-boundedness of $\sigma(D)$ follows directly from Plancherel's theorem together with the $L^{\infty}$-boundedness of $\sigma$. However, for general pseudo-differential operators, the boundedness of $\sigma$ in $L^{\infty}$ alone does not ensure $L^2$-boundedness. It is well known that additional regularity conditions on $\sigma$ are required to establish both $L^2$-boundedness and, more generally, $L^p$-boundedness for $p \neq 2$. To proceed further, we first introduce the relevant symbol classes.

For $m \in \mathbb{R}$ and $\rho, \delta \geq 0$, we denote by $S^m_{\rho, \delta}(\mathbb{R}^n)$, the set of all smooth functions on $\mathbb{R}^n \times \mathbb{R}^n$, which satisfies the following condition
\begin{align*}
    |\partial_x^{\alpha} \partial_{\xi}^{\beta} \sigma(x, \xi)| &\leq C_{\alpha, \beta} (1+|\xi|)^{m-\rho|\beta|+\delta|\alpha|} ,
\end{align*}
for all $\alpha, \beta \in \mathbb{N}^n$.

By the celebrated theorem of Calder\'on and Vaillancourt \cite{Calderon_Vaillancourt_Pseudo_differential_1971}, the pseudo-differential operator $\sigma(x,D)$ extends to a bounded operator on $L^2(\mathbb{R}^n)$ whenever $\sigma \in S^0_{\rho,\delta}(\mathbb{R}^n)$ for $0 \leq \delta \leq \rho \leq 1$ with $\delta \neq 1$. In the special case $\sigma \in S^0_{1,0}(\mathbb{R}^n)$, Laptev \cite{Laptev_Fourier_Integral_Op} showed that $\sigma(x,D)$ belongs to the class of Calder\'on-Zygmund operators. More generally, if $\sigma \in S^0_{1,\delta}(\mathbb{R}^n)$ with $0 \leq \delta < 1$, then $\sigma(x,D)$ is bounded on $L^p(\mathbb{R}^n)$ for all $1 < p < \infty$ (see \cite[Proposition~4]{Stein-book}). For the class $S^0_{1,0}(\mathbb{R}^n)$, Miller \cite{Miller_Pseudo_differential_1982} established weighted $L^p$-boundedness of $\sigma(x,D)$ for the Muckenhoupt $A_p$ weights with $1 < p < \infty$. Subsequently, Tang \cite{Tang_Pseudo_Commutators_2012} extended these results to symbols in $S^0_{1,\delta}(\mathbb{R}^n)$ with $0 \leq \delta < 1$, obtaining weighted $L^p$-boundedness for weight classes strictly larger than the classical Muckenhoupt classes. In the last decade, quantitative weighted estimates for a wide range of operators in harmonic analysis have become a central topic of research. In particular, Beltr\'an and Cladek \cite{Beltran-Cladek} investigated such estimates for pseudo-differential operators using the method of sparse domination, and \cite[Theorem~1.3]{Beltran-Cladek} yields the following result.

\begin{theorem}
    Let $\sigma \in S^0_{1,0}(\mathbb{R}^n)$. Then for $1\leq r <p<\infty$ and $\omega \in A_{p/r}$, there exists a constant $C>0$ independent of $\omega$ such that
    \begin{align*}
        \|\sigma(x,D) f\|_{L^p(\omega)} &\leq C [\omega]_{A_{p/r}}^{\max\left\{ \frac{1}{p - r}, 1 \right\}} \|f\|_{L^p(\omega)} .
    \end{align*}
\end{theorem}

Let us now turn our attention to the commutators of the pseudo-differential operators. For a function $b$, the commutator of an operator $T$ is denoted by $[b,T]$ and defined by
\begin{align*}
    [b, T]f &:= b\, T(f)-T(bf) .
\end{align*}
The study of commutators was initiated by Calder\'on \cite{Calderon_Commutators} and later developed by Coifman, Rochberg, and Weiss \cite{Coifman_Rochberg_Weiss_Commutators}. Commutators play a fundamental role in the analysis of the regularity of solutions to second-order elliptic partial differential equations.

When $T$ is the pseudo-differential operator $\sigma(x,D)$, it is known that the corresponding commutator $[b,\sigma(x,D)]$ associated with the symbol class $S^{0}_{1,\delta}(\mathbb{R}^n)$ for $0 \leq \delta < 1$ is bounded on $L^p(\mathbb{R}^n)$ for $1 < p < \infty$, and also on weighted spaces $L^p(\omega)$ with $\omega \in A_p$ (see \cite{Alvarez_Bagby_Kurtz_Weighted_estimates, Michalowski_Rule_Staubach_Pseudo_differential}). In \cite{Tang_Pseudo_Commutators_2012}, Tang extended the previously known weighted boundedness results for commutators of pseudo-differential operators with symbols in $S^{0}_{1,0}(\mathbb{R}^n)$ by considering more general weight classes that strictly contain the classical $A_p$ weights. More recently, Wen, Wu, and Xue \cite{Wen_Wu_Xue_pseudo_commutator_2020} studied quantitative weighted estimates for commutators of pseudo-differential operators with symbols $\sigma \in S^{m}_{1,0}(\mathbb{R}^n)$, $m < 0$, using the method of sparse domination.

The compactness of commutators has been extensively studied over the years. The problem was initiated by Uchiyama \cite{Uchiyama_Compactness_Commutators} for the Hankel-type operators and subsequently developed for various important operators; see \cite{Chen_Duong_Li_Wu_Compactness_Riesz_Transform_Stratified_group_2019}, \cite{Tao_Yang_Yuan_Zhang_compact_ball_Banach_2023}, \cite{Duong_Li_Mo_Wu_Yang_compactness_Bessel_2018} and references therein. Recently, in \cite{Guo_Zhou_Compactness_Pseudo_differential_2019}, the authors studied the compactness of pseudo-differential commutators for the general weight class introduced by \cite{Tang_Pseudo_Commutators_2012}, with symbol $\sigma \in S^{0}_{1, 0}(\mathbb{R}^n)$.

The theory of pseudo-differential operators has been developed beyond the Euclidean setting. For example, on compact Lie groups, Heisenberg groups, and more generally on graded Lie groups, see \cite{Ruzhansky_Turunen_Book}, \cite{Bahouri_Fermanian_Gallagher_PseudodifferentialHeisenberg}, \cite{Cardona_Delgado_Ruzhansky_JGA2021}, \cite{Fischer_Ruzhansky_QuantizationBook}, and the references therein. Outside of the group settings, pseudo-differential operators (pseudo-multipliers) have also been studied for various nonnegative self-adjoint operators on doubling metric measure spaces. In particular, for the Hermite operator see \cite{EppersonHermitePseudo}, \cite{BagchiThangaveluHermitePseudo}, \cite{Bui_Hermite_Pseudo_2020}, \cite{Ly_L2_boundedness_JFA}, \cite{Ly_Hermite_Weighted}; for the Grushin operator see \cite{Bagchi_Garg_L2_boundedness_JFA}, \cite{Bagchi_Basak_Garg_Ghosh_Grushin_2023}, \cite{Bagchi_Basak_Garg_Ghosh_Operator_valued_2024}. More generally, for any nonnegative self-adjoint operator on a space of homogeneous type, under suitable assumptions on the operator, Bernicot and Frey \cite{Bernicot_Frey_Pseudo_semigroup_2014} studied the unweighted $L^p$-boundedness of pseudo-differential operators. Quantitative weighted estimates for pseudo-multipliers in the case of Grushin operators were first obtained in \cite{Bagchi_Basak_Garg_Ghosh_Grushin_2023}, \cite{BBGG-2}, and \cite{Bagchi_Basak_Garg_Ghosh_Operator_valued_2024} using the sparse domination technique.

In this paper, we investigate pseudo-differential operator analogues (pseudo-multipliers) associated with Schr\"odinger operators and establish quantitative weighted $L^p$-boundedness results. Furthermore, we prove quantitative weighted boundedness and compactness for commutators of Schr\"odinger pseudo-multipliers. To present our main results, we begin with preliminaries on Schr\"odinger operators.

\subsection{Schr\"odinger pseudo-multiplier}
Consider the Schr\"odinger operator on $\mathbb{R}^n$ $(n\geq 3)$ given by
\begin{align*}
    \mathcal{L} = -\Delta + V(x)
\end{align*}
Here $\Delta$ is the Laplace operator on $\mathbb{R}^n$ and $V$ is the non-negative potential function. We assume that $V$ is not identically zero and $V \in RH_q$ for some $q >n$, that is $V \in L^q_{loc}(\mathbb{R}^n), V \geq 0$, and there exists a constant $C>0$ such that the reverse H\"older inequality
\begin{align*}
    \left(\frac{1}{|B(x,r)|} \int_{B(x,r)} V^q(y) \ dy \right)^{1/q} \leq C \left(\frac{1}{|B(x,r)|} \int_{B(x,r)} V(y) \ dy \right) ,
\end{align*}
holds for every $x \in \mathbb{R}^n$ and $r>0$, where $B(x,r)$ denotes the ball centered at $x$ and radius $r$.

The operator $\mathcal{L}$ is a non-negative and self-adjoint operator on $L^2(\mathbb{R}^n)$. For every Borel measurable function $F: \mathbb{R} \to \mathbb{C}$, the spectral theorem gives
\begin{align*}
    F(\sqrt{\mathcal{L}}) &= \int_0^{\infty} F(\sqrt{\lambda}) \, dE(\lambda) ,
\end{align*}
$dE(\lambda)$ be the spectral resolution of the operator $\mathcal{L}$ for $\lambda \geq 0$. Then $F(\sqrt{\mathcal{L}})$ is bounded on $L^2(\mathbb{R}^n)$ if and only if the spectral multiplier $F$ is $E$-essentially bounded. For other $p \in (1, \infty)$, the $L^p$ boundedness of $F(\sqrt{\mathcal{L}})$ was studied by Hebisch \cite{Hebisch_Multiplier_Schrodinger}. In particular, when $V(x) = |x|^2$, the operator $H = -\Delta + |x|^2$ is the well-known Hermite operator and the corresponding spectral multiplier has been studied by \cite{Mauceri_Weyl_Transform_JFA}, \cite{Thangavelu_Multiplier_Hermite}.

Now we define the pseudo multiplier associated with the Schr\"odinger operator. For that, first we need to define the symbol classes.
\begin{definition}
    For any $m \in \mathbb{R}$ and $\delta \geq 0$, we say that $\sigma \in C^{\infty}(\mathbb{R}^n \times \mathbb{R})$ belongs to the symbol class $S^m_{1, \delta}(\sqrt{\mathcal{L}})$ if it satisfies the following estimate:
    \begin{align*}
        \left|\partial_x^{\alpha} \partial_{\eta}^l \sigma(x, \eta) \right| \leq C_{l, \alpha} (1+\eta)^{m-l+\delta|\alpha|}
    \end{align*}
    for all $l \in \mathbb{N}$ and $\alpha \in \mathbb{N}^n$.
\end{definition}

Let $\sigma: \mathbb{R}^n \times \mathbb{R} \to \mathbb{C}$ be a smooth function. Take $\psi_0 \in C_c^{\infty}(-2,2)$ and $\psi_1 \in C_c^{\infty}(1/2,2)$ such that $0\leq \psi_0, \psi_1 \leq 1$ and 
\begin{align*}
    \sum_{j=0}^{\infty} \psi_j(\eta) =1,
\end{align*}
where $\psi_j(\eta)=\psi_1(2^{-(j-1)}\eta)$ for $j \geq 2$. Using this partition of unity, we decompose $\sigma$ as
\begin{align}
\label{Equation: dyadic decomposition of pseudo multiplier}
    \sigma(x, \lambda) = \sum_{j=0}^{\infty} \sigma(x, \lambda) \psi_j(\lambda) =: \sum_{j=0}^{\infty} \sigma_j(x, \lambda).
\end{align}
Let us define $F(x, \lambda)= \sigma_j(x, 2^j \sqrt{\lambda}) e^{\lambda} $. Then we have $\sigma_j(x, \sqrt{\lambda})= F(x, 2^{-2j} \lambda) e^{-2^{-2j}\lambda}$. Therefore, by the Fourier inversion formula in the second variable, we get
\begin{align}
\label{Taking inverse Fourier transform}
    \sigma_j(x, \sqrt{\lambda}) &= \frac{1}{2 \pi} \int_{\mathbb{R}} \widehat{F}(x, \tau) \exp{({-(1-i \tau)2^{-2j} \lambda})} \,  d\tau , \quad x \in \mathbb{R}^n
\end{align}
where for $x$-fixed, $\widehat{F}(x, \tau)$ stands for the Fourier transform of $F$ with respect to the second variable.
Then we define
\begin{align}\label{Definition: original defintion of sigmaj}
    \sigma_j(x, \sqrt{\mathcal{L}})f(x) &:= \frac{1}{2 \pi} \int_{\mathbb{R}} \widehat{F}(x, \tau) \exp{({-(1-i \tau)2^{-2j} \mathcal{L}})}f(x) \,  d\tau .
\end{align}
For each $j \geq 0$, one can see that $\sigma_j(x, \sqrt{\mathcal{L}})f \in L^2(\mathbb{R}^n)$ for $f \in L^2(\mathbb{R}^n)$. Indeed, since for each $\tau \in \mathbb{R}$ we have $|\exp{({-(1-i \tau)2^{-2j} \lambda})}| \leq 1$, therefore an application of Plancherel's theorem yields
\begin{align}
\label{Uniform boundedness of sigmaj term}
    \|\sigma_j(x, \sqrt{\mathcal{L}})f\|_{L^2} &\leq \frac{1}{2 \pi} \int_{\mathbb{R}} \sup_{x \in \mathbb{R}^n}|\widehat{F}(x, \tau)| \|\exp{({-(1-i \tau)2^{-2j} \mathcal{L}})}f\|_{L^2} \,  d\tau \\
    &\nonumber \leq C \|f\|_{L^2} \int_{\mathbb{R}} \sup_{x \in \mathbb{R}^n}|\widehat{F}(x, \tau)| \,  d\tau \leq C \|f\|_{L^2} ,
\end{align}
where we have used the fact $F(x, \lambda)$ is $C_c^{\infty}(\mathbb{R})$ function with respect to second variable and $F \in L^{\infty}(\mathbb{R}^n \times \mathbb{R})$. This further implies that for any positive integer $N$ and $f \in L^2(\mathbb{R}^n)$, $\mathcal{T}_{N}f =\sum_{j=0}^{N} \sigma_j(x, \sqrt{\mathcal{L}})f$ also belongs to $L^2(\mathbb{R}^n)$.

Let us define
\begin{align*}
    \mathcal{D}(\mathcal{L}) &:= \{f \in L^2(\mathbb{R}^n) : f = E_{[0, N)}f \quad \text{for some}\ N>0\} .
\end{align*}
It is easy to see that $\mathcal{D}(\mathcal{L})$ is dense in $L^2(\mathbb{R}^n)$. Moreover, for $f \in \mathcal{D}(\mathcal{L})$ one can easily show that $\mathcal{T}_N f$ converges in $L^2(\mathbb{R}^n)$ as $N \to \infty$. In fact, first note that if $f \in \mathcal{D}(\mathcal{L})$ then we can write $f = \sum_{j=0}^{N(f)} \psi_j(\sqrt{L})f$ for some $N(f)>0$. Since the support of $\sigma_j (x, \cdot)$ with respect to the second variable is contained in $[2^{2j-4}, 2^{2j}]$, we have
\begin{align*}
    \sum_{j>N(f)+2} \sigma_j(x, \sqrt{\mathcal{L}})f = 0 .
\end{align*}
Therefore for $f \in \mathcal{D}(\mathcal{L})$ we can conclude that as $N \to \infty$, $\mathcal{T}_N f = \sum_{j=0}^{N} \sigma_j(x, \sqrt{\mathcal{L}})f$ converges in $L^2$, which we denote by $Tf := \sigma(x, \sqrt{\mathcal{L}})f$.

Then we can also show that the following $L^2$-boundedness result holds for $\mathcal{T}_N$.
\begin{proposition}\label{prop:L2-sigma-j}
Let $\sigma \in S^0_{1,0}(\sqrt{\mathcal{L}})$. Then for each positive number $N$,
\begin{align*}
    \|\mathcal{T}_{N}\|_{L^2\rightarrow L^2} \leq C ,
\end{align*}
for some positive constant $C$ independent of $N$.

Moreover, the operator $\mathcal{T}_N=\sum_{j=0}^{N} \sigma_j(x, \sqrt{\mathcal{L}})$ converges in the strong operator topology of $\mathcal{B}(L^2(\mathbb{R}^n))$.
\end{proposition}
Note that when defining the operator $\sigma_j(x, \sqrt{\mathcal{L}})$, we used the fact that $\sigma_j$ is supported on a compact set in the second variable. However, for symbols $\sigma \in S_{1,0}^0(\sqrt{\mathcal{L}})$ it is not clear immediately how one can make sense of the definition \eqref{Definition: original defintion of sigmaj}. Nevertheless, Proposition \ref{prop:L2-sigma-j} indicates that for each positive integer $N$, the operator norm of $\mathcal{T}_N$ is bounded uniformly over $N$ and converges in the strong operator topology of $\mathcal{B}(L^2(\mathbb{R}^n))$ as $N\rightarrow\infty$. Based on this observation, we will define our pseudo-multipliers for general symbols $\sigma$ as the limit of the partial sum operator $\mathcal{T}_N$ of the operators $\sigma_j(x, \sqrt{\mathcal{L}})$ as follows. 

For $\sigma \in S^0_{1,0}(\sqrt{\mathcal{L}})$, we define the corresponding pseudo-multiplier by
\begin{align*}
    \sigma(x,{\mathcal{\sqrt{L}}})=\lim_{N\rightarrow\infty} \sum_{j=0}^{N} \sigma_j(x,{\mathcal{\sqrt{L}}}).
\end{align*}
Here, the limit is taken in the sense of the strong operator topology of $\mathcal{B}(L^2(\mathbb{R}^n))$, that is $\lim_{N\rightarrow\infty} \|\sigma(x,{\mathcal{\sqrt{L}}})f-\sum_{j=0}^{N} \sigma_j(x,{\mathcal{\sqrt{L}}})f\|_{L^2(\mathbb{R}^n)}=0$ for $f\in L^2(\mathbb{R}^n)$. Since $\|\mathcal{T}_N \|_{L^2 \to L^2} \leq C$ uniformly in $N$, then uniform boundedness principle implies that $T= \sigma(x, \sqrt{\mathcal{L}})$ is bounded operator on $L^2(\mathbb{R}^n)$.

\begin{remark}
In case of the Hermite operator, that is when $\mathcal{L} = H$, the Hermite pseudo-multiplier $\sigma_j(x, \sqrt{H})$ takes the following form: from \eqref{Definition: original defintion of sigmaj} and the spectral decomposition of $H$ for suitable functions $f$ we obtain
\begin{align*}
    \sigma_j(x, \sqrt{H})f(x) &= \frac{1}{2 \pi} \int_{\mathbb{R}} \widehat{F}(x, \tau) \exp{({-(1-i \tau)2^{-2j} H})}f(x) \,  d\tau \\
    &= \sum_{k=0}^{\infty} \left(\frac{1}{2 \pi} \int_{\mathbb{R}} \widehat{F}(x, \tau) \exp{({-(1-i \tau)2^{-2j} (2k+n)})} \,  d\tau \right) P_k f(x) \\
    &= \sum_{k=0}^{\infty} \sigma_j(x, \sqrt{2k+n}) P_k f(x) ,
\end{align*}
where $P_k$ is the projection onto the $k$-th eigen space corresponding to the eigen value $2k+n$. The last expression coincides with the definition of Hermite pseudo-multiplier as described in \cite{BagchiThangaveluHermitePseudo}.
    
\end{remark}

Now we state our first main result.
\begin{theorem}
\label{Theorem: L^2 boundedness of the symbol S^0_{1,0}}
    Let $\sigma \in S^0_{1,0}(\sqrt{\mathcal{L}})$. Then the operator $\sigma(x, \sqrt{\mathcal{L}})$ is of weak type $(1,1)$ and bounded on $L^p(\mathbb{R}^n)$ for $1<p\leq 2$.
\end{theorem}

Now, before going to our next results, let us first define the weight classes associated with the Schr\"odinger operators. We define $\Psi_{\theta}(B) = (1+r_B/\rho(x_B))^{\theta}$, where $\theta\geq 0$, and $x_B$, $r_B$ denote the center and radius of $B:=B(x_B, r_B)$ respectively, and $\rho$ is the critical function defined by
\begin{align}\label{Expression: Auxiliary function}
    \rho(x) &:= \sup_{r>0} \left\{r: \frac{1}{r^{n-2}}\int_{B(x,r)} V(y)\, dy \leq 1 \right\} .
\end{align}
We say that a non-negative locally integrable function $\omega$ belongs to the weight class $A_p^{\rho, \theta}$ for $1<p<\infty$ and $\theta\geq 0$, if there is a constant $C$ such that for all balls $B=B(x,r)$, $r>0$ there holds
\begin{align}
\label{Definition: weight class}
    [\omega]_{A_p^{\rho, \theta}} := \sup_{B} \left(\frac{1}{\Psi_{\theta}(B)|B|}\int_B \omega(y) \, dy \right) \left(\frac{1}{\Psi_{\theta}(B)|B|}\int_B \omega^{-\frac{1}{p-1}}(y) \, dy \right)^{p-1} \leq C ,
\end{align}
and we say that $\omega$ belongs to the weight class $A_1^{\rho, \theta}$, if there exists a constant $C$ such that
\begin{align*}
    \mathcal{M}_{\rho, \theta}\omega(x) \leq C\, \omega(x) \quad \text{for a.e.}\ x \in \mathbb{R}^n ,
\end{align*}
where
\begin{align}
\label{Definition: Maximal function}
    \mathcal{M}_{\rho, \theta}f(x) &= \sup_{B \ni x} \frac{1}{\Psi_{\theta}(B) |B|} \int_B |f| \, dy .
\end{align}
This weight class was first introduced by Bongioanni, Harboure and Salinas in \cite{Bongioanni_Harboure_Salinas_Schrodinger_first}, where they studied weighted Lebesgue space boundedness of various important operators of harmonic analysis related to the Schr\"odinger operator, for example the Riesz transform, the fractional integral, and the square function. After that, a lot of work has been done on this weight class $A_p^{\rho, \theta}$, see \cite{Bongioanni_Harboure_Salinas_Schrodinger_second}, \cite{Tang_Weighted_Schrodinger_Forum_Math_2015}, \cite{Tang_Wang_Zhu_Weighted_Schrodinger_2019} and references therein. Note that for $\theta = 0$, the weight class $ A_{p}^{\rho, 0}$ coincides with the classical Muckenhoupt's $A_p$ weight class. Since $\Psi_{\theta}(B) \geq 1$, we have $A_p \subset A_p^{\rho, \theta}$ for $1<p<\infty$. It is also known that this $A_p^{\rho, \theta}$ weight class is strictly larger than the classical weight class $A_p$, for more details see \cite[Section 2]{Tang_Weighted_Schrodinger_Forum_Math_2015}.

In the sequel we denote the operator $\sigma(x, \sqrt{\mathcal{L}})$ by $T$. Now we state our second main result, which yields the following quantitative weighted estimate for the operator $T$.
\begin{theorem}
\label{Theorem: Weighted L^p boundedness of pseudo multiplier}
    Let $\sigma \in S^0_{1,0}(\sqrt{\mathcal{L}})$. Then for $1\leq r <p<\infty$ and $\omega \in A_{p/r}^{\rho, \theta}$, there exists a constant $C>0$ independent of $\omega$ such that
    \begin{align*}
        \|T f\|_{L^p(\omega)} &\leq C [\omega]_{A_{p/r}^{\rho, \theta}}^{\max\left\{ \frac{1}{p - r}, 1 \right\}} \|f\|_{L^p(\omega)} .
    \end{align*}
\end{theorem}

Note that in \cite{BagchiThangaveluHermitePseudo}, the first author and Thangavelu studied weighted $L^p$ estimates for the pseudo-multiplier associated with the Hermite operator $H$ for weights in the Muckenhoupt $A_p$ class. Therefore, Theorem \ref{Theorem: Weighted L^p boundedness of pseudo multiplier} generalizes Theorem 1.4 of \cite{BagchiThangaveluHermitePseudo} in two directions: first, it applies to the more general Schr\"odinger operator $\mathcal{L}$; second, it holds for a broader weight class than the Muckenhoupt $A_p$ class.

Next, we discuss the quantitative weighted estimates for commutators for pseudo-multipliers associated to the Schr\"odinger operators. In order to state our result, we have to define the space of bounded mean oscillation associated to the Schr\"odinger operator $\mathcal{L}$. The bounded mean oscillation space associated to the Schr\"odinger operator $\mathcal{L}$ is denoted by $BMO_{\theta}(\rho)$ and is defined by
\begin{align*}
    \|f\|_{BMO_{\theta}(\rho)} &= \sup_{B \subset \mathbb{R}^n} \frac{1}{\Psi_{\theta}(B)|B|} \int_{B} |f(x)-f_B| \, dx < \infty .
\end{align*}
where $f_B = \frac{1}{|B|} \int_B f$. Note that $BMO \subset BMO_{\theta}(\rho)$, where $BMO$ is the classical bounded mean oscillation space in $\mathbb{R}^n$. Actually, this containment is strict, for details see \cite[Page 116]{Bongioanni_Harboure_Salinas_Schrodinger_JFFA_2011}. 

Regarding the operator $[b, T]$, the following is our first main result in this direction, which provides a quantitative version of weighted $L^p$-boundedness result.
\begin{theorem}
\label{Theorem: Weighted L^p boundedness of pseudo commutator}
    Suppose $\sigma \in S^0_{1,0}(\sqrt{\mathcal{L}})$. Let $b \in BMO_{\theta}(\rho)$, $1<p<\infty$ and $\omega \in A_p^{\rho, \theta}$. Then there exists a constant $C>0$ such that 
    \begin{align*}
        \|[b, T]f\|_{L^p(\omega)} &\leq C [\omega]_{A_{p}^{\rho, \theta}}^{2\max\left\{ \frac{1}{p - 1}, 1 \right\}} \|b\|_{BMO_{\theta}(\rho)} \|f\|_{L^p(\omega)} . 
    \end{align*}
\end{theorem}

The next theorem gives the compactness of commutators for pseudo-multipliers associated with the Schr\"odinger operators. Define $CMO_{\theta}(\rho)$ to be the closure of $C_c^{\infty}(\mathbb{R}^n)$ functions in the topology of $BMO_{\theta}(\rho)$.

\begin{theorem}
\label{Theorem: Weighted compactness of pseudo commutator}
    Suppose $\sigma \in S^0_{1,0}(\sqrt{\mathcal{L}})$. Let $b \in CMO_{\theta}(\rho)$, $1<p<\infty$ and $\omega \in A_p^{\rho, \theta}$. Then the commutator $[b, T]$ is a compact operator on $L^p(\omega)$.
\end{theorem}

To the best of our knowledge, compactness of commutators for pseudo-multipliers in the non-Euclidean setting have not been previously studied. In this regard, Theorems \ref{Theorem: Weighted L^p boundedness of pseudo commutator} and \ref{Theorem: Weighted compactness of pseudo commutator} may be viewed as an initial contributions to this line of research.

\subsection{Examples:}
Let us discuss some examples where our results can be applied. If we take $V(x) = |x|^{\alpha}$ with $\alpha q > -n$, then $V \in RH_q$ (see \cite{Shen_Schrodinger_operator_certain_potential_1995}). To apply our results, we also need $q>n$. In particular, when $\alpha=2$, that is $V(x)=|x|^2$, the pseudo-multiplier associated with the Hermite operator has been studied by many mathematicians, see, for example, \cite{EppersonHermitePseudo, BagchiThangaveluHermitePseudo, Bui_Hermite_Pseudo_2020, Ly_L2_boundedness_JFA, Bui_Duong_Ly_Hermite_Non_smooth_2026}.

Another interesting example is $V(x) = |x_n- \varphi(x')|^{\alpha}$ with $\alpha>0$, where $x' = (x_1, \ldots, x_{n-1}) \in \mathbb{R}^{n-1}$ and $\varphi : \mathbb{R}^{n-1} \to \mathbb{R}$ is a Lipschitz function. Then $V \in RH_q$ for all $q \in (1, \infty)$ (see \cite{Shen_Schrodinger_operator_certain_potential_1995}).

\medskip

The organization of our paper is as follows. In Section \ref{Section: preliminaries}, we recall various basic and important results related to the Schr\"odinger operator. We prove the quantitative weighted estimates of sparse operators and a quantitative version of reverse H\"older's inequality for the weight class associated with the Schr\"odinger operator $\mathcal{L}$ in Section \ref{Section: Sparse and reverse Holder}. Section \ref{Section: Kernel estimate} is devoted to the kernel estimates of the Schr\"odinger pseudo-multiplier $\sigma(x, \sqrt{\mathcal{L}})$. The unweighted $L^p$-boundedness of $\sigma(x, \sqrt{\mathcal{L}})$ (Proposition \ref{prop:L2-sigma-j} and Theorem \ref{Theorem: L^2 boundedness of the symbol S^0_{1,0}}) is proved in Section \ref{section: unweighted Lp boundedness of Schrodinger}, while the quantitative weighted $L^p$-boundedness of $\sigma(x, \sqrt{\mathcal{L}})$ (Theorem \ref{Theorem: Weighted L^p boundedness of pseudo multiplier}) is proved in Section \ref{Section: weighted Lp bound for operator}. In Section \ref{Section: Weighted boundedness of commutators} we prove quantitative weighted boundedness of the commutators of the Schr\"odinger pseudo-multiplier (Theorem \ref{Theorem: Weighted L^p boundedness of pseudo commutator}). Finally in Section \ref{Section: Compactness of commutators} we prove the compactness result (Theorem \ref{Theorem: Weighted compactness of pseudo commutator}) for the commutators associated to the Schr\"odinger pseudo-multiplier.

\medskip

\textbf{Notation:}
In this paper, we adopt standard notation. We denote $Q=Q(x_Q, r_Q)$, a cube centered at $x_Q$ and side length $r_Q$. For $\eta>0$, we write $\eta Q$ as the cube with side length $\eta$ times the side length of the cube $Q$. For a measurable set $E \subseteq \mathbb{R}^n$ and non-negative function $\omega$, we denote $\omega(E) = \int_E \omega $ and $\omega_E = \frac{\omega(E)}{|E|}$. For two non-negative numbers $N_1$ and $N_2$, by the expression $N_1\lesssim N_2$, we mean that there exists a constant $C>0$ such that $N_1\leq C N_2$. Whenever the implicit constant depends on $\epsilon$, we write $N_1\lesssim_{\epsilon}N_2$. We write $N_1\sim N_2$, when both $N_1\lesssim N_2$ and $N_2\lesssim N_1$ hold.

\section{Preliminaries}
\label{Section: preliminaries}
In this section, we give some basic details about the Schr\"odinger operator $\mathcal{L}$. First, recall that $\rho$ is the critical function defined in \eqref{Expression: Auxiliary function}. One can easily see that $0<\rho(x)<\infty$ if $V \neq 0$. Also, in particular one can check $\rho(x) \sim 1$ when $V = 1$ and $\rho(x) \sim \frac{1}{(1+|x|)}$ when $V=|x|^2$.

The following lemma provides a pointwise comparison for the critical function $\rho(x)$, which will be useful later in the proof. 
\begin{lemma}\cite[Corollary 1.5]{Shen_Schrodinger_operator_certain_potential_1995}
\label{Lemma: Equivalence of two potential}
    For all $x, y \in \mathbb{R}^n$, there exists constants $l_0>0, C_0>1$ and $C>0$ such that
    \begin{align*}
        \frac{1}{C_0}\left(1+\frac{|x-y|}{\rho(x)} \right)^{-l_0} \leq \frac{\rho(y)}{\rho(x)} \leq C_0 \left(1+\frac{|x-y|}{\rho(x)} \right)^{l_0/(l_0+1)}.
    \end{align*}
    Moreover, for any $\kappa>0$ and $x \in \mathbb{R}^n$, we have
    \begin{align*}
        C (1+\kappa)^{-l_0} \rho(x) \leq \rho(y) \leq C (1+\kappa)^{l_0/(l_0+1)} \rho(x) ,
    \end{align*}
    for all $y \in B(x, \kappa \rho(x))$.
\end{lemma}

\begin{remark}
\label{Remark: Equivalence of potential for all x,y}
For all $x, y \in \mathbb{R}^n$, from Lemma \ref{Lemma: Equivalence of two potential}, immediately we have
    \begin{align*}
        \frac{1}{C_0 \left(1+\frac{|x-y|}{\rho(x)}\right)^{l_0+1}} \lesssim \frac{1}{\left(1+\frac{|x-y|}{\rho(y)}\right)} \lesssim \frac{C_0}{ \left(1+\frac{|x-y|}{\rho(x)}\right)^{1/(l_0+1)}} .
    \end{align*}
\end{remark}

Now we discuss various maximal functions and their boundedness properties. For $r>0$ we define
\begin{align*}
    \mathcal{M}_{r, \rho, \theta}f(x) &= \sup_{B \ni x} \frac{1}{\Psi_{\theta}(B)} \left(\frac{1}{|B|} \int_B |f(y)|^r \, dy \right)^{1/r} .
\end{align*}
When $r=1$, we denote $\mathcal{M}_{1, \rho, \theta}f$ by $\mathcal{M}_{\rho, \theta}f$. When $\theta=0$, we denote $\mathcal{M}_{r,0,0}f$ by $\mathcal{M}_r f$, which is the standard Hardy-Littlewood maximal function. Then it is easy to see that
\begin{align*}
    |f(x)| \leq \mathcal{M}_{r, \rho, \theta}f(x) \leq \mathcal{M}_r f(x) \quad \text{for a.e.}\ x \in \mathbb{R}^n .
\end{align*}
Since
\begin{align*}
    \Psi_{\theta}(B) \leq \Psi_{\theta}(2 B) \leq 2^{\theta} \Psi_{\theta}(B) ,
\end{align*}
we can replace balls by cubes in the definition of $\mathcal{M}_{r, \rho, \theta}$ and again we denote it by the same notation. The same is true for the definition of $A_p^{\rho, \theta}$-weights (see \eqref{Definition: weight class}).

The following proposition gives the weighted $L^p$-boundedness of $\mathcal{M}_{r, \rho, \theta}$.
\begin{proposition}\cite[Proposition 2.7]{Bui_Bui_Duong_square_function_new_weight_2022}
\label{Prop: Maximal function new weight with r power}
For $r \in (0, \infty)$, $p \in (r, \infty)$ and $\theta \geq 0$, we have
\begin{align*}
    \|\mathcal{M}_{r, \rho, \theta}f\|_{L^p(\omega)} &\leq C [\omega]_{A_{p/r}^{\rho, \theta(p-r)}}^{\frac{1}{p-r}} \|f\|_{L^p(\omega)} .
\end{align*}
\end{proposition}

Next, we define the dyadic Hardy–Littlewood maximal function associated with this weight class and discuss its mapping properties, which will be needed later in the proofs of the main theorems. In order to do so let us first recall some properties of dyadic system (see \cite{Bui_Bui_Duong_square_function_new_weight_2022}).

\begin{lemma}\cite[Proposition 2.2]{Bui_Bui_Duong_square_function_new_weight_2022}
\label{Lemma: System of dyadic cubes}
There exists a finite collection of dyadic system $\{\mathscr{D}^k : k=1, \ldots, M_0\}$ and $\Lambda_0 >1$ such that for any ball $B \subset \mathbb{R}^n$ there exists a cube $Q \in \mathscr{D}^k $ for some $k$ and satisfies $B \subset Q \subset \Lambda_0 B$.
    
\end{lemma}

\begin{remark}
\label{Remark: system of dyadic cube}
The above lemma is also true if we replace $B$ by the cube $Q$.
\end{remark}

Let us define $\mathscr{D} = \bigcup_{k=1}^{M_0} \mathscr{D}^k$. Given a weight $\omega$, we define the weighted dyadic Hardy-Littlewood maximal operator  $\mathcal{M}_{\omega}^{\mathscr{D}}$ by
\begin{align*}
    \mathcal{M}_{\omega}^{\mathscr{D}} f(x) = \sup_{\substack{Q \in \mathscr{D} : \\ Q \ni x}} \frac{1}{\omega(Q)} \int_Q |f(y)| \, \omega(y) \, dy.
\end{align*}

Regarding the $L^p$-boundedness of $\mathcal{M}_{\omega}^{\mathscr{D}}$ we have the following result.

\begin{proposition}\cite[Theorem 15.1]{Lrner_Nazarov_Dyadic_Calculus_2019}
\label{Prop: Boundedness of dyadic maximal functions}
The operator $\mathcal{M}_{\omega}^{\mathscr{D}}$ is weak $(1,1)$ and strong $(p,p)$ for $1<p\leq \infty$. That is for $\alpha>0$, we have
\begin{align*}
    \omega(\{x \in \mathbb{R}^n : \mathcal{M}_{\omega}^{\mathscr{D}}f(x) > \alpha \}) \leq \frac{1}{\alpha} \|f\|_{L^1(\omega)} ,
\end{align*}
and
\begin{align*}
    \|\mathcal{M}_w^{\mathscr{D}}f\|_{L^p(\omega)} \leq \frac{p}{p-1} \|f\|_{L^p(\omega)} .
\end{align*}
    
\end{proposition}

Let us state some known results related to $BMO_{\theta}(\rho)$, which will be useful in the proof of Theorem \eqref{Theorem: Weighted L^p boundedness of pseudo commutator}.
\begin{lemma}\cite[Proposition 4.3]{Tang_Weighted_Schrodinger_Forum_Math_2015}
\label{Lemma: Complementary exponential function with BMO function}
    Suppose $f \in BMO_{\theta}(\rho)$. Then there exists $\gamma, C>0$ such that
    \begin{align*}
        \sup_{B} \frac{1}{|B|} \int_B \exp\left\{\frac{\gamma}{\|f\|_{BMO_{\theta}(\rho)} \Psi_{\theta'}(B)} |f(x)-f_B| \right\} \ dx \leq C ,
    \end{align*}
    where $\theta'= \theta (l_0+1)$.
\end{lemma}

Next, we discuss some facts about Luxemburg norm and the related maximal function. Let $\Phi$ be a nonnegative, continuous, convex, increasing function on $[0, \infty)$ with $\Phi(0)=0$ and $\Phi(t) \to \infty$ as $t \to \infty$. Such functions are called Young functions. The $\Phi$-average of a function $f$ over a cube $Q$ is defined by the following Luxemberg norm
\begin{align*}
    \|f\|_{\Phi, Q} &= \inf\left\{\lambda>0 : \frac{1}{|Q|} \int_Q \Phi \left(\frac{|f(y)|}{\lambda} \right) dy \leq 1 \right\}.
\end{align*}
In this case, we have the following generalized H\"older's inequality
\begin{align}
\label{Inequality: generalized Holders inequality}
    \frac{1}{|Q|} \int_Q |f(y) g(y)| dy &\leq \|f\|_{\Phi, Q} \|g\|_{\Bar{\Phi}, Q},
\end{align}
where $\Bar{\Phi}$ is the complementary Young function associated with $\Phi$. With the help of the Luxemburg norm, we define the corresponding maximal function by
\begin{align*}
    \mathcal{M}_{\Phi}f(x) &= \sup_{Q \ni x} \|f\|_{\Phi, Q} ,\quad \quad
    \text{and} \quad \quad \mathcal{M}_{\Phi, \rho, \theta} f(x) = \sup_{Q \ni x} \frac{1}{\Psi_{\theta}(Q)} \|f\|_{\Phi, Q} .
\end{align*}

In the following, we take $\Phi(t) = t \log(e+t)$, with the corresponding maximal function denoted by $\mathcal{M}_{\mathcal{L} \log \mathcal{L}, \rho, \theta}$. Then the complementary Young function is given by $\Bar{\Phi}(t) \sim e^t$, and the associated maximal function is denoted by $\mathcal{M}_{\exp{\mathcal{L}}, \rho, \theta}$. For more details, related to the Luxemberg norm, see \cite{rao_Ren_Theory_Orlich_1991}, \cite{Tang_Pseudo_Commutators_2012}, \cite{Wen-Wu_Xue_Sparse_rho_variation_2022}. 

As an application of Lemma \ref{Lemma: Complementary exponential function with BMO function}, we have the following result (see also \cite[Section 2]{Tang_Wang_Zhu_Weighted_Schrodinger_2019}):
\begin{align}
\label{Exponential norm of BMO function}
    \|b-b_Q\|_{\exp \mathcal{L}, Q} \leq C \|b\|_{BMO_{\theta}(\rho)} \Psi_{\theta'}(Q) ,
\end{align}
where $\theta'= \theta (l_0+1)$.

From Lemma 4.6 of \cite{Tang_Wang_Zhu_Weighted_Schrodinger_2019}, putting $\beta=0$, we get the following result, which we will use in the proof of Theorem \eqref{Theorem: Weighted L^p boundedness of pseudo commutator}.
\begin{lemma}
\label{Lemma: Relation between variant maximal function and L log L maximal function}
    Let $\mathcal{M}_{\rho, \theta/2}f$ be locally integrable. Then there exists positive constant $C_1$ and $C_2$ independent of $f$ and $x$ such that
    \begin{align*}
        C_2 \mathcal{M}_{\rho, \theta} \mathcal{M}_{\rho, \theta}f(x) \leq \mathcal{M}_{\mathcal{L} \log \mathcal{L}, \rho, \theta}f(x) \leq C_1 \mathcal{M}_{\rho, \theta/2} \mathcal{M}_{\rho, \theta/2}f(x) .
    \end{align*}
\end{lemma}

\section{Sparse operators and reverse H\"older's inequality}
\label{Section: Sparse and reverse Holder}
In this section, we define sparse family and discuss various properties of the sparse operators associated to the Schr\"odinger operator $\mathcal{L}$. Also we prove a quantitative version of reverse H\"older's inequality associated with the weight class $A_p^{\rho, \theta}$.

We say a collection $ \mathcal{S}$ of cubes of $\mathbb{R}^n$ to be an $\eta$-sparse family (for some $0<\eta<1$) if for every member $Q \in \mathcal{S}$ there exists a set $E_{Q} \subseteq Q$ such that $|E_{Q}|\geq \eta |Q|$, and the sets $\{E_{Q}\}_{Q \in {\mathcal{S}}}$ are pairwise disjoint.

Let $ p_0, r \in [1, \infty) $, $ p \in (r, \infty) $. For a sparse family $ \mathcal{S} $, we define the sparse operator
\begin{align*}
\mathcal{A}_{\mathcal{S}}^{r, p_0, \rho, N} f(x) &=
\left[\sum_{Q \in \mathcal{S}}\left( \frac{1}{|Q|} \int_Q |f|^{r} \right)^{\frac{p_0}{r}} \Psi_N(Q)^{-p_0} \chi_Q(x) \right]^{1/p_0},
\end{align*}
where
\begin{align*}
    \Psi_N(Q) &= \left(1 + \frac{r_Q}{\rho(x_Q)}\right)^N, \quad N > 0,
\end{align*}
and \( x_Q, r_Q \) denote the center and side length of \( Q \), respectively.

The following proposition gives the quantitative weighted $L^p$-estimate of the sparse operator $\mathcal{A}_{\mathcal{S}}^{r, p_0, \rho, N} $. For $p_0 =1$, similar result was proved in \cite[Theorem 2.6]{Bui_Bui_Duong_Singular_JGA_2021}.
\begin{proposition}
\label{Prop: Boundedness of sparse operator}
Let $ \mathcal{S} $ be an $ \eta $-sparse family for $0< \eta < 1$ and take $ p_0, r \in [1, \infty) $. For $ p \in (r, \infty) $ and $ \omega \in A_{p/r}^{\rho, \theta} $, we have
\begin{align*}
    \left\| \mathcal{A}_{\mathcal{S}}^{r, p_0, \rho, N} f \right\|_{L^p(\omega)} &\lesssim [\omega]_{A_{p/r}^{\rho, \theta}}^{\max\left\{ \frac{1}{p - r}, \frac{1}{p_0} \right\}} \|f\|_{L^p(\omega)},
\end{align*}
for $ N \geq \theta \cdot \frac{p}{r} \max\left\{ \frac{1}{p - r}, \frac{1}{p_0} \right\} $.
\end{proposition}

\begin{proof}
Note that it is enough to prove for $ p \geq r + p_0 $. If $p-r < p_0$, then $\max\left\{ \frac{1}{p - r}, \frac{1}{p_0} \right\} = \frac{1}{p-r}$. Hence, using the fact $\| \cdot \|_{\ell^{p_0}} \leq \| \cdot \|_{\ell^{p-r}}$ and from the case $p-r=p_0$ we get
\begin{align*}
    \left\| \mathcal{A}_{\mathcal{S}}^{r, p_0, \rho, N} f \right\|_{L^p(\omega)} &= \left\|\left[\sum_{Q \in \mathcal{S}}\left( \frac{1}{|Q|} \int_Q |f|^{r} \right)^{\frac{p_0}{r}} \Psi_N(Q)^{-p_0} \chi_Q(x) \right]^{1/p_0} \right\|_{L^p(\omega)} \\
    &\leq \left\|\left[\sum_{Q \in \mathcal{S}}\left( \frac{1}{|Q|} \int_Q |f|^{r} \right)^{\frac{p-r}{r}} \Psi_N(Q)^{-(p-r)} \chi_Q(x) \right]^{1/(p-r)} \right\|_{L^p(\omega)} \\
    &\lesssim [\omega]_{A_{p/r}^{\rho, \theta}}^{\max\left\{ \frac{1}{p - r}, \frac{1}{p_0} \right\}} \|f\|_{L^p(\omega)} .
\end{align*}
Therefore, we assume $ p \geq r + p_0 $. Note that because of remark \ref{Remark: system of dyadic cube}, we can also assume $\mathcal{S} \subset \mathscr{D}$. Let us set $ q := \left(p/p_0 \right)' = \frac{p}{p - p_0}$. For $ \omega \in A_{p/r}^{\rho, \theta} $, take a function $ g \in L^q(\omega^{1 - q}) = \left(L^{p/p_0}(\omega)\right)^* $. Then by duality, we can write
\begin{align*}
    \left\| \mathcal{A}_{\mathcal{S}}^{r, p_0, \rho, N} f \right\|_{L^p(\omega)}^{p_0} &= \left\| |\mathcal{A}_{\mathcal{S}}^{r, p_0, \rho, N} f|^{p_0} \right\|_{L^{p/p_0}(\omega)} \\
    &= \sup_{\|g\|_{L^q(\omega^{1 - q})} \leq 1} \left| \int_{\mathbb{R}^n} \sum_{Q \in \mathcal{S}}\left( \frac{1}{|Q|} \int_Q |f|^{r} \right)^{\frac{p_0}{r}} \Psi_N(Q)^{-p_0} \chi_Q(x) \, g(x) \, dx \right| \\
    &= \sup_{\|g\|_{L^q(\omega^{1 - q})} \leq 1} \left\{\sum_{Q \in \mathcal{S}} |Q| \left( \frac{1}{|Q|} \int_Q |f|^{r} \right)^{\frac{p_0}{r}} \Psi_N(Q)^{-p_0} \left( \frac{1}{|Q|} \int_Q |g(x)| \, dx \right) \right\} .
\end{align*}
Note that $\frac{1}{p_0 q} = \frac{1}{p_0}-\frac{1}{p}$. Setting $\nu= \omega^{1-(p/r)'}$, and applying H\"older's inequality we get
\begin{align}
\label{Boundedness of sparse operator with r power}
    \left\| \mathcal{A}_{\mathcal{S}}^{r, p_0, \rho, N} f \right\|_{L^p(\omega)}^{p_0} &= \sup_{\|g\|_{L^q(\omega^{1 - q})} \leq 1} \sum_{Q \in \mathcal{S}} \left\{ \nu(E_Q)^{p_0/p} \left(\frac{|Q|}{\nu(Q)} \right)^{p_0/r} \left( \frac{1}{|Q|} \int_Q |f|^{r} \right)^{\frac{p_0}{r}} \right\} \\
    &\nonumber \left\{ \omega(E_Q)^{1/q} \frac{|Q|}{\omega(Q)} \left( \frac{1}{|Q|} \int_Q |g| \right) \right\} \left\{\frac{\omega(Q)^{\frac{1}{p_0}}}{\omega(E_Q)^{\frac{1}{p_0 q}}} \frac{\nu(Q)^{1/r}}{\nu(E_Q)^{1/p}} \frac{\Psi_N(Q)^{-1}}{|Q|^{1/r}} \right\}^{p_0} \\
    &\nonumber \leq \sup_{\|g\|_{L^q(\omega^{1 - q})} \leq 1} C_{\omega, \theta}^{p_0} \left\{\sum_{Q \in \mathcal{S}} \nu(E_Q) \left(\frac{|Q|}{\nu(Q)} \right)^{p/r} \left( \frac{1}{|Q|} \int_Q |f|^{r} \right)^{\frac{p}{r}} \right\}^{p_0/p} \\
    &\nonumber \hspace{4cm} \left\{\sum_{Q \in \mathcal{S}} \omega(E_Q) \left(\frac{|Q|}{\omega(Q)} \right)^q \left( \frac{1}{|Q|} \int_Q |g| \right)^q \right\}^{1/q} ,
\end{align}
where
\begin{align*}
    C_{\omega, \theta} &= \sup_{Q \in \mathscr{D}} \frac{\omega(Q)^{\frac{1}{p_0}}}{\omega(E_Q)^{\frac{1}{p_0}-\frac{1}{p}}} \frac{\nu(Q)^{1/r}}{\nu(E_Q)^{1/p}} \frac{\Psi_N(Q)^{-1}}{|Q|^{1/r}} .
\end{align*}
Let us first compute the following. Application of Proposition \ref{Prop: Boundedness of dyadic maximal functions} and the fact that $E_Q$ are disjoint yields
\begin{align*}
    \sum_{Q \in \mathcal{S}} \omega(E_Q) \left( \frac{|Q|}{\omega(Q)} \right)^{q} \left( \frac{1}{|Q|} \int_Q |g| \right)^q &= \sum_{Q \in \mathcal{S}} \omega(E_Q) \left( \frac{1}{\omega(Q)} \int_Q |g| \omega^{-1} \omega \right)^q \\
    &\lesssim \sum_{Q \in \mathcal{S}} \int_{E_Q} \mathcal{M}_{\omega}^{\mathscr{D}}(g \omega^{-1})(x)^q \omega(x)\,dx \\
    &\lesssim \left\| \mathcal{M}_{\omega}^{\mathscr{D}}(g \omega^{-1}) \right\|_{L^q(\omega)}^q \\
    &\lesssim \left\|g \right\|_{L^q(\omega^{1-q})}^q .
\end{align*}
Using a similar analysis as above, we get
\begin{align*}
    \sum_{Q \in \mathcal{S}} \nu(E_Q) \left(\frac{|Q|}{\nu(Q)} \right)^{p/r} \left( \frac{1}{|Q|} \int_Q |f|^{r} \right)^{\frac{p}{r}} &= \sum_{Q \in \mathcal{S}} \nu(E_Q) \left( \frac{1}{\nu(Q)} \int_Q |f|^{r} \nu^{-1} \nu \right)^{\frac{p}{r}} \\
    &\lesssim \left\| \mathcal{M}_{\nu}^{\mathscr{D}}(|f|^{r} \nu^{-1}) \right\|_{L^{p/r}(\nu)}^{p/r} \\
    &\lesssim \left\|f \right\|_{L^p(\omega)}^p .
\end{align*}
Therefore, putting the above two estimates into \eqref{Boundedness of sparse operator with r power} provides
\begin{align*}
    \left\| \mathcal{A}_{\mathcal{S}}^{r, p_0, \rho, N} f \right\|_{L^p(\omega)}^{p_0} \lesssim \sup_{\|g\|_{L^q(\omega^{1 - q})} \leq 1} C_{\omega, \theta}^{p_0} \left\|f \right\|_{L^p(\omega)}^{p_0} \left\|g \right\|_{L^q(\omega^{1-q})} \lesssim C_{\omega, \theta}^{p_0} \left\|f \right\|_{L^p(\omega)}^{p_0} .
\end{align*}
Now it remains to show that
\begin{align*}
    C_{\omega, \theta} &\lesssim [\omega]_{A_{p/r}^{\rho, \theta}}^{\max\left\{ \frac{1}{p - r}, \frac{1}{p_0} \right\}} .
\end{align*}
Fix a cube $Q \in \mathscr{D}$. Then, using the definition of sparse family and H\"older's inequality we obtain
\begin{align*}
    |Q|^{p/r} &\leq \eta^{-p/r} |E_Q|^{p/r} \leq \eta^{-p/r} \left[\left( \int_{E_Q} \omega \right)^{r/p} \left( \int_{E_Q} \omega^{-\frac{r}{p} (\frac{p}{r})'} \right)^{\frac{1}{(p/r)'}} \right]^{p/r} \leq \eta^{-p/r} \omega(E_Q) \nu(E_Q)^{\frac{p}{r}-1} .
\end{align*}
Hence, the above estimate implies
\begin{align}
\label{First time making of weight in sparse proof}
    \frac{\omega(Q)}{\omega(E_Q)} \left( \frac{\nu(Q)}{\nu(E_Q)} \right)^{\frac{p}{r}-1} \Psi_{\theta}(Q)^{-\frac{p}{r}} &\leq \eta^{-\frac{p}{r}}\, \frac{\omega(Q)}{|Q|} \left( \frac{\nu(Q)}{|Q|} \right)^{\frac{p}{r}-1} \Psi_{\theta}(Q)^{-\frac{p}{r}} \\
    &\nonumber \leq \eta^{-\frac{p}{r}} \left(\frac{1}{\Psi_{\theta}(Q) |Q|} \int_Q \omega \right) \left( \frac{1}{\Psi_{\theta}(Q) |Q|} \int_Q \omega^{-\frac{r}{p-r}} \right)^{\frac{p-r}{r}} \\
    &\nonumber \leq \eta^{-\frac{p}{r}}\, [\omega]_{A_{p/r}^{\rho, \theta}} .
\end{align}
Recall that $\Psi_{\theta}(Q) = \left(1+\frac{r_Q}{\rho(x_Q)} \right)^{\theta} \geq 1$. If we choose $N \geq \theta \frac{p}{r} \max\left\{ \frac{1}{p - r}, \frac{1}{p_0} \right\} $, then one can easily see that
\begin{align*}
    \Psi_N(Q)^{-1} \leq \Psi_{\theta}(Q)^{-\frac{1}{r}} \Psi_{\theta}(Q)^{-\frac{p}{r} \max\{\frac{1}{p_0}-\frac{1}{p}, \frac{1}{p} \frac{r}{p-r} \}} .
\end{align*}
With the help of the above fact and \eqref{First time making of weight in sparse proof}, we deduce that
\begin{align*}
    C_{\omega, \theta} &= \sup_{Q \in \mathscr{D}} \frac{\omega(Q)^{\frac{1}{p_0}}}{\omega(E_Q)^{\frac{1}{p_0}-\frac{1}{p}}} \frac{\nu(Q)^{1/r}}{\nu(E_Q)^{1/p}} \frac{\Psi_N(Q)^{-1}}{|Q|^{1/r}} \\
    &\leq \sup_{Q \in \mathscr{D}} \left[\frac{\omega(Q)}{\Psi_{\theta}(Q) |Q|} \left( \frac{\nu(Q)}{\Psi_{\theta}(Q) |Q|} \right)^{\frac{p}{r}-1} \right]^{1/p} \\
    &\hspace{3cm} \left[\left(\frac{\omega(Q)}{\omega(E_Q)} \right)^{\frac{1}{p_0}-\frac{1}{p}} \left(\frac{\nu(Q)}{\nu(E_Q)} \right)^{1/p} \Psi_{\theta}(Q)^{-\frac{p}{r} \max\{\frac{1}{p_0}-\frac{1}{p}, \frac{1}{p} \frac{r}{p-r} \}} \right] \\
    &\lesssim [\omega]_{A_{p/r}^{\rho, \theta}}^{\frac{1}{p}} \sup_{Q \in \mathscr{D}} \left[\frac{\omega(Q)}{\omega(E_Q)} \left(\frac{\nu(Q)}{\nu(E_Q)} \right)^{\frac{p}{r}-1} \Psi_{\theta}(Q)^{-\frac{p}{r}} \right]^{\max\{\frac{1}{p_0}-\frac{1}{p}, \frac{1}{p} \frac{r}{p-r} \}} \\
    &\lesssim [\omega]_{A_{p/r}^{\rho, \theta}}^{\frac{1}{p}} [\omega]_{A_{p/r}}^{\max\{\frac{1}{p_0}-\frac{1}{p}, \frac{1}{p} \frac{r}{p-r} \}} \\
    &\lesssim [\omega]_{A_{p/r}^{\rho, \theta}}^{\max\left\{ \frac{1}{p - r}, \frac{1}{p_0} \right\}} .
\end{align*}
This completes the proof of the proposition.
\end{proof}

Since in this paper we also study the quantitative weighted estimates for commutators of the Schr\"odinger pseudo-multipliers, in order to do so, for a sparse family $\mathcal{S}$, let us define the sparse operator $\mathcal{T}_{\mathcal{S}, b}^{\rho, N}$ as follows,
\begin{align*}
\mathcal{T}_{\mathcal{S}, b}^{\rho, N} f(x) &=
\sum_{Q \in \mathcal{S}} |b(x)-b_{Q}| \left( \frac{1}{|Q|} \int_Q f \right) \Psi_N(Q)^{-1} \chi_Q(x) ,\quad \quad x\in \mathbb{R}^n ,
\end{align*}
and the adjoint operator $\mathcal{T}_{\mathcal{S}, b}^{\rho, N, \star}$ to $\mathcal{T}_{\mathcal{S}, b}^{\rho, N}$ as 
\begin{align*}
    \mathcal{T}_{\mathcal{S}, b}^{\rho, N, \star}f(x) &= \sum_{Q \in \mathcal{S}} \left( \frac{1}{|Q|} \int_Q |b-b_{Q}| f \right) \Psi_N(Q)^{-1} \chi_Q(x) \quad \quad x\in \mathbb{R}^n .
\end{align*}

\begin{proposition}
\label{Prop: Operator norm for sparse commutator operator}
Let $ \mathcal{S} $ be a sparse family. For $1 < p < \infty$ and $\omega \in A_{p}^{\rho, \theta}$, we define $\nu=\omega^{-\frac{1}{p-1}}$. Then for $b \in BMO_{\theta}(\rho)$ and large positive $N$, we have the following estimates
\begin{align*}
    \max\{\|\mathcal{T}_{\mathcal{S}, b}^{\rho, N}\|_{L^{p}(\omega) \to L^p(\omega)} &= \|\mathcal{T}_{\mathcal{S}, b}^{\rho, N, \star}\|_{L^{p'}(\nu) \to L^{p'}(\nu)}, \|\mathcal{T}_{\mathcal{S}, b}^{\rho, N, \star}\|_{L^{p}(\omega) \to L^{p}(\omega)} \} \\
    &\lesssim [\omega]_{A_{p}^{\rho, \theta}}^{2\max\{1, \frac{1}{p-1}\}} \|b\|_{BMO_{\theta}(\rho)} = [\nu]_{A_{p'}^{\rho, \theta}}^{2\max\{1, \frac{1}{p'-1}\}} \|b\|_{BMO_{\theta}(\rho)} .
\end{align*}
    
\end{proposition}

\begin{proof}
First note that
\begin{align}
\label{Operator of Of Operator and its adjoint}
    & \|\mathcal{T}_{\mathcal{S}, b}^{\rho, N}\|_{L^{p}(\omega) \to L^p(\omega)} = \sup_{\substack{\|f\|_{L^p(\omega)} \leq 1 \\ \|g\|_{L^{p'}(\nu)} \leq 1}} \left|\int_{\mathbb{R}^n} \mathcal{T}_{\mathcal{S}, b}^{\rho, N}f(x) g(x) \, dx \right| \\
    &\nonumber = \sup_{\substack{\|f\|_{L^p(\omega)} \leq 1 \\ \|g\|_{L^{p'}(\nu)} \leq 1}} \left|\int_{\mathbb{R}^n} \sum_{Q \in \mathcal{S}} |b(x)-b_{Q}| \left( \frac{1}{|Q|} \int_Q f(y)\, dy \right) \Psi_N(Q)^{-1} \chi_Q(x) g(x) \, dx \right| \\
    &\nonumber = \sup_{\substack{\|f\|_{L^p(\omega)} \leq 1 \\ \|g\|_{L^{p'}(\nu)} \leq 1}} \left| \int_{\mathbb{R}^n} f(y) \left( \sum_{Q \in \mathcal{S}} \left( \frac{1}{|Q|} \int_Q |b(x)-b_{Q}| g(x) \, dx \right) \Psi_N(Q)^{-1} \chi_Q(y) \right) \,dy \right| \\
    &\nonumber = \|\mathcal{T}_{\mathcal{S}, b}^{\rho, N, \star}\|_{L^{p'}(\nu) \to L^{p'}(\nu)} .
\end{align}
Before proceeding, let us consider the following fact. If $P \subseteq Q$, then applying Lemma \ref{Lemma: Equivalence of two potential} yields
\begin{align*}
    \frac{r_P}{\rho(x_P)} \lesssim \frac{r_P}{\rho(x_Q)} \left(1+\frac{|x_P-x_Q|}{\rho(x_Q)} \right)^{l_0} \lesssim \frac{r_Q}{\rho(x_Q)} \left(1+\frac{r_Q}{\rho(x_Q)} \right)^{l_0} \lesssim \left(1+\frac{r_Q}{\rho(x_Q)} \right)^{1+l_0} .
\end{align*}
From this we get, whenever $P \subseteq Q$, it implies
\begin{align}
\label{Dominated of psi theta by psi theta l}
    \Psi_{\theta}(P) \lesssim \Psi_{\theta (1+l_0)}(Q) .
\end{align}
Note that because of remark \ref{Remark: system of dyadic cube}, we can assume $\mathcal{S} \subset \mathscr{D}$. By Lemma 5.1 of \cite{Lerner_Ombrosi_Rivera_Pointwise_Commutators_2017}, there exists sparse family $\widetilde{\mathcal{S}}$ containing $\mathcal{S}$ and for every $Q \in \widetilde{\mathcal{S}}$,
\begin{align}
\label{Sparse domination of BMO functions}
    |b(x)-b_Q| &\lesssim \sum_{R \in \widetilde{\mathcal{S}}, R \subseteq Q} \left(\frac{1}{|R|} \int_R |b(y)-b_R| \, dy \right) \chi_{R}(x) ,
\end{align}
for a.e. $x \in Q$.

Therefore, using \eqref{Sparse domination of BMO functions} and \eqref{Dominated of psi theta by psi theta l}, we obtain
\begin{align*}
    & \Psi_N(Q)^{-1} \int_Q |b(x)-b_Q| |f(x)|\, dx \\
    &\lesssim \Psi_N(Q)^{-1} \sum_{R \in \widetilde{\mathcal{S}}, R \subseteq Q} \left(\frac{1}{|R|} \int_R |b(y)-b_R| \, dy \right) \int_R |f(x)|\, dx \\
    &\lesssim \Psi_N(Q)^{-1} \Psi_{\theta (1+l_0)}(Q) \sum_{R \in \widetilde{\mathcal{S}}, R \subseteq Q} \left(\frac{1}{\Psi_{\theta}(R) |R|} \int_R |b(y)-b_R| \, dy \right) |R| \left(\frac{1}{|R|} \int_R |f(x)|\, dx \right) \\
    &\lesssim \|b\|_{BMO_{\theta}(\rho)} \Psi_{-N+(\theta+M)(1+l_0)}(Q) |Q| \sum_{R \in \widetilde{\mathcal{S}}, R \subseteq Q} \left(\frac{1}{|R|} \int_R |f| \right) \left(1+\frac{r_R}{\rho(x_R)} \right)^{-M} \chi_R(z) \\
    &\lesssim \|b\|_{BMO_{\theta}(\rho)} \left(1+\frac{r_Q}{\rho(x_Q)} \right)^{-N+(\theta+M)(1+l_0)} |Q| \inf_{ Q \ni z} \mathcal{A}_{\widetilde{\mathcal{S}}}^{1,1,\rho, M}f(z) \\
    &\lesssim \|b\|_{BMO_{\theta}(\rho)} \left(1+\frac{r_Q}{\rho(x_Q)} \right)^{-N+(\theta+M)(1+l_0)} \int_Q \mathcal{A}_{\widetilde{\mathcal{S}}}^{1,1,\rho, M}f(x)\, dx ,
\end{align*}
where $M>0$, which we choose later.

Then, using the above estimate we get
\begin{align*}
    & |\mathcal{T}_{\mathcal{S}, b}^{\rho, N, \star}f(x)| \leq \sum_{Q \in \mathcal{S}} \left( \frac{1}{|Q|} \int_Q |b-b_{Q}| |f| \right) \Psi_N(Q)^{-1} \chi_Q(x) \\
    &\lesssim \|b\|_{BMO_{\theta}(\rho)} \sum_{Q \in \mathcal{S}} \left( \frac{1}{|Q|} \int_Q \mathcal{A}_{\widetilde{\mathcal{S}}}^{1,1,\rho, M}f(x)\, dx \right) \Psi_M(Q)^{-1} \chi_Q(x) \left(1+\frac{r_Q}{\rho(x_Q)} \right)^{-N+M+(\theta+M)(1+l_0)} \\
    &\lesssim \|b\|_{BMO_{\theta}(\rho)} \mathcal{A}_{\widetilde{\mathcal{S}}}^{1,1,\rho, M}(\mathcal{A}_{\widetilde{\mathcal{S}}}^{1,1,\rho, M}f)(x) ,
\end{align*}
provided we choose $N>M+(\theta+M)(1+l_0)$.

Hence applying Proposition \ref{Prop: Boundedness of sparse operator} with $ M \geq \theta p \max\left\{ \frac{1}{p - 1}, 1 \right\} $, we obtain
\begin{align*}
    \|\mathcal{T}_{\mathcal{S}, b}^{\rho, N, \star}f\|_{L^{p}(\omega)} &\lesssim \|b\|_{BMO_{\theta}(\rho)} \|\mathcal{A}_{\widetilde{\mathcal{S}}}^{1,1,\rho, M}(\mathcal{A}_{\widetilde{\mathcal{S}}}^{1,1,\rho, M}f)\|_{L^{p}(\omega)}  \\
    &\lesssim [\omega]_{A_{p}^{\rho, \theta}}^{2\max\{1, \frac{1}{p-1}\}} \|b\|_{BMO_{\theta}(\rho)} \|f\|_{L^{p}(\omega)}.
\end{align*}
Similarly, using the fact $[\nu]_{A_{p'}^{\rho, \theta}} = [\omega]_{A_p^{\rho, \theta}}^{\frac{1}{p-1}}$ and Proposition \ref{Prop: Boundedness of sparse operator} with $ M \geq \theta p \max\left\{ \frac{1}{p - 1}, 1 \right\} $ yields
\begin{align*}
    \|\mathcal{T}_{\mathcal{S}, b}^{\rho, N, \star}f\|_{L^{p'}(\nu)} &\lesssim \|b\|_{BMO_{\theta}(\rho)} \|\mathcal{A}_{\widetilde{\mathcal{S}}}^{1,1,\rho, M}(\mathcal{A}_{\widetilde{\mathcal{S}}}^{1,1,\rho, M}f)\|_{L^{p'}(\nu)}  \\
    &\lesssim [\nu]_{A_{p'}^{\rho, \theta}}^{2\max\{1, \frac{1}{p'-1}\}} \|b\|_{BMO_{\theta}(\rho)} \|f\|_{L^{p'}(\nu)} \\
    &\lesssim [\omega]_{A_{p}^{\rho, \theta}}^{2\max\{1, \frac{1}{p-1}\}} \|b\|_{BMO_{\theta}(\rho)} \|f\|_{L^{p'}(\nu)} .
\end{align*}
Therefore, combining the above two estimates and \eqref{Operator of Of Operator and its adjoint} gives the required result.
    
\end{proof}

Next, we prove a quantitative version of the reverse H\"older's inequality for the weight class $A_p^{\rho, \theta}$ associated to the Schr\"odinger operator $\mathcal{L}$. The following result was proved in \cite[Lemma 3.28]{Perez_Singular_Integrals_Weights_2015} for the classical Muckenhoupt weight class $A_p$. Our result can be seen as a generalization of \cite[Lemma 3.28]{Perez_Singular_Integrals_Weights_2015} for the weight class $A_p^{\rho, \theta}$.

\begin{proposition}
\label{Proposition: Reverse Holders inequality}
For $1<p<\infty$, let $\omega \in A_p^{\rho, \theta}$ and $r_{\omega}=1+\frac{1}{2^{2p(1+\theta)+n+1} [\omega]_{A_p^{\rho, \theta}}}$. Then for any cube $Q$, there exists a positive number $N_0$ and $C>0$ independent of $\omega$ such that
\begin{align*}
    \left(\frac{1}{|Q|} \int_Q \omega^{r_{\omega}} \, dx \right)^{\frac{1}{r_{\omega}}} \leq C \left(\frac{1}{|Q|} \int_{Q} \omega \right) \Psi_{N_0}(Q) .
\end{align*}
    
\end{proposition}

In order to prove the above reverse H\"older's inequality, we need the following crucial lemma, which is proved in \cite[Lemma 2.3]{Dziubanski_Zienkiewicz_Hardy_Schrodinger_1999} (see also \cite[Lemma 4.2]{Bui_Bui_Duong_square_function_new_weight_2022}).

\begin{lemma}
\label{Lemma: Decomposition of space into ball}
Let $\rho$ be the function defined in \eqref{Expression: Auxiliary function}. Then there exists a sequence $\{x_m\}_{m \in \mathbb{R}^n}$ such that
\begin{enumerate}
    \item $\bigcup_{m \in \mathbb{N}} B(x_m, \rho(x_m)) = \mathbb{R}^n$. \\
    \item There exists positive constants $C$ and $N$ such that, for every $\kappa \geq 1$ we have $$\sum_{m \in \mathbb{N}} \chi_{B(x_m, \kappa \rho(x_m))} \leq C \kappa^{N} .$$
\end{enumerate}
    
\end{lemma}

\begin{proof}[Proof of Proposition \ref{Proposition: Reverse Holders inequality}]
Let us fix a cube $Q=Q(x_Q,r)$ of center $x_Q$ and side length $r$. We divide our analysis into two cases. 

\subsection*{Case-I: \texorpdfstring{$r \leq \rho(x_Q)$}{}}
Let us set $\omega_Q = \frac{1}{|Q|} \int_Q \omega = \frac{\omega(Q)}{|Q|}$. Then, using the layer cake decomposition, for $\delta>0$ we write
\begin{align*}
    &\frac{1}{|Q|} \int_Q \omega(x)^{\delta} \omega(x)\, dx \\
    &= \frac{\delta}{|Q|} \int_0^{\infty} t^{\delta} \omega(\{x \in Q: \omega(x) > t \}) \frac{dt}{t} \\
    &= \frac{\delta}{|Q|} \int_0^{\omega_Q} t^{\delta} \omega(\{x \in Q: \omega(x) > t \}) \frac{dt}{t} + \frac{\delta}{|Q|} \int_{\omega_Q}^{\infty} t^{\delta} \omega(\{x \in Q: \omega(x) > t \}) \frac{dt}{t} \\
    &= F_1 + F_2 .
\end{align*}
Estimate of $F_1$: In this case, we have
\begin{align*}
    F_1 &= \frac{\delta}{|Q|} \int_0^{\omega_Q} t^{\delta} \omega(\{x \in Q: \omega(x) > t \}) \frac{dt}{t} \leq \delta \omega_Q \int_0^{\omega_Q} t^{\delta-1} \, dt \leq \omega_Q^{\delta+1} .
\end{align*}
Estimate of $F_2$: For any cube $Q$, let
\begin{align*}
    E_Q &= \left\{x \in Q: \omega(x) \leq \frac{1}{2^{p(1+\theta)-1} [\omega]_{A_p^{\rho, \theta}}} \omega_Q \right\} .
\end{align*}
Then we make the following claim:
\begin{align}
\label{First Claim : for reverse Holder}
    |E_Q| &\leq \frac{1}{2} |Q| .
\end{align}
Let us start by proving the claim. Applying H\"older's inequality for any $f \geq 0$ we get
\begin{align}
\label{First use of Holder in reverse Holder}
    &\left(\frac{1}{|Q|} \int_Q f(y) \, dy \right)^p \omega(Q) \\
    &\nonumber \leq \left(\int_Q f(y)^p \omega(y) \, dy \right) \left(\frac{1}{\Psi_{\theta}(Q) |Q|} \int_Q \omega(y)^{-p'/p}\, dy \right)^{p/p'} \left(\frac{1}{\Psi_{\theta}(Q) |Q|} \int_Q \omega(y) \, dy \right) \Psi_{\theta}(Q)^p \\
    &\nonumber \leq [\omega]_{A_p^{\rho, \theta}} \Psi_{\theta}(Q)^p \left(\int_Q f(y)^p \omega(y) \, dy \right) \\
    &\nonumber \leq 2^{\theta p} [\omega]_{A_p^{\rho, \theta}} \left(\int_Q f(y)^p \omega(y) \, dy \right) ,
\end{align}
where in the last line we used the fact $\Psi_{\theta}(Q) \leq 2^{\theta}$, since $r \leq \rho(x_Q)$.

For $E \subset Q$, if we take $f=\chi_{E}$ in \eqref{First use of Holder in reverse Holder}, then
\begin{align*}
    \left(\frac{|E|}{|Q|} \right)^p \leq 2^{\theta p} [\omega]_{A_p^{\rho, \theta}} \frac{\omega(E)}{\omega(Q)} .
\end{align*}
In particular, if we take $E_Q$, then the above yields
\begin{align*}
    \left(\frac{|E_Q|}{|Q|} \right)^p &\leq 2^{\theta p} [\omega]_{A_p^{\rho, \theta}} \frac{\omega(E_Q)}{\omega(Q)} \leq 2^{\theta p} [\omega]_{A_p^{\rho, \theta}} \frac{1}{\omega(Q)} \frac{\omega_Q |E_Q|}{2^{p(1+\theta)-1} [\omega]_{A_p}^{\rho, \theta}} \leq \frac{1}{2^{p-1}} \frac{|E_Q|}{|Q|} .
\end{align*}
From this, the required claim \eqref{First Claim : for reverse Holder} follows.

Now we make another claim. For every $t> \omega_Q$,
\begin{align*}
    \omega\left(\{x \in Q: \omega(x)>t\} \right) \leq 2^{n+1} t \left|\left\{x \in Q: \omega(x) > \frac{t}{2^{p(1+\theta)-1}[\omega]_{A_p^{\rho, \theta}}} \omega_Q \right\}\right| .
\end{align*}
Since $t>\omega_Q$, hence we can consider the Calder\'on-Zygmund decomposition of $\omega$ at the level $t$, and find a family a disjoint cubes $\{Q_j\}$ contained in $Q$ such that
\begin{enumerate}
    \item $t<\omega_{Q_i} \leq 2^n t$ \quad for each $i$.
    \item $\bigcup_i Q_i = \{x \in Q : \mathcal{M}_Q^{\mathscr{D}}\omega(x)>t\}$,
\end{enumerate}
where $\mathcal{M}_Q^{\mathscr{D}}$ is the dyadic maximal function restricted to $Q$.

Note that $\omega(x) \leq \mathcal{M}_Q^{\mathscr{D}}\omega(x)$ a.e., therefore, except for a null set, we have the following.
\begin{align*}
    \{x \in Q : \omega(x)>t\} \subset \{x \in Q : \mathcal{M}_Q^{\mathscr{D}}\omega(x)>t\} = \bigcup_i Q_i .
\end{align*}
Using the above facts and claim \eqref{First Claim : for reverse Holder} we obtain
\begin{align*}
    \omega(\{x \in Q : \omega(x)>t\}) &\leq \sum_i \omega(Q_i) \leq 2^n t \sum_i |Q_i| \\
    &\leq 2^{n+1} t \sum_i \left|\left\{x \in Q_i : \omega(x) > \frac{1}{2^{p(1+\theta)-1} [\omega]_{A_p^{\rho, \theta}}} \omega_{Q_i} \right\}\right| \\
    &\leq 2^{n+1} t \left|\left\{x \in Q : \omega(x) > \frac{1}{2^{p(1+\theta)-1} [\omega]_{A_p^{\rho, \theta}}} t \right\}\right| .
\end{align*}
Therefore the above estimate and layer cake decomposition give
\begin{align*}
    F_2 &= \frac{\delta}{|Q|} \int_{\omega_Q}^{\infty} t^{\delta} \omega(\{x \in Q: \omega(x) > t \}) \frac{dt}{t} \\
    &\leq \frac{2^{n+1} \delta}{|Q|} \int_{\omega_Q}^{\infty} t^{\delta+1} \left|\left\{x \in Q : \omega(x) > \frac{1}{2^{p(1+\theta)-1} [\omega]_{A_p^{\rho, \theta}}} t \right\}\right| \frac{dt}{t} \\
    &\leq (2^{p(1+\theta)-1} [\omega]_{A_p^{\rho, \theta}})^{\delta+1} \frac{2^{n+1} \delta}{|Q|} \int_{0}^{\infty} t^{\delta+1} |\{x \in Q : \omega(x) > t \}| \frac{dt}{t} \\
    &= (2^{p(1+\theta)-1} [\omega]_{A_p^{\rho, \theta}})^{\delta+1} \frac{2^{n+1} \delta}{\delta+1} \frac{1}{|Q|} \int_{Q} \omega^{\delta+1} \, dx .
\end{align*}
Now, if we choose $\delta= \frac{1}{2^{2p(1+\theta)+n+1} [\omega]_{A_p^{\rho, \theta}}}$, then we can see
\begin{align*}
    F_2 &\leq \frac{1}{2} \frac{1}{|Q|} \int_{Q} \omega(x)^{\delta+1} \, dx .
\end{align*}
Therefore combining the estimate of $F_1$ and $F_2$ we obtain
\begin{align*}
    \frac{1}{|Q|} \int_{Q} \omega(x)^{\delta+1} \, dx &\leq \left(\frac{1}{|Q|} \int_Q \omega(x) \, dx \right)^{\delta+1} + \frac{1}{2} \frac{1}{|Q|} \int_{Q} \omega(x)^{\delta+1} \, dx ,
\end{align*}
and this implies the lemma for the case $r \leq \rho(x_Q)$ with $r_{\omega}=\delta+1$.

\subsection*{Case-II: \texorpdfstring{$r > \rho(x_Q)$}{}}
In this case using Lemma \ref{Lemma: Decomposition of space into ball} we can find a sequence $\{x_m\}_{m \in \mathbb{N}}$ and balls $B(x_m, \rho(x_m))$ such that $\bigcup_{m \in \mathbb{N}} B(x_m, \rho(x_m)) = \mathbb{R}^n$. Let us set $\mathcal{F} = \{m: Q(x_Q,r) \cap B(x_m, \rho(x_m)) \neq \emptyset\}$. For each $m \in \mathcal{F}$, it follows from Lemma \ref{Lemma: Equivalence of two potential} that there exist constants $l_0, C>0$ such that
\begin{align}
\label{exp: relation between xm and xnot}
    C^{-1} \rho(x_Q) \left(1+\frac{r}{\rho(x_Q)} \right)^{-l_0} \leq \rho(x_m) &\leq C \left(1+\frac{r}{\rho(x_Q)} \right)^{l_0/(l_0+1)} r .
\end{align}
We denote $\mathcal{C}_r = C \left(1+\frac{r}{\rho(x_Q)} \right)^{l_0/(l_0+1)}$. Note that we can choose $C>1$ so that $\mathcal{C}_r>1$. Then the above estimate immediately gives
\begin{align*}
    \bigcup_{m \in \mathcal{F}} B(x_m, \rho(x_m)) \subset \mathcal{C}_r Q(x_Q,r) .
\end{align*}
Choose $\delta>0$ to be the same as in Case-I. Then using Case-I, \eqref{exp: relation between xm and xnot}, along with Lemma \ref{Lemma: Decomposition of space into ball} we obtain
\begin{align}
\label{Exp: Estimate of reverse holder for r bigger}
    \left( \int_{Q} \omega(x)^{\delta+1} \, dx \right)^{\frac{1}{\delta+1}} &\leq C \sum_{m \in \mathcal{F}} \left( \int_{B(x_m, \rho(x_m))} \omega(x)^{\delta+1} \, dx \right)^{\frac{1}{\delta+1}} \\
    &\nonumber \leq 2 C \sum_{m \in \mathcal{F}} \left( \int_{B(x_m, \rho(x_m))} \omega(x) \, dx \right) |B(x_m, \rho(x_m))|^{-\frac{\delta}{\delta+1}} \\
    &\nonumber \leq C \rho(x_Q)^{-\frac{n \delta}{\delta+1}} \left(1+\frac{r}{\rho(x_Q)} \right)^{\frac{l_0 n \delta}{\delta+1}} \left( \int_{\mathcal{C}_r Q(x_Q,r)} \omega(x) \, dx \right).
\end{align}
Since $\omega \in A_p^{\rho, \theta}$ and $\mathcal{C}_r>1$, the last integral on the right hand side of the above inequality is dominated by
\begin{align}
\label{Estimate of the integral in RHS}
    \int_{\mathcal{C}_r Q(x_Q,r)} \omega(x) \, dx & \leq C \mathcal{C}_r^{n p} |Q|^p \left( \int_{\mathcal{C}_r Q(x_Q,r)} \omega(x)^{-\frac{1}{p-1}} \, dx \right)^{-p+1} \left(1+\frac{\mathcal{C}_r r}{\rho(x_Q)} \right)^{\theta p} \\
    &\nonumber \leq C |Q| \left( \int_{Q} \omega(x)^{-\frac{1}{p-1}} \, dx \right)^{-p+1} \left(1+\frac{r}{\rho(x_Q)} \right)^{\theta p + \frac{\theta p l_0}{l_0+1}+ \frac{n p l_0}{l_0+1}} \\
    &\nonumber \leq C \left( \int_{Q} \omega(x) \, dx \right) \left(1+\frac{r}{\rho(x_Q)} \right)^{\theta p + \frac{\theta p l_0}{l_0+1}+ \frac{n p l_0}{l_0+1}} ,
\end{align}
where in the last line we used the following fact,
\begin{align*}
    1 &= \frac{1}{|Q|} \int_Q \omega^{1/p} \omega^{-1/p} \leq \left\{\left(\frac{1}{|Q|} \int_{Q} \omega \right) \left(\frac{1}{|Q|} \int_{Q} \omega^{-\frac{1}{p-1}} \right)^{p-1} \right\}^{1/p} .
\end{align*}
Putting the estimate \eqref{Estimate of the integral in RHS} into \eqref{Exp: Estimate of reverse holder for r bigger} we get
\begin{align*}
    & \left(\frac{1}{|Q|} \int_{Q} \omega(x)^{\delta+1} \, dx \right)^{\frac{1}{\delta+1}} \\
    &\leq C \frac{|Q|}{|Q|^{\frac{1}{\delta+1}}} \rho(x_Q)^{-\frac{n \delta}{\delta+1}} \left(\frac{1}{|Q|} \int_{Q} \omega(x) \, dx \right) \left(1+\frac{r}{\rho(x_Q)} \right)^{\theta p + \frac{\theta p l_0}{l_0+1}+ \frac{n p l_0}{l_0+1}+ \frac{l_0 n \delta}{\delta+1}} \\
    &\leq C \left(\frac{r}{\rho(x_Q)} \right)^{\frac{n \delta}{\delta+1}} \left(\frac{1}{|Q|} \int_{Q} \omega(x) \, dx \right) \left(1+\frac{r}{\rho(x_Q)} \right)^{\theta p + \frac{\theta p l_0}{l_0+1}+ \frac{n p l_0}{l_0+1}+ \frac{l_0 n \delta}{\delta+1}} \\
    &\leq C \left(\frac{1}{|Q|} \int_{Q} \omega(x) \, dx \right) \Psi_{N_0}(Q) ,
\end{align*}
where $N_0=\theta p + (\theta+n)\frac{ p l_0}{l_0+1}+ (l_0+1)\frac{ n \delta}{\delta+1}$.

This completes the proof of the proposition.
\end{proof}

As an application of Proposition \ref{Proposition: Reverse Holders inequality}, we have the following lemma, which will be useful in the weighted boundedness of commutators of Schr\"odinger pseudo-multipliers.
\begin{lemma}
\label{Lemma: llogl estimate in terms of maximal function}
Let $s>1$ and $\omega \in A_p^{\rho, \theta}$. Then there exists a positive number $N_0$ such that
\begin{align*}
    \|f \omega\|_{\mathcal{L} \log\mathcal{L}, Q} \lesssim [\omega]_{A_p^{\rho, \theta}} \inf_{Q \ni x} \mathcal{M}_{\omega}(|f|^s)(x)^{1/s} \omega_Q \Psi_{N_0}(Q)^{1/s'} ,
\end{align*}
for every cube $Q$.
    
\end{lemma}

The above lemma can be proved by using the same line of arguments as in \cite[Lemma 3.1]{Cao_Yabuta_Multilinear_Littlewood_Paley_2019} with obvious modification. Therefore, we omit the details to make our paper less technical.

\section{Kernel estimates}
\label{Section: Kernel estimate}
In this section, we establish a pointwise estimate of the kernel and its gradient of the pseudo-multiplier associated with the Schr\"odinger operator $\mathcal{L}$. The proof of our kernel estimate is based on the method developed in \cite{Duong_Ouhabaz_Sikora_Sharp_Multiplier_2002}, using the heat kernel estimate for the Schr\"odinger operator $\mathcal{L}$.

Let $\{e^{-t\mathcal{L}}\}_{t>0}$ be the heat semigroup associated with $\mathcal{L}$. Then we can write $\{e^{-t\mathcal{L}}\}_{t>0}$ in terms of the kernel representation as
\begin{align*}
    e^{-t\mathcal{L}}f(x) = \int_{\mathbb{R}^n} K_{e^{-t\mathcal{L}}}(x,y) f(y) \ dy , \quad f \in C_c^{\infty}(\mathbb{R}^n) .
\end{align*}

The pointwise estimate of the kernel $K_{e^{-t\mathcal{L}}}(x,y)$ is well known and given by the following lemma. We denote $p_t(x,y) = K_{e^{-t\mathcal{L}}}(x,y)$.
\begin{lemma}\cite[Lemma 2.2, Lemma 3.8]{Duong_Yan_Zhang_Schrodinger_operator_2014}
\label{Lemma: Heat kernel estimate for L}
    For every $N>0$, there exists the constants $C_N$ and $c$ such that for $x,y \in \mathbb{R}^n$, $t>0$,
    \begin{align*}
        0\leq p_t(x,y) \leq C_N\, t^{-n/2} e^{-\frac{|x-y|^2}{ct}} \left(1+\frac{\sqrt{t}}{\rho(x)}+\frac{\sqrt{t}}{\rho(y)} \right)^{-N} ,
    \end{align*}
    and
    \begin{align*}
        |\nabla p_t(x,y)| + |\sqrt{t} \nabla \partial_t p_t(x,y)| \leq C_N\, t^{-(n+1)/2} e^{-\frac{|x-y|^2}{ct}} \left(1+\frac{\sqrt{t}}{\rho(x)}+\frac{\sqrt{t}}{\rho(y)} \right)^{-N} ,
    \end{align*}
    where $\nabla$ is the gradient in $\mathbb{R}^n$.
\end{lemma}

Using the decomposition as in \eqref{Equation: dyadic decomposition of pseudo multiplier}, we write
\begin{align}\label{eq:operator-decom}
    \sigma(x, \sqrt{\mathcal{L}}) &= \sum_{j=0}^{\infty} \sigma_j(x, \sqrt{\mathcal{L}}) ,
\end{align}
where $\sigma_j(x, \lambda) = \sigma(x, \lambda) \psi_j(\lambda)$.

Let $K_{j}(x,y)$ denote the kernel corresponding to the pseudo-multiplier $\sigma_j(x, \sqrt{\mathcal{L}})$.
\begin{proposition}
\label{Prop: Kernel estimate}
    Let $\Gamma \in \mathbb{N}^n$. Then for all $N\geq 0$ and $\beta \geq 0$, there exists a constant $C>0$ such that
    \begin{align*}
        \left(1+ \frac{|x-y|}{\rho(x)} + \frac{|x-y|}{\rho(y)} \right)^{N} |\nabla^{\Gamma} K_j(x,y)| &\leq C |x-y|^{-\beta} 2^{j(n+|\Gamma|-\beta)} ,
    \end{align*}
    with $|\Gamma|=0$ or $1$.
\end{proposition}

In order to prove the above proposition, we need the following lemma.

\begin{lemma}
\label{Lemma: L infinity gradient estimate of heat kernel}
    For any $\tau \in \mathbb{R}$, $R>0$ and $s \geq 0$, there exists $C>0$ such that
    \begin{align*}
         |\nabla^{\Gamma} p_{(1+i \tau)R^{-2}}(x,y)| (1+R |x-y|)^s &\leq C (1+|\tau|)^s R^{(n+|\Gamma|)} R^N \rho(x)^{N/2} \rho(y)^{N/2},
    \end{align*}
    for all $x, y \in \mathbb{R}^n$.
\end{lemma}
\begin{proof}
Let $y \in \mathbb{R}^n$. Choose a function $f$ such that  $\|f\|_{L^1(\mathbb{R}^n)} = 1$ and supp$f \subseteq \mathbb{R}^n \setminus B(y,r)$. We define the holomorphic function $F_y : \{ z\in \mathbb{C} : Re(z) >0 \} \to \mathbb{C} $ by
\begin{align*}
    F_y(z) = e^{-z R^2} R^{-(n+|\Gamma|)} \left(\int_{\mathbb{R}^n} \nabla^{\Gamma} p_z(x,y) f(x) \ dx \right).
\end{align*}
Put $z= |z|e^{i \theta}$. One can show $\esssup_{x \in \mathbb{R}^n} |\nabla^{\Gamma} p_z(x,y)| \leq \esssup_{x \in \mathbb{R}^n} |\nabla^{\Gamma} p_{|z|\cos \theta}(x,y)| $.

Indeed, this can be seen as follows. A linear operator $S$ is bounded from $L^1(\mathbb{R}^n)$ to $L^{\infty}(\mathbb{R}^n)$ if and only if it is an integral operator with kernel $K(x,y)$ such that $\esssup_{x \in \mathbb{R}^n} |K(x,y)|$ is finite, in which case
\begin{align*}
    \esssup_{x \in \mathbb{R}^n} |K(x,y)| = \|S\|_{1 \to \infty}.
\end{align*}
Therefore, we can see
\begin{align*}
    & \esssup_{x \in \mathbb{R}^n} |p_z(x,y)| = \|\exp(-z \mathcal{L})\|_{1 \to \infty} \\
    &\leq \|\exp(-|z|\cos \theta \mathcal{L}/2)\|_{1 \to 2} \|\exp(-i |z| \sin \theta \mathcal{L})\|_{2 \to 2} \|\exp(-|z|\cos \theta \mathcal{L}/2)\|_{2 \to \infty} \\
    &= \|\exp(-|z|\cos \theta \mathcal{L}/2)\|_{1 \to 2}^2 \|\exp(-i |z| \sin \theta \mathcal{L})\|_{2 \to 2} \\
    &= \|\exp(-|z|\cos \theta \mathcal{L})\|_{1 \to \infty} = \esssup_{x \in \mathbb{R}^n} |p_{|z|\cos \theta}(x,y)| ,
\end{align*}
where we have used the fact that $\|T^* T\|_{1\to \infty} = \|T\|_{1\to 2}^2$ and $\mathcal{L}$ is self-adjoint.

Then we get
\begin{align}
\label{Expression of Fyz}
    |F_y(z)| &\leq e^{-R^2 |z|\cos \theta} R^{-(n+|\Gamma|)} \esssup_{x \in \mathbb{R}^n} |\nabla^{\Gamma} p_z(x,y)| \|f\|_{L^1(\mathbb{R}^n)} \\
    &\nonumber \leq e^{-R^2 |z|\cos \theta} R^{-(n+|\Gamma|)} \esssup_{x \in \mathbb{R}^n} |\nabla^{\Gamma} p_{|z|\cos \theta}(x,y)| .
\end{align}
From \Cref{Lemma: Heat kernel estimate for L}, for all $t>0$ we have
\begin{align*}
    |\nabla^{\Gamma} p_t(x,y)| &\leq C t^{-(n+|\Gamma|)/2} \left(1+\frac{\sqrt{t}}{\rho(x)} \right)^{-N} \left(\frac{\sqrt{t}}{\rho(y)} \right)^{-N} \\
    &\leq C t^{-(n+|\Gamma|)/2} t^{-N/2} \rho(y)^{N} .
\end{align*}
Therefore, using the above estimate in \eqref{Expression of Fyz} we get
\begin{align*}
    |F_y(z)| &\leq C e^{-R^2 |z|\cos \theta} R^{-(n+|\Gamma|)} (|z|\cos \theta)^{-(n+|\Gamma|)/2} (|z|\cos \theta)^{-N/2} \rho(y)^{N} \\
    &\leq C R^{-(n+|\Gamma|)} (|z|\cos \theta)^{-\frac{n+|\Gamma|+N}{2}} \rho(y)^{N}.
\end{align*}
Also note that on the support of $f$, i.e. for $|x-y| >r$ ,
\begin{align*}
    |\nabla^{\Gamma} p_t(x,y)| &\leq C t^{-(n+|\Gamma|)/2} e^{-\frac{|x-y|^2}{ct}} \left(1+\frac{\sqrt{t}}{\rho(x)}+\frac{\sqrt{t}}{\rho(y)} \right)^{-N} \\
    &\leq C t^{-(n+|\Gamma|)/2} e^{-\frac{r^2}{ct}} t^{-N/2} \rho(y)^{N}.
\end{align*}
Again with the help of the above gradient estimate of the heat kernel, from \eqref{Expression of Fyz} and taking $\theta = 0$ we get
\begin{align*}
    |F_y(z)| &\leq C R^{-(n+|\Gamma|)} |z|^{-\frac{n+|\Gamma|+N}{2}} e^{-\frac{r^2}{c|z|}} \rho(y)^{N}.
\end{align*}
Now if we take $z=|z|e^{i \theta} = (1+i \tau) R^{-2}$ then $|z|= R^{-2}(1+|\tau|^2)^{1/2}, |z|\cos \theta = R^{-2}, \cos \theta = (1+|\tau|^2)^{-1/2}$. Also putting $a_1= CR^{-(n+|\Gamma|)} \rho(y)^{N}, a_2 = r^2/c, \beta_1 = (n+|\Gamma|+N)/2, \beta_2 =1$ in Phragmen-Lindel\"of theorem (see \cite{Duong_Ouhabaz_Sikora_Sharp_Multiplier_2002}, Lemma 4.2) we get
\begin{align*}
    |F_y((1+i \tau)R^{-2})| &\leq C R^N \exp\left(-c^{\prime} (rR/(1+|\tau|))^2 \right) \rho(y)^{N}.
\end{align*}
Which is the same as
\begin{align*}
    & \left| e^{-(1+i \tau)R^{-2} R^2} R^{-(n+|\Gamma|)} \left(\int_{\mathbb{R}^n \setminus B(y,r)} \nabla^{\Gamma} p_{(1+i \tau)R^{-2}}(x,y) f(x) \ dx \right) \right| \\
    &\leq C R^N \exp\left(-c^{\prime} (rR/(1+|\tau|))^2 \right) \rho(y)^{N}.
\end{align*}
Now taking supremum over all $f \in L^1(\mathbb{R}^n)$ with $\|f\|_{L^1(\mathbb{R}^n)}=1$ we get
\begin{align*}
    \sup_{x \in \mathbb{R}^n \setminus B(y,r)} |\nabla^{\Gamma} p_{(1+i \tau)R^{-2}}(x,y)| &\leq C R^{(n+|\Gamma|)} R^N \exp\left(-c^{\prime} (rR/(1+|\tau|))^2 \right) \rho(y)^{N}.
\end{align*}
Therefore
\begin{align*}
    & \sup_{x \in \mathbb{R}^n} \left\{ |\nabla^{\Gamma} p_{(1+i \tau)R^{-2}}(x,y)| |x-y|^s \right\} \\
    &\leq \sum_{k=0}^{\infty} \sup_{k(1+|\tau|)R^{-1} \leq |x-y| \leq (k+1)(1+|\tau|)R^{-1}} \left\{ |\nabla^{\Gamma} p_{(1+i \tau)R^{-2}}(x,y)| |x-y|^s \right\} \\
    &\leq (1+|\tau|)^s R^{-s} \sum_{k=0}^{\infty} (k+1)^s \sup_{|x-y| \geq k(1+|\tau|)R^{-1}} \left\{ |\nabla^{\Gamma} p_{(1+i \tau)R^{-2}}(x,y)| \right\} \\
    &\leq C (1+|\tau|)^s R^{-s} R^{(n+|\Gamma|)} R^N \rho(y)^{N} \sum_{k=0}^{\infty} (k+1)^s \exp(-c^{\prime}k^2) \\
    &\leq C (1+|\tau|)^s R^{-s} R^{(n+|\Gamma|)} R^N \rho(y)^{N}.
\end{align*}
Now using the above estimate and interpolation with $s=0$ we get
\begin{align}
\label{Estimate: last Heat kernel gradient estimate with weight with y}
    |\nabla^{\Gamma} p_{(1+i \tau)R^{-2}}(x,y)| (1+R |x-y|)^s &\leq C (1+|\tau|)^s R^{(n+|\Gamma|)} R^N \rho(y)^{N}.
\end{align}
Similarly, following the same calculation by fixing the $x$-variable instead of $y$-variable we can also show that
\begin{align}
\label{Estimate: last Heat kernel gradient estimate with weight with x}
    |\nabla^{\Gamma} p_{(1+i \tau)R^{-2}}(x,y)| (1+R |x-y|)^s &\leq C (1+|\tau|)^s R^{(n+|\Gamma|)} R^N \rho(x)^{N}.
\end{align}
Now combining the estimates \eqref{Estimate: last Heat kernel gradient estimate with weight with y} and \eqref{Estimate: last Heat kernel gradient estimate with weight with x} we obtain
\begin{align*}
    |\nabla^{\Gamma} p_{(1+i \tau)R^{-2}}(x,y)| (1+R |x-y|)^s &\leq C (1+|\tau|)^s R^{(n+|\Gamma|)} R^N \rho(x)^{N/2} \rho(y)^{N/2} ,
\end{align*}
which is our desired estimate.
\end{proof}

Next, we proceed to prove the Proposition \ref{Prop: Kernel estimate}.
\begin{proof}[Proof of Proposition \ref{Prop: Kernel estimate}]
Let us define $F(x, \lambda)= \sigma_j(x, 2^j\sqrt{\lambda}) e^{\lambda} $. Then we have $\sigma_j(x, \sqrt{\lambda})= F(x, 2^{-2j} \lambda) e^{-2^{-2j}\lambda}$. Therefore, by the Fourier inversion formula, we get
\begin{align*}
    \sigma_j(x, \sqrt{\lambda}) &= \frac{1}{2 \pi} \int_{\mathbb{R}} \widehat{F}(x, \tau) \exp{(-(1-i \tau)2^{-2j} \lambda)} \  d\tau ,
\end{align*}
where $\widehat{F}(x, \tau)$ stands for the Fourier transform of $F$ with respect to the second variable for fixed $x$. From this we also write
\begin{align}
\label{Kernel expression for K_j}
    K_j(x,y) &= \frac{1}{2\pi} \int_{\mathbb{R}} \widehat{F}(x, \tau) K_{\exp{(-(1-i \tau)2^{-2j} \mathcal{L})}}(x,y) \ d\tau .
\end{align}
Then, by the Leibniz formula, we write
\begin{align*}
    \nabla_x^{\Gamma} K_j(x,y) &= C \sum_{\Gamma_1+\Gamma_2=\Gamma} \int_{\mathbb{R}} \nabla_x^{\Gamma_1} \widehat{F}(x, \tau) \nabla_x^{\Gamma_2} K_{\exp{(-(1-i \tau)2^{-2j} \mathcal{L})}}(x,y) \ d\tau .
\end{align*}
Therefore using Lemma \ref{Lemma: L infinity gradient estimate of heat kernel}, H\"older's inequality and $\sigma \in S^0_{1,0}(\sqrt{\mathcal{L}})$, we get
\begin{align}
\label{Estimate: Pseudo-kernel with weight}
    & |\nabla_x^{\Gamma} K_j(x,y)| (1+2^j |x-y|)^s \\
    &\nonumber \leq C \sum_{\Gamma_1+\Gamma_2=\Gamma} \int_{\mathbb{R}} |\nabla_x^{\Gamma_1} \widehat{F}(x, \tau)| |\nabla_x^{\Gamma_2} K_{\exp{(-(1-i \tau)2^{-2j} \mathcal{L})}}(x,y)| (1+2^j |x-y|)^s \ d\tau \\
    &\nonumber \leq C 2^{j n} 2^{j |\Gamma_2|} 2^{j N} \rho(x)^{N/2} \rho(y)^{N/2} \int_{\mathbb{R}} |\nabla_x^{\Gamma_1} \widehat{F}(x, \tau)| (1+|\tau|)^s \ d\tau \\
    &\nonumber \leq C 2^{j n} 2^{j |\Gamma_2|} 2^{j N} \rho(x)^{N/2} \rho(y)^{N/2} \left
    (\int_{\mathbb{R}} |\nabla_x^{\Gamma_1} \widehat{F}(x, \tau)|^2 (1+|\tau|^2)^{s+1/2+\varepsilon} \ d\tau \right)^{1/2} \\
    & \hspace{8cm} \nonumber \left(\int_{\mathbb{R}} \frac{d\tau}{(1+|\tau|^2)^{1/2+\varepsilon}} \right)^{1/2} \\
    &\nonumber \leq C 2^{j n} 2^{j |\Gamma_2|} 2^{j N} \rho(x)^{N/2} \rho(y)^{N/2} \|\nabla_x^{\Gamma_1} F(x, \cdot)\|_{L^2_{s+1/2+\varepsilon}} \\
    &\nonumber \leq C 2^{j n} 2^{j |\Gamma_2|} 2^{j N} \rho(x)^{N/2} \rho(y)^{N/2} \|\nabla_x^{\Gamma_1}\sigma_j(x, 2^j \cdot)\|_{L^{\infty}_{s+1/2+\varepsilon}} \\
    &\nonumber \leq C 2^{j n} 2^{j |\Gamma|} 2^{j N} \rho(x)^{N/2} \rho(y)^{N/2} ,
\end{align}
where we have used the fact that, for any $s>0$, $\|\sigma(x, 2^j \cdot)\|_{L^{\infty}_{s}} \leq C$.

Similarly one can also prove for $N \geq 0$,
\begin{align}
\label{Gradient kernel estimate with only x}
    |\nabla_x^{\Gamma} K_j(x,y)| (1+2^j |x-y|)^s &\leq C 2^{jn} 2^{j |\Gamma|} 2^{j N} \rho(x)^{N} ,
\end{align}
and
\begin{align}
\label{Gradient kernel estimate with only y}
    |\nabla_x^{\Gamma} K_j(x,y)| (1+2^j |x-y|)^s &\leq C 2^{jn} 2^{j |\Gamma|} 2^{jN} \rho(y)^{N} .
\end{align}
Now, from \eqref{Estimate: Pseudo-kernel with weight} we can see that
\begin{align*}
    |\nabla_x^{\Gamma} K_j(x,y)| |x-y|^s &\leq C 2^{j n} 2^{j |\Gamma|} 2^{-j s} 2^{j N} \rho(x)^{N/2} \rho(y)^{N/2} .
\end{align*}
Then taking $s=N+\beta$ with $\beta \geq 0$ we get
\begin{align*}
    |\nabla_x^{\Gamma} K_j(x,y)| &\leq C |x-y|^{-N-\beta} 2^{j n} 2^{j |\Gamma|} 2^{-j N-j \beta} 2^{j N} \rho(x)^{N/2} \rho(y)^{N/2} .
\end{align*}
So for any $N\geq 0$
\begin{align*}
    \left( \frac{|x-y|}{\rho(x)} \right)^{N/2} \left( \frac{|x-y|}{\rho(y)} \right)^{N/2} |\nabla_x^{\Gamma} K_j(x,y)| &\leq C |x-y|^{-\beta} 2^{j n} 2^{j |\Gamma|} 2^{-j \beta} .
\end{align*}
With the help of \eqref{Gradient kernel estimate with only x} and \eqref{Gradient kernel estimate with only y} and using the same idea as above for any $N \geq 0$ we obtain
\begin{align*}
    \left( \frac{|x-y|}{\rho(x)} \right)^{N} |\nabla_x^{\Gamma} K_j(x,y)| &\leq C |x-y|^{-\beta} 2^{j n} 2^{j |\Gamma|} 2^{-j \beta} ,\\
    \text{and} \quad \quad \left( \frac{|x-y|}{\rho(y)} \right)^{N} |\nabla_x^{\Gamma} K_j(x,y)| &\leq C |x-y|^{-\beta} 2^{j n} 2^{j |\Gamma|} 2^{-j \beta} .
\end{align*}
Hence, combining all the above three estimates yields
\begin{align}
\label{Kernel estimate Kj with dyadic support}
    \left(1+ \frac{|x-y|}{\rho(x)} + \frac{|x-y|}{\rho(y)} \right)^{N} |\nabla_x^{\Gamma} K_j(x,y)| &\leq C |x-y|^{-\beta} 2^{j n} 2^{j |\Gamma|} 2^{-j \beta} .
\end{align}
Estimate for $\nabla_y^{\Gamma} K_j(x,y)$ is similar, even easier, so we omit the details for brevity. This completes the proof of the Proposition \ref{Prop: Kernel estimate}.
\end{proof}

\begin{lemma}
\label{Lemma: Difference kernel estimate with sum over j}
    Let $Q=Q(x_Q, r_Q)$ be a cube centered at $x_Q$ with side length $r_Q>0$. Then for any $y, z \in Q$ and $w \in (2 Q)^c$ we have
    \begin{enumerate}
        \item $\displaystyle{\sum_{j=0}^{\infty} |K_j(y,w)-K_j(z,w)| \leq C |y-z| \left(1+\frac{|x_Q-w|}{\rho(x_Q)} \right)^{-N}  |x_Q-w|^{-(n+1)}} $,
        \item $\displaystyle{\sum_{j=0}^{\infty} |K_j(y,w)| \leq C \left(1+\frac{|x_Q-w|}{\rho(x_Q)} \right)^{-N}  |x_Q-w|^{-n}} $ ,
    \end{enumerate}
    for all $N\geq 0$.
\end{lemma}
\begin{proof}
    First, we will prove $(1)$. For each $j$, using mean value theorem and Proposition (\ref{Prop: Kernel estimate}) we get
    \begin{align*}
        |K_j(y,w)-K_j(z,w)| &\leq C |y-z| \int_{0}^{1} |\nabla K_j((1-t)y+tz,w)|\ dt \\
        &\leq C |y-z| \left(1+ \frac{|\gamma(t)-w|}{\rho(\gamma(t))} \right)^{-N'} |\gamma(t)-w|^{-\beta} 2^{j (n+1)} 2^{-j \beta} .
    \end{align*}
    where $\gamma(t) = (1-t)y+tz$. Since $z,w \in Q$ and $y \in Q_k$, we have $|x_Q-w| \sim |\gamma(t)-w|$. Then applying Remark (\ref{Remark: Equivalence of potential for all x,y}) we have
    \begin{align*}
        |K_j(y,w)-K_j(z,w)| &\leq C |y-z| \left(1+ \frac{|x_Q-w|}{\rho(x_Q)} \right)^{-N'/(l_0+1)} |x_Q-w|^{-\beta} 2^{j (n+1)} 2^{-j \beta} .
    \end{align*}
    Now summing over $j$ we can see
    \begin{align*}
        &\sum_{j=0}^{\infty} |K_j(y,w)-K_j(z,w)| \\
        &\leq C |y-z| \left(1+ \frac{|x_Q-w|}{\rho(x_Q)} \right)^{-N'/(l_0+1)} \Bigg[\sum_{2^j \leq |x_Q-w|^{-1}} |x_Q-w|^{-\beta}  2^{j (n+1)} 2^{-j \beta} \\
        & \hspace{8cm} + \sum_{2^j > |x_Q-w|^{-1}} |x_Q-w|^{-\beta}  2^{j (n+1)} 2^{-j \beta} \Bigg] \\
        &=: I + II .
    \end{align*}
    For $I$, taking $\beta=0$ we get
    \begin{align}
    \label{Estimate: Sum of j for I}
        I &\leq C |y-z| \left(1+ \frac{|x_Q-w|}{\rho(x_Q)} \right)^{-N'/(l_0+1)} \sum_{2^j \leq |x_Q-w|^{-1}} 2^{j (n+1)} \\
        &\nonumber\leq C |y-z| \left(1+\frac{|x_Q-w|}{\rho(x_Q)} \right)^{-N} |x_Q-w|^{-(n+1)} ,
    \end{align}
    and for $II$, choosing $\beta>n+1$ we obtain
    \begin{align}
    \label{Estimate: Sum of j for II}
        II &\leq C |y-z| \left(1+ \frac{|x_Q-w|}{\rho(x_Q)} \right)^{-N'/(l_0+1)} |x_Q-w|^{-\beta} \sum_{2^j > |x_Q-w|^{-1}}  2^{j (n+1)} 2^{-j \beta} \\
        &\nonumber \leq C |y-z| \left(1+\frac{|x_Q-w|}{\rho(x_Q)} \right)^{-N}  |x_Q-w|^{-(n+1)} ,
    \end{align}
    where $N=N'/(l_0+1)$. Hence, by combining (\ref{Estimate: Sum of j for I}) and (\ref{Estimate: Sum of j for II}) we get our required estimate. 
    
    The proof of part $(2)$ is similar to part $(1)$. The only difference is that here we have to use the pointwise estimate of the kernel instead of the gradient estimate of the Kernel. So we omit the details.
    
    This completes the proof of the lemma.
\end{proof}

\begin{remark}
If $x, y \in \mathbb{R}^n$, then using the similar idea as Lemma \ref{Lemma: Difference kernel estimate with sum over j}, one can prove the following:
\begin{align}
\label{Kernel estimate for all x and y}
    \sum_{j=0}^{\infty} |K_j(x,y)| &\lesssim \left(1+\frac{|x-y|}{\rho(x)} \right)^{-N}  |x-y|^{-n}  ,
\end{align}
for all $N\geq 0$.
    
\end{remark}

\section{\texorpdfstring{$L^p$}{} boundedness of Schr\"odinger pseudo-multipliers}
\label{section: unweighted Lp boundedness of Schrodinger}
The aim of this section is to prove the $L^p$-boundedness of the Schr\"odinger pseudo-multiplier $ \sigma(x, \sqrt{\mathcal{L}})$ for the symbol class $\sigma \in S^0_{1,0}(\sqrt{\mathcal{L}})$ (Proposition \ref{prop:L2-sigma-j} and Theorem \ref{Theorem: L^2 boundedness of the symbol S^0_{1,0}}). In order to prove the $L^p$-boundedness for $1<p\leq 2$, we show that $T$ is weak type $(1,1)$ and bounded on $L^2$, then an interpolation argument gives the required result.

\begin{proof}[Proof of Proposition \ref{prop:L2-sigma-j}]

Recall that $\mathcal{T}_N f =\sum_{j=0}^{N} \sigma_j(x, \sqrt{\mathcal{L}})f$. We want to prove that if $\sigma \in S^0_{1,0}(\sqrt{\mathcal{L}})$, then 
\begin{align}
\label{proof of uniform operator norm}
    \|\mathcal{T}_{N}\|_{L^2\rightarrow L^2} \leq C ,
\end{align}
where the positive constant $C$ independent of $N$. 

Note that in order to prove \eqref{proof of uniform operator norm}, it is enough to show that for $f \in \mathcal{D}(\mathcal{L})$,
\begin{align}
\label{L2 boundedness for Tl operator}
    \|\mathcal{T}_N f\|_{L^2} &\leq C \|f\|_{L^2},
\end{align}
where $C$ is the non-negative constant and independent of $N$.

In order to prove \eqref{L2 boundedness for Tl operator}, we claim that for large $r>0$,
\begin{align}
\label{eq:L2-claim}
    \int_{B(u,1)} |\mathcal{T}_N f(x)|^2 \ dx &\leq C \int_{\mathbb{R}^n} \frac{|f(y)|^2}{(1+|y-u|)^{2r}} dy ,
\end{align}
for all $u \in \mathbb{R}^n$. 

Assuming the claim \eqref{eq:L2-claim} for a moment, we first conclude the proof of \eqref{L2 boundedness for Tl operator}. Integrating the left-hand side of \eqref{eq:L2-claim} with respect to $u$, we get
\begin{align*}
    \int_{\mathbb{R}^n} \int_{B(u,1)} |\mathcal{T}_N f(x)|^2 \ dx \ du &= \int_{\mathbb{R}^n} \left(\int_{\mathbb{R}^n} \chi_{\{u: |x-u|<1\}}(u) du \right) |\mathcal{T}_N f(x)|^2 \ dx \sim \int_{\mathbb{R}^n} |\mathcal{T}_N f(x)|^2 \ dx .
\end{align*}
On the other hand, for the right-hand side of \eqref{eq:L2-claim}, again integrating with respect to $u$, yields
\begin{align*}
    \int_{\mathbb{R}^n} \int_{\mathbb{R}^n} \frac{|f(y)|^2}{(1+|y-u|)^{2r}} dy \ du &= \int_{\mathbb{R}^n} \left(\int_{\mathbb{R}^n} \frac{du}{(1+|y-u|)^{2r}} \right) |f(y)|^2 \ dy \lesssim \int_{\mathbb{R}^n} |f(y)|^2 \ dy ,
\end{align*}
whenever $r>n/2$.

Combining the above two estimates yields \eqref{L2 boundedness for Tl operator}. Therefore, to complete the proof of \eqref{L2 boundedness for Tl operator}, it suffices to prove the claim \eqref{eq:L2-claim}. 

Let us define
    \begin{align*}
    K_j^z(x,y) &= \frac{1}{2\pi} \int_{\mathbb{R}} \widehat{F}(z, \tau) K_{\exp{(-(1-i \tau)2^{-2j} \mathcal{L})}}(x,y) \ d\tau .
\end{align*}
where $F(x, \lambda)= \sigma_j(x, 2^{j}\sqrt{\lambda}) e^{\lambda} $. Then note that $K_{j}^x(x,y) = K_j(x,y)$, where $K_j(x,y)$ is the kernel of $\sigma_j(x, \sqrt{\mathcal{L}})$ (see \eqref{Kernel expression for K_j}).

Fix $u \in \mathbb{R}^n$ and write $f=f_1 + f_2$ such that $\supp{f_1} \subseteq B(u, 3)$ and $\supp{f_2} \subseteq B(u,2)^c$. Also, let $\Theta_u \in C_c^{\infty}(\mathbb{R}^n)$ be such that $\Theta_u =1$ in $B(u,1)$ and supported on $B(u,3)$. In view of this decomposition of $f$, it suffices to estimate \eqref{eq:L2-claim} for $\mathcal{T}_N f_1$ and $\mathcal{T}_N f_2$ separately.

\subsection*{Estimate for \texorpdfstring{$\mathcal{T}_N f_1$:}{}}
Note that for $x \in B(u,1)$, we write
\begin{align*}
    \mathcal{T}_N f_1(x) = \int_{\mathbb{R}^n} \Theta_u(x) \sum_{j=0}^N K_j^{x}(x,y) f_1(y) \ dy.
\end{align*}
Therefore, using the Sobolev embedding theorem we get
\begin{align*}
    |\mathcal{T}_N f_1(x)|^2 &\lesssim \sup_{z \in \mathbb{R}^n} \left| \int_{\mathbb{R}^n} \Theta_u(z) \sum_{j=0}^N K_j^{z}(x,y) f_1(y) \ dy \right|^2 \\
    &\lesssim \sum_{i=1}^n \int_{\mathbb{R}^n} \left| \int_{\mathbb{R}^n} \frac{\partial}{\partial z_i} \left(\Theta_u(z) \sum_{j=0}^N K_j^{z}(x,y) \right) f_1(y) \ dy \right|^2 dz \\
    &\lesssim \sum_{i=1}^n \int_{B(u,3)} \left| \int_{\mathbb{R}^n} \frac{\partial}{\partial z_i} \left(\Theta_u(z) \sum_{j=0}^N K_j^{z}(x,y) \right) f_1(y) \ dy \right|^2 dz . 
\end{align*}
Then using Plancherel's theorem and since $\sigma \in S_{1,0}^0(\sqrt{\mathcal{L}})$ we have
\begin{align*}
    \int_{B(u,1)} |\mathcal{T}_l f_1(x)|^2 \ dx &\lesssim \sum_{i=1}^n \int_{B(u,3)} \left( \int_{B(u,1)} \left| \int_{\mathbb{R}^n} \frac{\partial}{\partial z_i} \left(\Theta_u(z) \sum_{j=0}^N K_j^{z}(x,y) \right) f_1(y) \ dy \right|^2 dx \right) \ dz \\
    &\lesssim \sum_{i=1}^n \int_{B(u,3)} \sup_{\eta} \left| \sum_{j=0}^N \frac{\partial}{\partial z_i} \sigma_j(z, \eta) \right|^2 \|f_1\|_{L^2}^2 \ dz \\
    &\lesssim \|f_1\|_{L^2}^2 \int_{B(u,3)} dz \\
    &\lesssim \int_{\mathbb{R}^n} \frac{|f(y)|^2}{(1+|y-u|)^{2r}} \ dy ,
\end{align*}
where in the last line we have used the fact that $\supp{f_1} \subseteq B(u, 1)$.

\subsection*{Estimate for \texorpdfstring{$\mathcal{T}_N f_2$:}{}} For $x\in B(u,1)$ we have
\begin{align*}
    \mathcal{T}_N f_2(x) &= \int_{B(u,2)^c} |x-y|^r \Theta_u(x) \sum_{j=0}^N K_j^{x} (x,y) |x-y|^{-r} f_2(y) \ dy .
\end{align*}
Note that if $y \in \supp{f_2}$ then $|y-u| \geq 2$. So that this implies
\begin{align*}
    |x-y| \geq \frac{1}{2} |y-u| \geq \frac{1}{4} (1+|y-u|) .
\end{align*}
Therefore, by the Cauchy-Schwarz inequality we have
\begin{align}
\label{Application of Cauchy-Schwartz in f2 case}
    |\mathcal{T}_N f_2(x)|^2 &\lesssim \left(\int_{\mathbb{R}^n} |x-y|^{2r} |\Theta_u(x)|^2 \left|\sum_{j=0}^N K_j^x(x,y) \right|^2 dy \right) \left(\int_{\mathbb{R}^n} \frac{|f_2(y)|^2}{(1+|y-u|)^{2r}} dy \right). 
\end{align}
Now, an application of the Sobolev embedding theorem yields
\begin{align}
\label{Application of Sobolov embedding in f2 case}
    & \int_{\mathbb{R}^n} |x-y|^{2r} |\Theta_u(x)|^2 \left|\sum_{j=0}^N K_j^x(x,y) \right|^2 dy \\
    &\nonumber \lesssim \int_{\mathbb{R}^n} |x-y|^{2r} \sup_{z \in \mathbb{R}^n} \left|\Theta_u(z) \sum_{j=0}^N K_j^z(x,y) \right|^2 dy \\
    &\nonumber \lesssim \sum_{i=1}^{n} \int_{\mathbb{R}^n} |x-y|^{2r} \int_{\mathbb{R}^n} \left|\frac{\partial}{\partial z_i} \left(\Theta_u(z) \sum_{j=0}^N K_j^z(x,y) \right) \right|^2 dz \, dy \\
    &\nonumber \lesssim \sum_{i=1}^{n} \int_{B(u,3)} \int_{\mathbb{R}^n} |x-y|^{2r} \left|\frac{\partial}{\partial z_i} \left(\Theta_u(z) \sum_{j=0}^N K_j^z(x,y) \right) \right|^2 dy \, dz .
\end{align}
Now claim that for sufficiently large $r$ we have
\begin{align}
\label{Claim in the L2 boundedness}
    \left( \int_{\mathbb{R}^n} |x-y|^{2r} \left|\frac{\partial}{\partial z_i} \left(\Theta_u(z) \sum_{j=0}^N K_j^z(x,y) \right) \right|^2 dy \right)^{1/2} &\leq C ,
\end{align}
where the constant $C$ is independent of $N$.

Note that
\begin{align}
\label{Use of triangle ineq for fixed z}
    & \left( \int_{\mathbb{R}^n} |x-y|^{2r} \left|\frac{\partial}{\partial z_i} \left(\Theta_u(z) \sum_{j=0}^N K_j^z(x,y) \right) \right|^2 dy \right)^{1/2} \\
    &\nonumber \leq \sum_{j=0}^N \left( \int_{\mathbb{R}^n} \left( |x-y|^{r} \left|\frac{\partial}{\partial z_i} \left(\Theta_u(z) K_j^z(x,y) \right) \right| \right)^2 dy \right)^{1/2} .
\end{align}
For $z \in \mathbb{R}^n$, let $\widetilde{K}_j^z(x,y)$ be the kernel corresponding to the multiplier $\frac{\partial}{\partial z_i} (\Theta_u(z) \sigma_j(z, \cdot))$. Then using part (a) of Lemma 4.3 of \cite{Duong_Ouhabaz_Sikora_Sharp_Multiplier_2002} and $\sigma \in S^0_{1,0}(\sqrt{\mathcal{L}})$, for $\varepsilon>0$ we get
\begin{align*}
    \left(\int_{\mathbb{R}^n} (2^j |x-y|)^{2r} | \widetilde{K}_j^z(x,y)|^2 dy \right)^{1/2} &\lesssim |B(x, 2^{-j})|^{-1/2} \left\|\frac{\partial}{\partial z_i} (\Theta_u(z) \sigma_j(z, 2^j \cdot)) \right\|_{L^{\infty}_{r+\varepsilon}} \\
    &\lesssim 2^{j n/2} 2^{j (r+\varepsilon)} 2^{-j (r+\varepsilon)} \\
    &\lesssim 2^{j n/2}.
\end{align*}
Now choosing $r>n/2$ from \eqref{Use of triangle ineq for fixed z} we obtain
\begin{align*}
    \left( \int_{\mathbb{R}^n} |x-y|^{2r} \left|\frac{\partial}{\partial z_i} \left(\Theta_u(z) \sum_{j=0}^N K_j^z(x,y) \right) \right|^2 dy \right)^{1/2} &\leq C \sum_{j=0}^{\infty} 2^{-j(r-n/2)} \leq C ,
\end{align*}
which proves the claim \eqref{Claim in the L2 boundedness}.

Therefore plugging the estimates \eqref{Claim in the L2 boundedness}, \eqref{Application of Sobolov embedding in f2 case} into \eqref{Application of Cauchy-Schwartz in f2 case} we have
\begin{align*}
    \int_{B(u,1)} |\mathcal{T}_N f(x)|^2 \ dx &\lesssim \int_{B(u,1)} \int_{\mathbb{R}^n} \frac{|f_2(y)|^2}{(1+|y-u|)^{2r}} \, dy \, dx \lesssim \int_{\mathbb{R}^n} \frac{|f_2(y)|^2}{(1+|y-u|)^{2r}} \, dy .
\end{align*}
This completes the proof of the claim \eqref{eq:L2-claim}.
\end{proof}

\begin{proof}[Proof of Theorem \ref{Theorem: L^2 boundedness of the symbol S^0_{1,0}}]

Weak $(1,1)$ boundedness of $\sigma(x, \sqrt{\mathcal{L}})$ follows from the more general result proved in Theorem 3.7 of \cite{Bagchi_Basak_Garg_Ghosh_Grushin_2023}. Therefore using standard interpolation argument we can conclude the $L^p$-boundedness result of $\sigma(x, \sqrt{\mathcal{L}})$ for $1<p\leq 2$.
\end{proof}

\section{Quantitative weighted \texorpdfstring{$L^p$}{}-estimates for the Schr\"odinger pseudo-multipliers}
\label{Section: weighted Lp bound for operator}
In this section, we present the proof of Theorem~\ref{Theorem: Weighted L^p boundedness of pseudo multiplier}, which provides the quantitative weighted estimates for the Schr\"odinger pseudo-multiplier $\sigma(x, \sqrt{\mathcal{L}})$ for $\sigma$ belonging to the symbol class $S^0_{1,0}(\sqrt{\mathcal{L}})$. The proof takes some inspiration from the approach in \cite{Bui_Bui_Duong_square_function_new_weight_2022}. We decompose the operator into two parts: one is dominated by the sparse operator $\mathcal{A}_{\mathcal{S}}^{r, 1, \rho, N}$, and the other by the maximal functions $\mathcal{M}_{r,\rho,\theta}$ associated with $\mathcal{L}$. By applying the weighted $L^p$-boundedness of the sparse operator (see Proposition~\ref{Prop: Boundedness of sparse operator}) and of the maximal functions (see Proposition~\ref{Prop: Maximal function new weight with r power}), we obtain the desired result. Before proceeding to the proof, we first discuss two useful propositions related to the sparse domination result.

Let $T$ be a sublinear operator. Then for $\alpha>0$, the localized sharp grand maximal truncation operator denoted by $\mathcal{M}_{T, \alpha}^{\#}$, defined by
\begin{align*}
    \mathcal{M}_{T, \alpha}^{\#}f(x) &= \sup_{Q \ni x}\, \esssup_{x', x'' \in Q} |T(f \chi_{\mathbb{R}^n \setminus \alpha Q})(x')-T(f \chi_{\mathbb{R}^n \setminus \alpha Q})(x'')| .
\end{align*}

Then we have the following sparse domination result for the operator $T$, which directly follows from the proof of \cite[Theorem 1.1]{Lerner_Ombrosi_pointwise_sparse_domination_2020}.
\begin{proposition}
\label{Prop: Pointwise sparse domination}
Let $T$ be a linear operator such that $T$ is of weak type $(q,q)$ and $\mathcal{M}_{T, \alpha}^{\#}$ is of weak type $(r,r)$ for some $\alpha \geq 3$ and $1\leq q, r<\infty$. Let $s=\max\{q,r\}$. Then for every function $f \in L^s(\mathbb{R}^n)$ supported in $\alpha Q$ for some cube $Q$, there exists a sparse family $\mathcal{S}$ of subcubes of $Q$ such that for a.e. $x \in Q$,
\begin{align*}
    |T(f \chi_{\alpha Q})(x)| \chi_Q(x) &\lesssim \sum_{P \in \mathcal{S}} \left(\frac{1}{|\alpha P|} \int_{\alpha P} |f|^s \right)^{1/s} \chi_{P}(x) .
\end{align*}
\end{proposition}

As an application of Proposition \ref{Prop: Pointwise sparse domination}, we have the following sparse domination result for the Schr\"odinger pseudo-multipliers.
\begin{proposition}
\label{Prop: Pointwise sparse domination for pseudo}
Let $\sigma \in S^{0}_{1,0}(\sqrt{\mathcal{L}})$ and $1\leq r < \infty$, $T= \sigma(x, \sqrt{\mathcal{L}})$.  Then for every function $f \in L^r(\mathbb{R}^n)$ supported in $\alpha Q$ for some cube $Q$ and $\alpha \geq 3$, there exists a sparse family $\mathcal{S}$ of subcubes of $Q$ such that for a.e. $x \in Q$,
\begin{align*}
    |T(f \chi_{\alpha Q})(x)| \chi_Q(x) &\lesssim \sum_{P \in \mathcal{S}} \left(\frac{1}{|\alpha P|} \int_{\alpha P} |f|^r \right)^{1/r} \chi_{P}(x) .
\end{align*}
    
\end{proposition}

\begin{proof}
In view of Proposition \ref{Prop: Pointwise sparse domination}, it is enough to check that $T$ is of weak type $(1,1)$ and $\mathcal{M}_{T, \alpha}^{\#}$ is of weak type $(r,r)$ for some $\alpha \geq 3$ and $1\leq r<\infty$. From Theorem \ref{Theorem: L^2 boundedness of the symbol S^0_{1,0}} we already know that $T$ is weak $(1,1)$. Therefore, it remains to verify that $\mathcal{M}_{T, \alpha}^{\#}$ is of weak type $(r,r)$.

Let $x \in \mathbb{R}^n$ and $Q= Q(x_Q, r_Q)$ be a cube such that $Q \ni x$. Also, let $f \in C_c^{\infty}(\mathbb{R}^n)$. Then, applying Lemma \ref{Lemma: Difference kernel estimate with sum over j} and H\"older's inequality we obtain
\begin{align*}
    |T(f \chi_{\mathbb{R}^n \setminus \alpha Q})(x')-T(f \chi_{\mathbb{R}^n \setminus \alpha Q})(x'')| &\lesssim \int_{|x_Q-y|>\alpha r} \sum_{j=0}^{\infty} |K_j(x', y) - K_j(x'',y)| |f(y)| \ dy \\
    &\lesssim \int_{|x_Q-y|>\alpha r} |x'-x''| |x_Q-y|^{-(n+1)} |f(y)| \ dy \\
    &\lesssim \sum_{k=\lfloor \log_2 \alpha \rfloor}^{\infty} \int_{2^k r_Q < |x_Q-y|\leq 2^{k+1} r_Q} \frac{r_Q}{(2^{k+1}r_Q)^{n+1}} |f(y)| \ dy \\
    &\lesssim \sum_{k=\lfloor \log_2 \alpha \rfloor}^{\infty} 2^{-k} \frac{1}{|Q_k|} \int_{Q_k} |f(y)| \ dy \\
    &\lesssim \sum_{k=\lfloor \log_2 \alpha \rfloor}^{\infty} 2^{-k} \mathcal{M}_{r}f(x) \lesssim \mathcal{M}_{r}f(x) .
\end{align*}
This implies $\mathcal{M}_{T, \alpha}^{\#}f(x) \lesssim \mathcal{M}_{r}f(x)$. Since $\mathcal{M}_{r}$ is of weak type $(r,r)$, we can conclude that $\mathcal{M}_{T, \alpha}^{\#}$ is also of weak type $(r,r)$.
\end{proof}

Now we are ready to prove Theorem \ref{Theorem: Weighted L^p boundedness of pseudo multiplier}.

\begin{proof}[Proof of Theorem \ref{Theorem: Weighted L^p boundedness of pseudo multiplier}]

First, using Lemma \ref{Lemma: Decomposition of space into ball}, we write
\begin{align}
\label{Decomposing the operator}
    \|Tf\|_{L^p(\omega)}^p &\leq \sum_{m \in \mathbb{N}} \|\chi_{B_m} Tf\|_{L^p(\omega)}^p ,
\end{align}
where $B_m = B(x_m, \rho(x_m))$.

For each $m \in \mathbb{N}$, using Lemma \ref{Lemma: System of dyadic cubes} there exists a cube $Q_m \in \mathscr{D}^k$ with side length $r_{Q_m}$, such that $B_m \subset Q_m \subset \Lambda_0 B_m$ for some $\Lambda_0>1$. Now with the help of $Q_m$, we decompose $\chi_{B_m} Tf$ as follows.
\begin{align*}
    \chi_{B_m}(x)\, Tf(x) &= \chi_{B_m}(x)\, T(f \chi_{4Q_m})(x) + \chi_{B_m}(x)\, T(f \chi_{(4Q_m)^c})(x), \quad x \in \mathbb{R}^n.
\end{align*}

Therefore, from \eqref{Decomposing the operator} we can write
\begin{align}
\label{Decomposing the operator second time}
    \|Tf\|_{L^p(\omega)}^p &\lesssim \sum_{m \in \mathbb{N}} \|\chi_{B_m} T(f \chi_{4Q_m})\|_{L^p(\omega)}^p + \sum_{m \in \mathbb{N}} \|\chi_{B_m} T(f \chi_{(4Q_m)^c})\|_{L^p(\omega)}^p \\
    &\nonumber =: I_1 + I_2 .
\end{align}

\subsection{Estimate of \texorpdfstring{$I_1$}{}}
For each $m$, using Proposition \ref{Prop: Pointwise sparse domination for pseudo}, there exists a sparse family $\mathcal{S}_m$ of subcubes of $Q_m$ such that for every $x \in Q_m$,
\begin{align*}
    |T(f \chi_{4Q_m})(x)| &\lesssim \sum_{P \in \mathcal{S}_m} \left(\frac{1}{|4 P|} \int_{4 P} |f|^r \right)^{1/r} \chi_{P}(x) .
\end{align*}
Recall that $B_m = B(x_m,\rho(x_m)) \subset Q_m \subset \Lambda_0 B_m$. Therefore from Lemma \ref{Lemma: Equivalence of two potential}, for any $x \in Q_m$ yields
\begin{align*}
    \rho(x) \sim \rho(x_m) \sim r_{Q_m} .
\end{align*}
Since $P \subset Q_m$, the above observation implies that
\begin{align}
\label{For P in cubr radius and critial function relation}
    \frac{r_P}{\rho(x_P)} \sim \frac{r_P}{\rho(x_m)} \sim \frac{r_P}{r_{Q_m}} \lesssim 1 .
\end{align}
Consequently, we get
\begin{align*}
    |T(f \chi_{4Q_m})(x)| \chi_{Q_m}(x) &\lesssim \sum_{P \in \mathcal{S}_m} \left(\frac{1}{|4 P|} \int_{4 P} |f|^r \right)^{1/r} \left(1+\frac{r_P}{\rho(x_P)}\right)^{-N} \chi_{P}(x) \\
    &\lesssim \mathcal{A}_{\mathcal{S}}^{r, 1, \rho, N}(f \chi_{4Q_m})(x) .
\end{align*}
Finally applying Proposition \ref{Prop: Boundedness of sparse operator} for any $ N \geq \theta \cdot \frac{p}{r} \max\left\{ \frac{1}{p - r}, 1 \right\} $, we can conclude that
\begin{align*}
    I_1 \lesssim \sum_{m \in \mathbb{N}} \|\chi_{B_m} T(f \chi_{4Q_m})\|_{L^p(\omega)}^p &\lesssim \sum_{m \in \mathbb{N}} \|\mathcal{A}_{\mathcal{S}}^{r, 1, \rho, N}(f \chi_{4Q_m})\|_{L^p(\omega)}^p \\
    &\lesssim [\omega]_{A_{p/r}^{\rho, \theta}}^{p \max\left\{ \frac{1}{p - r}, 1 \right\}} \sum_{m \in \mathbb{N}} \|f \chi_{4Q_m}\|_{L^p(\omega)}^p \\
    &\lesssim [\omega]_{A_{p/r}^{\rho, \theta}}^{p \max\left\{ \frac{1}{p - r}, 1 \right\}} \|f\|_{L^p(\omega)}^p ,
\end{align*}
where in the last line we have used Lemma \ref{Lemma: Decomposition of space into ball}.

\subsection{Estimate of \texorpdfstring{$I_2$}{}}
Applying Lemma \ref{Lemma: Difference kernel estimate with sum over j}, choosing $N=\theta/(p-r)+1 $ and H\"older's inequality for any $x \in B_m$ we get
\begin{align}
\label{Dominated the operator by maximal function}
    |T(f \chi_{(4Q_m)^c})(x)| & \lesssim \int_{|x_m-y|>4\rho(x_m)} \sum_{j=0}^{\infty} |K_j(x,y)| |f(y)| \ dy  \\
    &\nonumber \lesssim \int_{|x_m-y|>4\rho(x_m)} \left(1+\frac{|x_m-y|}{\rho(x_m)} \right)^{-(\frac{\theta}{p-r}+1)} |x_m-y|^{-n} |f(y)| \ dy \\
    &\nonumber \lesssim \sum_{k=2}^{\infty} \int_{2^k \rho(x_m) < |x_m-y|\leq 2^{k+1} \rho(x_m)} \frac{|f(y)| \ dy}{(1+\frac{2^k \rho(x_m)}{\rho(x_m)})^{\frac{\theta}{p-r}+1} (2^{k+1} \rho(x_m))^{n}} \\
    &\nonumber \lesssim \sum_{k=2}^{\infty} \frac{2^{-k}}{(1+\frac{2^k \rho(x_m)}{\rho(x_m)})^{\frac{\theta}{p-r}} |2^{k+1} Q_m|} \int_{2^{k+1}Q_m} |f(y)| \ dy \\
    &\nonumber \lesssim \sum_{k=2}^{\infty} 2^{-k} \left(1+\frac{2^k \rho(x_m)}{\rho(x_m)} \right)^{-\frac{\theta}{p-r}} \left(\frac{1}{|2^{k+1} Q_m|} \int_{2^{k+1}Q_m} |f(y)|^r \ dy \right)^{1/r} \\
    &\nonumber \lesssim \sum_{k=2}^{\infty} 2^{-k} \mathcal{M}_{r,\rho, \frac{\theta}{p-r}}f(x) \lesssim \mathcal{M}_{r,\rho, \frac{\theta}{p-r}}f(x) .
\end{align}
Finally, using Proposition \ref{Prop: Maximal function new weight with r power} we obtain
\begin{align*}
    I_2 \lesssim \sum_{m \in \mathbb{N}} \|\chi_{B_m} T(f \chi_{(4Q_m)^c})\|_{L^p(\omega)}^p &\lesssim \sum_{m \in \mathbb{N}} \|\chi_{B_m} \mathcal{M}_{r,\rho, \frac{\theta}{p-r}}f\|_{L^p(\omega)}^p \\
    &\lesssim \|\mathcal{M}_{r,\rho, \frac{\theta}{p-r}}f\|_{L^p(\omega)}^p \\
    &\lesssim [\omega]_{A_{p/r}^{\rho, \theta}}^{p \max\left\{ \frac{1}{p - r}, 1 \right\}} \|f\|_{L^p(\omega)}^p . 
\end{align*}
This completes the proof of the Theorem \ref{Theorem: Weighted L^p boundedness of pseudo multiplier}.
\end{proof}

\section{Quantitative weighted \texorpdfstring{$L^p$}{}-estimates for commutators of the Schr\"odinger pseudo-multipliers}
\label{Section: Weighted boundedness of commutators}
This section is devoted to the proof of Theorem \ref{Theorem: Weighted L^p boundedness of pseudo commutator}, which provides quantitative weighted estimates for the commutators of the Schr\"odinger pseudo-multiplier $\sigma(x, \sqrt{\mathcal{L}})$ with symbol $\sigma \in S^0_{1,0}(\sqrt{\mathcal{L}})$. The strategy is as follows. As in Theorem \ref{Theorem: Weighted L^p boundedness of pseudo multiplier}, we first decompose the operator into two parts: one is dominated by the sparse operator $\mathcal{T}_{\mathcal{S}, b}^{\rho, N}$ and its adjoint $\mathcal{T}_{\mathcal{S}, b}^{\rho, N, \star}$; the other is controlled using the $L^p$-boundedness of the maximal functions $\mathcal{M}_{\rho, \theta}$, $\mathcal{M}_{\mathcal{L} \log \mathcal{L}, \rho, \theta}$, and $\mathcal{M}_w^{\mathscr{D}}$.

Before entering the proof, and similar to Section \ref{Section: weighted Lp bound for operator}, we establish a sparse domination result for the commutator of Schr\"odinger pseudo-multipliers, obtained as a consequence of a more general sparse domination principle.

\begin{proposition}
\label{Prop: Pointwise sparse domination for commutators}
Let $T$ be a sublinear operator such that $T$ is of weak type $(q,q)$ and $\mathcal{M}_{T, \alpha}^{\#}$ is of weak type $(r,r)$ for some $\alpha \geq 3$, $1\leq q, r<\infty$ and let $b \in L^1_{loc}$. Also let $s=\max\{q,r\}$. Then for every function $f \in L^s(\mathbb{R}^n)$ supported in $\alpha Q$ for some cube $Q$, there exists a $3^n$ dyadic system $\{\mathscr{D}^j : j=1, \ldots, 3^n \}$, a cube $R=R_Q \in \mathscr{D}^j$ for some $j$, for which $3Q \subset R_Q$, $|R_Q|\leq 9^n |Q|$ and sparse family $\mathcal{S}$ of subcubes of $Q$ such that for a.e. $x \in Q$
\begin{align*}
    |[b, T](f \chi_{\alpha Q})(x)| \chi_Q(x) &\lesssim \sum_{P \in \mathcal{S}} |b(x)-b_{R_P}| \left(\frac{1}{|\alpha P|} \int_{\alpha P} |f|^{s} \right)^{1/s} \chi_{P}(x) \\
    &\hspace{2cm} + \sum_{P \in \mathcal{S}} \left(\frac{1}{|\alpha P|} \int_{\alpha P} |b-b_{R_P}|^s |f|^s \right)^{1/s} \chi_{P}(x) .
\end{align*}
    
\end{proposition}

The proof of the above proposition can be found in \cite[Theorem 1.1]{Wang_some_remarks_sparse_multilinear_2021}, more specifically one can see, equation 3.5 of \cite{Wang_some_remarks_sparse_multilinear_2021}.

\begin{proposition}
\label{Prop: Pointwise sparse domination for pseudo commutator}
Let $T= \sigma(x, \sqrt{\mathcal{L}})$ with $\sigma \in S^{0}_{1,0}(\sqrt{\mathcal{L}})$ and let $b \in L^1_{loc}$. Then for every function $f \in L^r(\mathbb{R}^n)$ $(1\leq r < \infty)$ supported in $\alpha Q$ for some cube $Q$ and $\alpha \geq 3$, there exists $3^n$ dyadic system $\{\mathscr{D}^j : j=1, \ldots, 3^n \}$, a cube $R=R_Q \in \mathscr{D}^j$ for some $j$, for which $3Q \subset R_Q$, $|R_Q|\leq 9^n |Q|$ and sparse family $\mathcal{S}$ of subcubes of $Q$ such that for a.e. $x \in Q$
\begin{align*}
    |[b, T](f \chi_{\alpha Q})(x)| \chi_Q(x) &\lesssim \sum_{P \in \mathcal{S}} |b(x)-b_{R_P}| \left(\frac{1}{|\alpha P|} \int_{\alpha P} |f|^{r} \right)^{1/r} \chi_{P}(x) \\
    &\hspace{2cm} + \sum_{P \in \mathcal{S}} \left(\frac{1}{|\alpha P|} \int_{\alpha P} |b-b_{R_P}|^r |f|^r \right)^{1/r} \chi_{P}(x) .
\end{align*}
    
\end{proposition}

\begin{proof}
Similar to Proposition \ref{Prop: Pointwise sparse domination for pseudo}, since $T$ is of weak type $(1,1)$ and $\mathcal{M}_{T, \alpha}^{\#}$ is of weak type $(r,r)$ for some $\alpha \geq 3$ and $1\leq r<\infty$, the proof follows from Proposition \ref{Prop: Pointwise sparse domination for commutators}.
\end{proof}

Now we proceed to prove Theorem \ref{Theorem: Weighted L^p boundedness of pseudo commutator}.

\begin{proof}[Proof of Theorem \ref{Theorem: Weighted L^p boundedness of pseudo commutator}]
Analogously to \eqref{Decomposing the operator second time}, we write
\begin{align*}
    \|[b, T]f\|_{L^p(\omega)}^p &\lesssim \sum_{m \in \mathbb{N}} \|\chi_{B_m} [b, T](f \chi_{4Q_m})\|_{L^p(\omega)}^p + \sum_{m \in \mathbb{N}} \|\chi_{B_m} [b, T](f \chi_{(4Q_m)^c})\|_{L^p(\omega)}^p \\
    &\nonumber =: J_1 + J_2 ,
\end{align*}
where $B_m$ and $Q_m$ are the balls and cubes respectively, as defined in the proof of Theorem \ref{Theorem: Weighted L^p boundedness of pseudo multiplier}.

\subsection{Estimate of \texorpdfstring{$J_1$}{}}
For each $m$, applying Proposition \ref{Prop: Pointwise sparse domination for pseudo commutator}, there exists a sparse family $\mathcal{S}_m$ of subcubes of $Q_m$ such that for $x \in Q_m$,
\begin{align*}
    |[b, T](f \chi_{4 Q_m})(x)| &\lesssim \sum_{P \in \mathcal{S}_m} |b(x)-b_{R_P}| \left(\frac{1}{|4 P|} \int_{4 P} |f| \right) \chi_{P}(x) \\
    &\hspace{2cm} + \sum_{P \in \mathcal{S}_m} \left(\frac{1}{|4 P|} \int_{4 P} |b-b_{R_P}| |f| \right) \chi_{P}(x) .
\end{align*}
Similarly to \eqref{For P in cubr radius and critial function relation}, in this case we also have $\frac{r_P}{\rho(x_P)} \lesssim 1$. Therefore, for $x \in Q_m$ we obtain
\begin{align*}
    |[b, T](f \chi_{4 Q_m})(x)| &\lesssim \sum_{P \in \mathcal{S}_m} |b(x)-b_{R_P}| \left(\frac{1}{|4 P|} \int_{4 P} |f| \right) \left(1+\frac{r_P}{\rho(x_P)} \right)^{-N} \chi_{P}(x) \\
    &\hspace{2cm} + \sum_{P \in \mathcal{S}_m} \left(\frac{1}{|4 P|} \int_{4 P} |b-b_{R_P}| |f| \right) \left(1+\frac{r_P}{\rho(x_P)} \right)^{-N} \chi_{P}(x) \\
    &\lesssim \mathcal{T}_{\mathcal{S}, b}^{\rho, N}(|f| \chi_{4Q_m})(x) + \mathcal{T}_{\mathcal{S}, b}^{\rho, N, \star}(|f| \chi_{4Q_m})(x) .
\end{align*}
Let $\nu=\omega^{-\frac{1}{p-1}}$. Consequently, using Proposition \ref{Prop: Operator norm for sparse commutator operator} we conclude that
\begin{align*}
    J_1 &= \sum_{m \in \mathbb{N}} \|\chi_{B_m} [b, T](f \chi_{4Q_m})\|_{L^p(\omega)}^p \\
    &\lesssim \sum_{m \in \mathbb{N}} \|\mathcal{T}_{\mathcal{S}, b}^{\rho, N}(|f| \chi_{4Q_m})\|_{L^p(\omega)}^p + \sum_{m \in \mathbb{N}} \|\mathcal{T}_{\mathcal{S}, b}^{\rho, N, \star}(|f| \chi_{4Q_m})\|_{L^p(\omega)}^p \\
    &\lesssim [\omega]_{A_{p}^{\rho, \theta}}^{2 p \max\left\{ \frac{1}{p - 1}, 1 \right\}} \|b\|_{BMO_{\theta}(\rho)}^p \sum_{m \in \mathbb{N}} \|f \chi_{4Q_m}\|_{L^p(\omega)}^p \\
    &\lesssim [\omega]_{A_{p}^{\rho, \theta}}^{2 p \max\left\{ \frac{1}{p - 1}, 1 \right\}} \|b\|_{BMO_{\theta}(\rho)}^p \|f\|_{L^p(\omega)}^p ,
\end{align*}
where in the last line we have used the Lemma \ref{Lemma: Decomposition of space into ball}.

\subsection{Estimate of \texorpdfstring{$J_2$}{}}
Let us first write $\chi_{B_m} [b, T](f \chi_{(4Q_m)^c})$ as follows:
\begin{align*}
    & \chi_{B_m}(x) [b, T](f \chi_{(4Q_m)^c})(x) \\
    &= \chi_{B_m}(x) \left( b(x)-b_{Q_m} \right) T(f \chi_{(4Q_m)^c})(x) - \chi_{B_m}(x) T((b-b_{Q_m}) f \chi_{(4Q_m)^c})(x) .
\end{align*}
So that we obtain
\begin{align*}
    J_2 &\lesssim \sum_{m \in \mathbb{N}} \|\chi_{B_m} \left( b-b_{Q_m} \right) T(f \chi_{(4Q_m)^c})\|_{L^p(\omega)}^p + \sum_{m \in \mathbb{N}} \|\chi_{B_m} T((b-b_{Q_m}) f \chi_{(4Q_m)^c})\|_{L^p(\omega)}^p \\
    &=: J_{21} + J_{22} .
\end{align*}

\subsubsection{Estimate of \texorpdfstring{$J_{21}$}{}}
By duality we write
\begin{align*}
    & \|\chi_{B_m} \left( b-b_{Q_m} \right) T(f \chi_{(4Q_m)^c})\|_{L^p(\omega)} \\
    &= \sup_{\|g\|_{L^{p'}(\omega)} \leq 1} \left| \int_{\mathbb{R}^n} \chi_{B_m}(x) \left( b(x)-b_{Q_m} \right) T(f \chi_{(4Q_m)^c})(x) g(x) \omega(x)\, dx \right| .
\end{align*}
From \eqref{Dominated the operator by maximal function} taking $r=1$, one can easily see that, for any $x, y \in B_m \subset Q_m$,
\begin{align*}
    |T(f \chi_{(4Q_m)^c})(x)| &\lesssim \inf_{Q_m \ni y} \mathcal{M}_{\rho, \frac{\theta}{p-1}}f(y) .
\end{align*}
Therefore, using the above estimate and generalized H\"older's inequality \eqref{Inequality: generalized Holders inequality} we obtain
\begin{align}
\label{Use of generalized Holder in commutator}
    & \|\chi_{B_m} \left( b-b_{Q_m} \right) T(f \chi_{(4Q_m)^c})\|_{L^p(\omega)} \\
    &\nonumber \lesssim \sup_{\|g\|_{L^{p'}(\omega)} \leq 1} \left( \frac{1}{|Q_m|} \int_{\mathbb{R}^n} | b(x)-b_{Q_m}||g(x)| \omega(x)\, dx \right) \inf_{Q_m \ni y} \mathcal{M}_{\rho, \frac{\theta}{p-1}}f(y) \, |Q_m| \\
    &\nonumber \lesssim \sup_{\|g\|_{L^{p'}(\omega)} \leq 1} \|b-b_{Q_m}\|_{\exp \mathcal{L}, Q_m} \|g \omega\|_{\mathcal{L} \log \mathcal{L}, Q_m} \inf_{Q_m \ni y} \mathcal{M}_{\rho, \frac{\theta}{p-1}}f(y) \, |Q_m| .
\end{align}
Note that from \eqref{Exponential norm of BMO function} for $\theta'=\theta (l_0+1)$, we have
\begin{align}
\label{Estimate of b in exp norm}
    \|b-b_{Q_m}\|_{\exp \mathcal{L}, Q_m} \leq C \Psi_{\theta'}(Q_m) \|b\|_{BMO_{\theta}({\rho})} .
\end{align}
Also for $s>1$, using Lemma \ref{Lemma: llogl estimate in terms of maximal function} there exists a positive number $N_0$ such that
\begin{align}
\label{Estimate of g omega in LlogL norm}
    \|g \omega\|_{\mathcal{L} \log \mathcal{L}, Q_m} &\leq C [\omega]_{A_{p}^{\rho, \theta}} \inf_{Q_m \ni y} \mathcal{M}_{\omega}^{\mathscr{D}}(|g|^s)^{1/s}(y) \left(\frac{1}{|Q_m|} \int_{Q_m} \omega \right) \Psi_{N_0}(Q_m)^{1/s'} .
\end{align}
It is important to note that, since here $Q_m \in \mathscr{D}^k$, so we get $\mathcal{M}_{\omega}^{\mathscr{D}}$ in place of $\mathcal{M}_{\omega}$ in the above estimate.

Since $Q_m \subseteq \Lambda_0 B_m$ with $\Lambda_0>1$ and $B_m=B(x_m, \rho(x_m))$, for any $\theta\geq 0$ we have $\Psi_{\theta}(Q_m) \lesssim 1$. Therefore, plugging the estimates \eqref{Estimate of b in exp norm} and \eqref{Estimate of g omega in LlogL norm} into \eqref{Use of generalized Holder in commutator} and using H\"older's inequality we get
\begin{align*}
    & \|\chi_{B_m} \left( b-b_{Q_m} \right) T(f \chi_{(4Q_m)^c})\|_{L^p(\omega)} \\
    &\lesssim \sup_{\|g\|_{L^{p'}(\omega)} \leq 1} [\omega]_{A_{p}^{\rho, \theta}} \|b\|_{BMO_{\theta}({\rho})} \int_{\mathbb{R}^n} \chi_{Q_m}(x) \mathcal{M}_{\rho, \frac{\theta}{p-1}}f(x) \mathcal{M}_{\omega}^{\mathscr{D}}(|g|^s)^{1/s}(x) \omega(x) \, dx \\
    &\lesssim \sup_{\|g\|_{L^{p'}(\omega)} \leq 1} [\omega]_{A_{p}^{\rho, \theta}} \|b\|_{BMO_{\theta}({\rho})} \|\chi_{Q_m} \mathcal{M}_{\rho, \frac{\theta}{p-1}}f\|_{L^p(\omega)} \|\mathcal{M}_{\omega}^{\mathscr{D}}(|g|^s)^{1/s}\|_{L^{p'}(\omega)} \\
    &\lesssim [\omega]_{A_{p}^{\rho, \theta}} \|b\|_{BMO_{\theta}({\rho})} \|\chi_{Q_m} \mathcal{M}_{\rho, \frac{\theta}{p-1}}f\|_{L^p(\omega)} ,
\end{align*}
where in the last line we have used Proposition \ref{Prop: Boundedness of dyadic maximal functions} for $1<s<p'$.

Now using the above estimate and applying Proposition \ref{Prop: Maximal function new weight with r power} yields
\begin{align*}
    J_{21} &= \sum_{m \in \mathbb{N}} \|\chi_{B_m} \left( b-b_{Q_m} \right) T(f \chi_{(4Q_m)^c})\|_{L^p(\omega)}^p \\
    &\lesssim [\omega]_{A_{p}^{\rho, \theta}}^p \|b\|_{BMO_{\theta}({\rho})}^p \sum_{m \in \mathbb{N}} \|\chi_{Q_m} \mathcal{M}_{\rho, \frac{\theta}{p-1}}f\|_{L^p(\omega)}^p \\
    &\lesssim [\omega]_{A_{p}^{\rho, \theta}}^p \|b\|_{BMO_{\theta}({\rho})}^p \| \mathcal{M}_{\rho, \frac{\theta}{p-1}}f\|_{L^p(\omega)}^p \\
    &\lesssim [\omega]_{A_{p}^{\rho, \theta}}^p [\omega]_{A_{p}^{\rho, \theta}}^{\frac{p}{p-1}} \|b\|_{BMO_{\theta}({\rho})}^p \| f\|_{L^p(\omega)}^p \\
    &\lesssim [\omega]_{A_{p}^{\rho, \theta}}^{2 p \max\left\{ \frac{1}{p - 1}, 1 \right\}} \|b\|_{BMO_{\theta}(\rho)}^p \|f\|_{L^p(\omega)}^p .
\end{align*}

\subsubsection{Estimate of \texorpdfstring{$J_{22}$}{}}
Applying Lemma \ref{Lemma: Difference kernel estimate with sum over j} and taking $N=2\theta/(p-1) +\theta'+1$ with $\theta'=\theta (l_0+1)$ we have
\begin{align*}
    & |\chi_{B_m}(x) T((b-b_{Q_m}) f \chi_{(4Q_m)^c})(x)| \\
    &\lesssim \int_{|x_m-y|>4\rho(x_m)} \sum_{j=0}^{\infty} |K_j(x, y)| |b(y)-b_{Q_m}| |f(y)| \ dy \\
    &\lesssim \int_{|x_m-y|>2\rho(x_m)} \left(1+\frac{|x_m-y|}{\rho(x_m)} \right)^{-(\frac{2\theta}{p-1}+\theta'+1)} |x_m-y|^{-n} |b(y)-b_{Q_m}| |f(y)| \ dy \\
    &\lesssim \sum_{k=2}^{\infty} \int_{2^k \rho(x_m) < |x_m-y| \leq 2^{k+1} \rho(x_m)} \frac{1}{(1+\frac{2^k \rho(x_m)}{\rho(x_m)})^{\frac{2\theta}{p-1}+\theta'+1} (2^{k+1} \rho(x_m))^{n}} |b(y)-b_{Q_m}| |f(y)| \ dy \\
    &\lesssim \sum_{k=2}^{\infty} \frac{1}{(1+\frac{2^k \rho(x_m)}{\rho(x_m)})^{\frac{2\theta}{p-1}+\theta'+1} |2^{k+1} Q_m|} \int_{2^{k+1} Q_m} |b(y)-b_{Q_m}| |f(y)| \ dy .
\end{align*}
For each $k \geq 2$, writing $b(y)-b_{Q_m} = b(y)-b_{2^{k+1} Q_m} + b_{2^{k+1} Q_m} -b_{Q_m}$ and using generalized H\"older's inequality (\ref{Inequality: generalized Holders inequality}), Lemma \ref{Lemma: Complementary exponential function with BMO function}, the right hand side of the above estimate is dominated by
\begin{align*}
    & \sum_{k=2}^{\infty} \frac{1}{(1+\frac{2^k \rho(x_m)}{\rho(x_m)})^{\frac{2\theta}{p-1} +\theta'+1} |2^{k+1} Q_m|} \int_{2^{k+1} Q_m} |b(y)-b_{2^{k+1} Q_m}| |f(y)| \ dy \\
    & \hspace{3cm} + \sum_{k=2}^{\infty} \frac{1}{(1+\frac{2^k \rho(x_m)}{\rho(x_m)})^{\frac{2\theta}{p-1}+\theta'+1} |2^{k+1} Q_m|} |b_{2^{k+1} Q_m}-b_{Q_m}| \int_{2^{k+1} Q_m} |f(y)| \ dy \\
    &\lesssim \sum_{k=2}^{\infty} \frac{1}{(1+\frac{2^k \rho(x_m)}{\rho(x_m)})^{\frac{2\theta}{p-1}+ \theta'+1}} \|b-b_{2^{k+1} Q_m}\|_{\exp{\mathcal{L}}, 2^{k+1} Q_m} \|f\|_{\mathcal{L}\log \mathcal{L}, 2^{k+1} Q_m} \\
    & \hspace{5cm} + \sum_{k=2}^{\infty} \frac{k \|b\|_{BMO_{\theta}(\rho)}}{(1+\frac{2^k \rho(x_m)}{\rho(x_m)})^{\frac{2\theta}{p-1}+1} |2^{k+1} Q_m|} \int_{2^{k+1} Q_m} |f(y)| \ dy  \\
    &\lesssim \sum_{k=2}^{\infty} \frac{2^{-k}}{(1+\frac{2^k \rho(x_m)}{\rho(x_m)})^{\theta (l_0+1)}} \|b-b_{2^{k+1} Q_m}\|_{\exp{\mathcal{L}}, 2^{k+1} Q_m} \frac{1}{(1+\frac{2^k \rho(x_m)}{\rho(x_m)})^{\frac{2\theta}{p-1}}}\|f\|_{\mathcal{L}\log \mathcal{L}, 2^{k+1} Q_m} \\
    & \hspace{5cm} + \sum_{k=2}^{\infty} \frac{k 2^{-k} \|b\|_{BMO_{\theta}(\rho)}}{(1+\frac{2^k \rho(x_m)}{\rho(x_m)})^{\frac{2\theta}{p-1}} |2^{k+1} Q_m|} \int_{2^{k+1} Q_m} |f(y)| \ dy \\
    &\lesssim \sum_{k=2}^{\infty} 2^{-k} \|b\|_{BMO_{\theta}(\rho)} \mathcal{M}_{\mathcal{L} \log \mathcal{L}, \rho, \frac{2\theta}{p-1}}(f)(x) + \sum_{k=2}^{\infty} k 2^{-k} \|b\|_{BMO_{\theta}(\rho)} \mathcal{M}_{\rho, \frac{2\theta}{p-1}}(f)(x) \\
    &\lesssim \|b\|_{BMO_{\theta}(\rho)} \mathcal{M}_{\mathcal{L} \log \mathcal{L}, \rho, \frac{2\theta}{p-1}}(f)(x) ,
\end{align*}
where in the second inequality, we have used the fact $|b_{Q}-b_{Q_k}| \lesssim k \Psi_{\theta}(Q_k) \|b\|_{BMO_{\theta}(\rho)}$ and in the last line we have used $\mathcal{M}_{\rho, \theta}(f)(x) \leq \mathcal{M}_{\mathcal{L} \log \mathcal{L}, \rho, \theta}(f)(x) $ (see \cite[p. 2526]{Tang_Weighted_Schrodinger_Forum_Math_2015}).

Finally, using the above estimate, Lemma \ref{Lemma: Relation between variant maximal function and L log L maximal function} and Proposition \ref{Prop: Maximal function new weight with r power}, we get
\begin{align*}
    J_{22} &\lesssim \sum_{m \in \mathbb{N}} \|\chi_{B_m} T((b-b_{Q_m}) f \chi_{(4Q_m)^c})\|_{L^p(\omega)}^p \\
    &\lesssim  \|b\|_{BMO_{\theta}(\rho)}^p \sum_{m \in \mathbb{N}} \|\chi_{B_m} \mathcal{M}_{\mathcal{L} \log \mathcal{L}, \rho, \frac{2\theta}{p-1}}f\|_{L^p(\omega)}^p \\
    &\lesssim  \|b\|_{BMO_{\theta}(\rho)}^p \|\mathcal{M}_{\mathcal{L} \log \mathcal{L}, \rho, \frac{2\theta}{p-1}}f\|_{L^p(\omega)}^p \\
    &\lesssim  \|b\|_{BMO_{\theta}(\rho)}^p \|\mathcal{M}_{\rho, \frac{\theta}{p-1}} \mathcal{M}_{\rho, \frac{\theta}{p-1}}f\|_{L^p(\omega)}^p \\
    &\lesssim  [\omega]_{A_{p}^{\rho, \theta}}^{2 p \max\left\{ \frac{1}{p - 1}, 1 \right\}} \|b\|_{BMO_{\theta}(\rho)}^p \|f\|_{L^p(\omega)}^p .
\end{align*}
This completes the proof of Theorem \ref{Theorem: Weighted L^p boundedness of pseudo commutator}.
\end{proof}

\section{Compactness of commutators for the Schr\"odinger pseudo-multipliers}
\label{Section: Compactness of commutators}

In this section, we prove the compactness of commutators for the Schr\"odinger pseudo-multiplier, that is, Theorem \ref{Theorem: Weighted compactness of pseudo commutator}. To prove our theorem, we need the following Kolmogorov-Riesz compactness theorem.
\begin{theorem}\cite[Theorem 1.1]{Xue_Yabuta_Yan_Kolmogorov_Riesz_theorem}
\label{Theorem: Characterizations of compactness}
Let $1 < p < \infty$ and $\omega \in A_p^{\rho, \theta}$. The subset $\mathcal{F}$ in $L^p(\omega)$ is totally bounded if $\mathcal{F}$ satisfies the following conditions:
\begin{enumerate}
       \item  Norm boundedness uniformly: $\displaystyle{ \sup_{f \in \mathcal{F}} \|f\|_{L^p (\omega)} < \infty}$;
       \item  Control uniformly away from origin: $\displaystyle{\lim_{A \rightarrow \infty} \|\chi_{\mathcal{F}_A} f \|_{L^p (\omega)}=0}$, uniformly in $f\in \mathcal{F}$, where $\mathcal{F}_A = \{x \in \mathbb{R}^n : |x| > A\}$.
       \item Equicontinuous uniformly: $\displaystyle{\lim_{t \rightarrow 0} \| f(\cdot + t) - f(\cdot)\|_{L^p (\omega)} = 0   }$ uniformly in $f\in \mathcal{F}$.
\end{enumerate}
\end{theorem}

Now, we will discuss the proof of Theorem \ref{Theorem: Weighted compactness of pseudo commutator}.

\begin{proof}[Proof of Theorem \ref{Theorem: Weighted compactness of pseudo commutator}]
Since $b \in CMO_{\theta}(\rho) \subseteq BMO_{\theta}(\rho)$, then applying Theorem \ref{Theorem: Weighted L^p boundedness of pseudo commutator} for $T=\sigma(x, \sqrt{\mathcal{L}})$ we get $[b, T]$ is bounded on $L^p(\omega)$ for $1<p<\infty$ and $\omega \in A_p^{\rho, \theta}$. Therefore, to prove that $[b,T]$ is compact, for any bounded set $F \subseteq L^p(\omega)$, it suffices to show that $\{[b, T]f : f \in F \}$ is totally bounded in $L^p(\omega)$. By definition, it follows that since $b \in CMO_{\theta}(\rho)$, for $\epsilon>0$ there exists a sequence of functions $b_n \in C_c^{\infty}(\mathbb{R}^n)$ such that $\|b-b_n\|_{BMO_{\theta}(\rho)} < \epsilon$. Then using Theorem \ref{Theorem: Weighted L^p boundedness of pseudo commutator}, we get
\begin{align*}
    \|([b,T]-[b_n, T])f \|_{L^p(\omega)} &\leq \|[b-b_n, T]f\|_{L^p(\omega)} \lesssim \|b-b_n\|_{BMO_{\theta}(\rho)} \|f\|_{L^p(\omega)} \lesssim \epsilon .
\end{align*}
So it is enough to verify that $[b, T]$ is compact on $L^p(\omega)$ for any $b \in C_c^{\infty}(\mathbb{R}^n)$. First, using \eqref{eq:operator-decom}, we write 
$$T=\sum_{j\geq 0} T_{j}, \quad \textit{where} \quad T_{j}= \sigma_{j}(x, \sqrt{\mathcal{L}}).$$

Let us choose $\phi \in C^{\infty}([0, \infty))$ so that $0 \leq \phi \leq 1$, $\phi(x) =1$ in $[0,1]$ and $=0$ outside $[0,2]$. Then for each $j\geq 0$ and for any $\gamma>0$, let us define 
\begin{align*}
    K_{j, \gamma}(x,y) &= K_j(x,y) [1-\phi(\gamma^{-1}|x-y|)] .
\end{align*}
Observe that for each $j\geq 0$, we have
\begin{align}
\label{Estimate: Dilated kernel estimate}
    |K_{j, \gamma}(x,y)| &\leq |K_j(x,y)| ,\\
\label{Estimate: Gradient estimate for dilated kernel}
   \text{and} \quad |\nabla K_{j, \gamma}(x,y)| &\lesssim |\nabla K_j(x,y)| + \frac{1}{|x-y|} |K_j(x,y)| .
\end{align}

The estimate \eqref{Estimate: Dilated kernel estimate} follows from the definition of $K_{j, \gamma}$ and the estimate \eqref{Estimate: Gradient estimate for dilated kernel} follows from the following estimate. For $|h|<\gamma$
\begin{align*}
    |K_{j, \gamma}(x+h,y)-K_{j, \gamma}(x,y)| &\leq 2|K_j(x+h,y)-K_j(x,y)| + C \frac{|h|}{|x-y|} |K_j(x,y)|,
\end{align*}
where $C$ is a positive constant depending on $\phi$. The above estimate can be easily proved using the properties of $\phi$.

For $\gamma>0$, we define 
\begin{align*}
    K_{\gamma}(x,y) &= \sum_{j\geq 0}^{\infty} K_{j, \gamma}(x,y), \quad x, y \in \mathbb{R}^n ,\quad
   \text{and} \quad \quad T_{\gamma}f(x) = \int_{\mathbb{R}^d} K_{\gamma}(x,y) f(y) \ dy .
\end{align*}
From the estimates \eqref{Estimate: Dilated kernel estimate} and \eqref{Estimate: Gradient estimate for dilated kernel}, it is clear that the behaviours of the kernel $K_{j, \gamma}(x,y)$ and its gradient $\nabla K_{j, \gamma}(x,y)$ are same as $K_{j}(x,y)$ and $\nabla K_{j}(x,y)$ respectively. 

Therefore, using a similar method as we used to prove \Cref{Theorem: Weighted L^p boundedness of pseudo commutator}, we get 
\begin{align}
\label{Estimate: L^p boundedness of dilated operator}
    \|T_{\gamma}f\|_{L^p(\omega)} &\lesssim \|f\|_{L^p(\omega)},
\end{align}
for all $1<p<\infty$ and $w \in A_{p}^{\rho,\theta}$.

Now, for any $b \in C_c^{\infty}(\mathbb{R}^d)$, using $|b(x)-b(y)| \leq C |x-y|$, taking $N=\theta/(p-1)$ and using \eqref{Kernel estimate for all x and y} we obtain
\begin{align}
\label{Estimate: Difference of commutator and dilated commutator}
    |[b, T]f(x)-[b, T_{\gamma}]f(x)| &\leq \int_{|x-y|\leq 2\gamma} \sum_{j\geq 0} |K_j(x,y)| |\phi(\gamma^{-1}|x-y|)||b(x)-b(y)||f(y)| \, dy \\
    &\nonumber \lesssim \int_{|x-y|\leq 2\gamma} \frac{1}{(1+ \frac{|x-y|}{\rho(x)} )^{\theta/(p-1)}} \frac{|f(y)|}{|x-y|^{n-1}} \, dy \\
    &\nonumber \lesssim \sum_{j=-\infty}^{0} \int_{2^j \gamma<|x-y|\leq 2^{j+1} \gamma} \frac{1}{(1+ \frac{2^j \gamma}{\rho(x)} )^{\theta/(p-1)}} \frac{|f(y)|}{(2^j \gamma)^{n-1}} \, dy \\
    &\nonumber \lesssim \mathcal{M}_{\rho, \frac{\theta}{p-1}}f(x) \sum_{j=-\infty}^{0} 2^j \gamma \lesssim \gamma \mathcal{M}_{\rho, \frac{\theta}{p-1}}f(x) .
\end{align}
Therefore, for $\omega \in A_p^{\rho, \theta}$ and using Proposition \ref{Prop: Maximal function new weight with r power}  we get
\begin{align}\label{Estimate: Difference of commutator and dilated commutator-Lp}
    \|[b, T]f(x)-[b, T_{\gamma}]f\|_{L^p(\omega)} &\lesssim \gamma \|f\|_{L^p(\omega)} .
\end{align}
This further implies
\begin{align*}
    \lim_{\gamma \to 0}\|[b, T]f(x)-[b, T_{\gamma}]f\|_{L^p(\omega)}=0 .
\end{align*}
As the norm limit of a compact operator is compact, therefore it is enough to show that $[b, T_{\gamma}]$ is compact on $L^p(\omega)$ for any $b \in C_c^{\infty}(\mathbb{R}^n)$, for sufficiently small $\gamma >0$. Now for arbitrary bounded set $F$ in $L^p(\omega)$, let us set $\mathcal{F} = \{[b, T_{\gamma}]f : f \in F \}$. Therefore, by Theorem \ref{Theorem: Characterizations of compactness}, it suffices to verify that $\mathcal{F}$ satisfies all three conditions of Theorem \ref{Theorem: Characterizations of compactness}.

\subsection*{Condition \texorpdfstring{$(1)$}{}:}
Note that Theorem \ref{Theorem: Weighted L^p boundedness of pseudo commutator} and (\ref{Estimate: Difference of commutator and dilated commutator-Lp}) imply $[b, T_{\gamma}]$ is bounded on $L^p(\omega)$ for all $1<p<\infty$ and $w \in A_{p}^{\rho,\theta}$. Therefore, we conclude
\begin{align*}
    \sup_{f \in F} \|[b, T_{\gamma}]f\|_{L^p(\omega)} &\leq C \sup_{f \in F} \|f\|_{L^p(\omega)} \leq C .
\end{align*}
Hence, the first condition of Theorem \ref{Theorem: Characterizations of compactness} is satisfied.

\subsection*{Condition \texorpdfstring{$(2)$}{}:}
Since $C_c^{\infty}(\mathbb{R}^n)$ is dense in $L^p(\omega)$, for $f \in F$ and $\epsilon>0$, there exists $g \in C_c^{\infty}(\mathbb{R}^n)$ such that $\|f-g\|_{L^p(\omega)}<\epsilon$. Let us take $\mathcal{F}_{A}$ as defined in \Cref{Theorem: Characterizations of compactness}. Then, using boundedness of $[b, T_{\gamma}]$ on $L^p(\omega)$, we get
\begin{align}
\label{In condition 2 reduction to compactly function}
    \|[b, T_{\gamma}]f \chi_{\mathcal{F}_A} \|_{L^p(\omega)} &\leq \|[b, T_{\gamma}](f-g) \chi_{\mathcal{F}_A} \|_{L^p(\omega)} + \|[b, T_{\gamma}]g \chi_{\mathcal{F}_A} \|_{L^p(\omega)} \\
    &\nonumber \leq C \epsilon + C \|[b, T_{\gamma}]g \chi_{\mathcal{F}_A} \|_{L^p(\omega)} .
\end{align}
As $b \in C_c^{\infty}(\mathbb{R}^n)$, we can choose $R>0$ large enough so that $\supp{b} \subseteq B(0, R)$. For $\kappa>2$, let us choose $A>0$ so that $A \geq \kappa R$. This implies $b(x)=0$ if $|x|>A>\kappa R$. Using this observation we write
\begin{align}\label{eq:commu-gamma-1}
    & \left(\int_{|x|>\kappa R} |[b, T_{\gamma}]g(x)|^p \omega(x) \ dx \right)^{1/p} \\
    \nonumber &\leq \left[\int_{|x|>\kappa R} \left( \int_{|y|<R} \sum_{j\geq 0} |K_j(x,y)| |b(y)| |g(y)| \ dy \right)^p \omega(x) \ dx \right]^{1/p} \\
    \nonumber&\lesssim \sum_{k=0}^{\infty} \left[\int_{2^k \kappa R< |x|\leq 2^{k+1} \kappa R} \left( \int_{|y|<R} \sum_{j\geq 0} |K_j(x,y)| |b(y)| |g(y)| \ dy \right)^p \omega(x) \ dx \right]^{1/p}.
\end{align}
Since $|x|>2^k \kappa R$ and $|y|<R$, we have $|x-y|>(1-1/\kappa)2^{k}\kappa R$ with $\kappa>2$. Therefore, using \eqref{Kernel estimate for all x and y}, Holder's inequality and using Lemma \ref{Lemma: Equivalence of two potential}, for each $k \geq 0$, the first integral of the right side of the above inequality is dominated by
\begin{align}\label{eq:commu-gamma-2}
    \nonumber \int_{|y|<R} \sum_{j\geq 0}|K_j(x,y)| |b(y)| |g(y)| \, dy &\lesssim \int_{|y|<R} \frac{1}{(1+ \frac{|x-y|}{\rho(x)} )^{N}} \frac{1}{|x-y|^{n}} |b(y)| |g(y)| \, dy \\
   &\lesssim \frac{\|b\|_{L^{\infty}}}{(1+ \frac{(1-\frac{1}{\kappa})2^{k}\kappa R}{\rho(x)})^{N}} \frac{1}{((1-\frac{1}{\kappa})2^{k}\kappa R)^{n}}  \int_{|y|<R} |g(y)| \, dy \\
    \nonumber &\lesssim \frac{\|b\|_{L^{\infty}}}{(1+ \frac{(1-\frac{1}{\kappa})2^{k}\kappa R}{\rho(0)})^{\frac{N}{l_0+1}}} \frac{\|g\|_{L^p(\omega)}}{(1-\frac{1}{\kappa})^n (2^{k}\kappa R)^{n}} \left( \int_{|y|<R} \omega^{-\frac{1}{p-1}} \right)^{1-\frac{1}{p}} .
\end{align}
Putting the estimate \eqref{eq:commu-gamma-2} into \eqref{eq:commu-gamma-1}, we get
\begin{align}\label{eq:commu-gamma-3}
    & \left(\int_{|x|>\kappa R} |[b, T_{\gamma}]g(x)|^p \omega(x) \ dx \right)^{1/p} \lesssim \|b\|_{L^{\infty}} \|g\|_{L^p(\omega)} \sum_{k=0}^{\infty} \frac{1}{(1+ \frac{(1-\frac{1}{\kappa})2^k \kappa R}{\rho(0)})^{\frac{N}{l_0+1}}} \frac{1}{(1-\frac{1}{\kappa})^n (2^k \kappa R)^{n}} \\
    \nonumber& \hspace{3cm} \times \left(\int_{B(0, 2^{k+1} \kappa R)} \omega(x) \, dx \right)^{1/p} \left( \int_{B(0, R)} \omega^{-1/(p-1)}(y) \, dy \right)^{1-\frac{1}{p}} .
\end{align}
Taking $Q= Q(0, 2 \sqrt{n} 2^{k+1} \kappa R)$ and $E=B(0,R)$ in \cite[Lemma 2.2, (iii)]{Tang_Weighted_Schrodinger_Forum_Math_2015} yields
\begin{align}\label{eq:doub-weight}
    \omega(5Q) &\lesssim \left( \frac{\Psi_{\theta}(Q)|Q|}{|E|} \right)^p \omega(E) \lesssim \omega(E) (1+2^{k+1}\kappa R/\rho(0))^{\theta p} (2^{k+1}\kappa)^{n p}.
\end{align}
Putting the estimate \eqref{eq:doub-weight} in \eqref{eq:commu-gamma-3} and using the fact $B(0, 2^{k+1} \kappa R) \subset Q(0, 2 \sqrt{n} 2^{k+1} \kappa R)$ we get 
\begin{align}
\label{Making of weight in compactness}
    & \left(\int_{|x|>\kappa R} |[b, T_{\gamma}]g(x)|^p \omega(x) \, dx \right)^{1/p} \\
    &\nonumber \lesssim \|b\|_{L^{\infty}} \|g\|_{L^p(\omega)} \sum_{k=0}^{\infty} \frac{(1+2^{k+1}\kappa R/\rho(0))^{\theta}}{(1+ \frac{(1-1/\kappa)2^k \kappa R}{\rho(0)})^{N/(l_0+1)}} \frac{(2^{k+1}\kappa)^{n}}{(1-1/\kappa)^n (2^k \kappa R)^{n}} \\
    &\nonumber \hspace{4cm} \times \left(\int_{E} \omega(x) \ dx \right)^{1/p} \left( \int_{E} \omega^{-1/(p-1)}(y) \, dy \right)^{1-1/p} \\
    &\nonumber \lesssim \frac{\|b\|_{L^{\infty}} \|g\|_{L^p(\omega)}}{(1-1/\kappa)^{n+N/(l_0+1)}} \sum_{k=0}^{\infty} \frac{(1+2^{k+1}\kappa R/\rho(0))^{\theta}}{(1+ 2^{k}\kappa R/\rho(0))^{N/(l_0+1)}} \left(1+\frac{R}{\rho(0)} \right)^{\theta} \\
    &\nonumber \hspace{2cm} \times \left(\frac{1}{\Psi_{\theta}(E)|E|}\int_{E} \omega(x) \, dx \right)^{1/p} \left( \frac{1}{\Psi_{\theta}(E)|E|} \int_{E} \omega^{-1/(p-1)}(y) \, dy \right)^{1-\frac{1}{p}} \\
    &\nonumber \lesssim \frac{\|b\|_{L^{\infty}} \|g\|_{L^p(\omega)} \|\omega\|_{A_p^{\theta, \rho}}^{1/p}}{(1-1/\kappa)^{n+N/(l_0+1)}} \sum_{k=0}^{\infty} \frac{(1+2^{k}\kappa R/\rho(0))^{2\theta}}{(1+ 2^{k}\kappa R/\rho(0))^{N/(l_0+1)}} .
\end{align}
Now, observe that the series on the right-hand side of the above inequality is convergent. Indeed, if $R>\rho(0)$, then we have $1/(1+2^k \kappa R/\rho(0)) \leq 1/(2^k \kappa)$. Hence, for $N>2\theta(l_0+1)$ we get
\begin{align*}
    \sum_{k=0}^{\infty} \frac{(1+2^{k}\kappa R/\rho(0))^{2\theta}}{(1+ 2^{k}\kappa R/\rho(0))^{N/(l_0+1)}} &\leq \sum_{k=0}^{\infty} \frac{1}{(2^k \kappa)^{N/(l_0+1)-2\theta}} \lesssim \frac{1}{\kappa^{N/(l_0+1)-2\theta}} .
\end{align*}
On the other hand for $R\leq \rho(0)$, there exists $l>0$ such that $2^l R> \rho(0)$. Therefore, as in the previous case, we have $1/(1+2^j \kappa R/\rho(0)) \leq 2^{l}/(2^j \kappa)$. So, for $N>2\theta(l_0+1)$ gives
\begin{align*}
    \sum_{k=0}^{\infty} \frac{(1+2^{k}\kappa R/\rho(0))^{2\theta}}{(1+ 2^{k}\kappa R/\rho(0))^{N/(l_0+1)}} &\leq \sum_{k=0}^{\infty} \frac{2^{l(N/(l_0+1)-2\theta)}}{(2^k \kappa)^{N/(l_0+1)-2\theta}} \lesssim \frac{2^{l(N/(l_0+1)-2\theta)}}{\kappa^{N/(l_0+1)-2\theta}} .
\end{align*}
Using the above estimates in \eqref{Making of weight in compactness} yields
\begin{align*}
    \left(\int_{|x|>\kappa R} |[b, T_{\gamma}]g(x)|^p \omega(x) \ dx \right)^{1/p} &\lesssim \frac{\|b\|_{L^{\infty}} \|g\|_{L^p(\omega)} \|\omega\|_{A_p^{\theta, \rho}}^{1/p}}{(1-1/\kappa)^{n+N/(l_0+1)} \kappa^{N/(l_0+1)-2\theta}} \max\{2^{l(N/(l_0+1)-2\theta)}, 1\}.
\end{align*}
Hence, for $p>1$ and choosing $N>2\theta(l_0+1)$, for $g \in C_c^{\infty}(\mathbb{R}^n)$, there holds
\begin{align*}
    \lim_{A \to \infty} \|[b, T_{\gamma}]g \chi_{\mathcal{F}_A} \|_{L^p(\omega)} & \leq \lim_{\kappa \to \infty} \left(\int_{|x|>\kappa R} |[b, T_{\gamma}]g(x)|^p \omega(x) \ dx \right)^{1/p} =0 .
\end{align*}
Therefore, the second condition of Theorem \ref{Theorem: Characterizations of compactness} is satisfied.

\subsection*{Condition \texorpdfstring{$(3)$}{}:}
We want to show that for any $\epsilon>0$, if $|t|$ is sufficiently small then,
\begin{align*}
    \|[b, T_{\gamma}]f(\cdot+t)-[b, T_{\gamma}]f(\cdot)\|_{L^p(\omega)} \leq C \epsilon \quad \quad \forall f \in F.
\end{align*}
Therefore similarly as in \eqref{In condition 2 reduction to compactly function}, it is enough to verify that, for $|t|$ is sufficiently small,
\begin{align*}
    \|[b, T_{\gamma}]g(\cdot+t)-[b, T_{\gamma}]g(\cdot)\|_{L^p(\omega)} \leq C \epsilon \quad \quad \forall g \in C_c^{\infty}(\mathbb{R}^n) \subseteq F.
\end{align*}
Let us take $\gamma \in (0,1)$ and choose $|t|<\gamma/8$. We write
\begin{align*}
    [b, T_{\gamma}]g(x+t)-[b, T_{\gamma}]g(x) &= \int_{|x-y|>\gamma/2} K_{\gamma}(x,y)(b(x+t)-b(x))g(y) \, dy \\
    &+ \int_{|x-y|>\gamma/2} (K_{\gamma}(x+t,y)-K_{\gamma}(x,y))(b(x+t)-b(y)) g(y) \, dy \\
    &- \int_{|x-y|\leq \gamma/2} K_{\gamma}(x,y)(b(x)-b(y)) g(y) \, dy \\
    &+ \int_{|x-y|\leq \gamma/2} K_{\gamma}(x+t,y) (b(x+t)-b(y)) g(y) \, dy \\
    &=: J_1 + J_2 + J_3 + J_4 .
\end{align*}
First, we estimate $J_1$. As $b \in C_c^{\infty}(\mathbb{R}^n)$, we have $|b(x+t)-b(x)| \leq C |t| \|\nabla b\|_{L^{\infty}} \leq C |t|$. Therefore, using (\ref{Estimate: L^p boundedness of dilated operator}) we get
\begin{align}
\label{Estimate of J1 case}
    \|J_1\|_{L^p(\omega)} & \leq C |t| \|T_{\gamma}g\|_{L^p(\omega)} \leq C |t| \|g\|_{L^p(\omega)} \leq C |t| .
\end{align}
Note that for $|x-y| \leq \gamma/2$ and $|t|<\gamma/8$, we have $|x+t-y|<5\gamma/8$. This implies $\phi(\gamma^{-1}|x+t-y|) = \phi(\gamma^{-1}|x-y|) =1$. This will further implies $K_{\gamma}(x+t, y)=K_{\gamma}(x,y)=0$. So, we have $J_3 = J_4 =0$. Therefore, it only remains to estimate $J_2$. 

To estimate $J_2$, we write
\begin{align*}
    K_{\gamma}(x+t, y)-K_{\gamma}(x,y) &=\sum_{j=0}^{\infty} \left( K_{j, \gamma}(x+t, y)-K_{j, \gamma}(x,y) \right)\\
    &= \sum_{j=0}^{\infty} [K_j(x+t,y)-K_j(x,y)](1+\phi(\gamma^{-1}|x+t-y|)) \\
    & \hspace{2cm} + \sum_{j=0}^{\infty} K_j(x,y)[\phi(\gamma^{-1}|x+t-y|)-\phi(\gamma^{-1}|x-y|)] .
\end{align*}
Using the observation above, $J_2$ is dominated by
\begin{align*}
    |J_2| &\leq \int_{|x-y|>\gamma/2} \sum_{j=0}^{\infty} |K_j(x+t,y)-K_j(x,y)||1+\phi(\gamma^{-1}|x+t-y|)| |b(x+t)-b(y)| |g(y)|\,  dy \\
    &+ \int_{|x-y|>\gamma/2} \sum_{j=0}^{\infty}|K_j(x,y)| |\phi(\gamma^{-1}|x+t-y|)-\phi(\gamma^{-1}|x-y|)| |b(x+t)-b(y)| |g(y)| \, dy \\
    &=: J_{21} + J_{22} .
\end{align*}
Let us start with the estimate of $J_{21}$. Note that $|x-y|>\gamma/2$ and $|t|<\gamma/8$ imply $|x-y|>4|t|$ and $|x+t-y| \leq |x-y|+|t| \leq 2|x-y|$. Taking $0<\alpha<1$, we write
\begin{align*}
    |b(x+t)-b(y)| &= |b(x+t)-b(y)|^{1-\alpha} |b(x+t)-b(y)|^{\alpha} \\
    & \lesssim \|b\|_{L^{\infty}}^{1-\alpha} \|\nabla b\|_{L^{\infty}}^{\alpha} |x+t-y|^{\alpha} \lesssim |x-y|^{\alpha} .
\end{align*}
Using the above fact and applying Lemma \ref{Lemma: Difference kernel estimate with sum over j} with $Q=Q(x, \frac{3|t|}{2})$ yields
\begin{align}
\label{Estimate of J21 case}
    J_{21} &\lesssim \int_{|x-y|>4|t|} \frac{|t|}{(1+ \frac{|x-y|}{\rho(x)})^{\frac{\theta}{p-1}}} \frac{1}{|x-y|^{n+1}} |x-y|^{\alpha} |g(y)| \, dy \\
    &\nonumber \lesssim \sum_{k=2}^{\infty} \int_{2^{k}|t|\leq |x-y|<2^{k+1}|t|} \frac{|t|}{(1+ \frac{|x-y|}{\rho(x)})^{\frac{\theta}{p-1}}} \frac{1}{|x-y|^{n+1-\alpha}} |g(y)| \, dy \\
    &\nonumber \lesssim \sum_{k=2}^{\infty} \frac{|t|}{(1+ \frac{2^k |t|}{\rho(x)})^{\frac{\theta}{p-1}}} \frac{1}{(2^k |t|)^{n+1-\alpha}} \int_{B(x, 2^{k+1}|t|)} |g(y)| \, dy \\
    &\nonumber \lesssim |t|^{\alpha} \mathcal{M}_{\rho, N} g(x) \sum_{k=2}^{\infty} 2^{-(1-\alpha)k} \lesssim |t|^{\alpha} \mathcal{M}_{\rho, \frac{\theta}{p-1}}g(x) .
\end{align}

Finally, we estimate $J_{22}$. Note that if $|x-y|>3 \gamma$ and $|t|<\gamma/8$, then $|x+t-y|>|x-y|-|t|>2\gamma$. So, we have $\phi(\gamma^{-1}|x+t-y|)=\phi(\gamma^{-1}|x-y|)=0$. Also note that
\begin{align*}
    |\phi(\gamma^{-1}|x+t-y|)-\phi(\gamma^{-1}|x-y|)| \lesssim |t|/\gamma .
\end{align*}
Therefore using the observations and \eqref{Kernel estimate for all x and y}, with $N= \frac{\theta}{p-1}$ gives
\begin{align}
\label{Estimate of J22 case}
    J_{22} &\lesssim \frac{|t|}{\gamma} \sum_{j=0}^{\infty} \int_{\frac{\gamma}{2}<|x-y|\leq 3\gamma} \frac{1}{(1+ \frac{|x-y|}{\rho(x)})^{\frac{\theta}{p-1}}} \frac{1}{|x-y|^{n}} |b(x+t)-b(y)| |g(y)| \, dy \\
    &\nonumber \lesssim \|\nabla b\|_{L^{\infty}} \frac{|t|}{\gamma} \int_{\frac{\gamma}{2}<|x-y|\leq 3\gamma} \frac{1}{(1+ \frac{|x-y|}{\rho(x)})^{\frac{\theta}{p-1}}} \frac{|x+t-y|}{|x-y|^{n}} |g(y)| \, dy \\
    &\nonumber \lesssim |t| \frac{1}{(1+ \frac{3\gamma}{\rho(x)})^{\frac{\theta}{p-1}}} \frac{1}{(3\gamma)^n} \int_{|x-y|\leq 3\gamma} |g(y)| \, dy \lesssim |t| \mathcal{M}_{\rho, \frac{\theta}{p-1}} g(x) .
\end{align}
Now estimates \eqref{Estimate of J21 case} and \eqref{Estimate of J22 case} together with Proposition \ref{Prop: Maximal function new weight with r power} provide
\begin{align}
\label{Estimate of J2 case}
    \|J_2\|_{L^p(\omega)} &\lesssim (|t|^{\alpha}+|t|)\|\mathcal{M}_{\rho, \frac{\theta}{p-1}}g\|_{L^p(\omega)} \lesssim (|t|^{\alpha}+|t|) \|g\|_{L^p(\omega)} \leq C (|t|^{\alpha}+|t|) .
\end{align}
Therefore combining the estimates \eqref{Estimate of J1 case} and \eqref{Estimate of J2 case}, for $|t|$ sufficiently small we get
\begin{align*}
    \|[b, T_{\gamma}]g(\cdot+t)-[b, T_{\gamma}]g(\cdot)\|_{L^p(\omega)} \leq C \epsilon .
\end{align*}
This completes the proof of Theorem \ref{Theorem: Weighted compactness of pseudo commutator}.
\end{proof}

\section*{Funding Declaration} 
The authors are grateful to the anonymous referee for the careful reading of the manuscript and for valuable comments and suggestions. S.~Bagchi and J.~Singh gratefully acknowledge the support provided by the FIST grant
(SR/FST/MS-II/2019/51) awarded to IISER Kolkata by the Government of India. R.~Basak is supported by the National Science and Technology Council of Taiwan under research grant number 113-2811-M-003-014. J.~Singh is supported by the Prime Minister's Research Fellowship (PMRF), Ministry of Education, Government of India. M. N.~Vempati's research was supported, in part, by
the Louisiana Experimental Program to Stimulate Competitive Research (EPSCoR), funded
by the National Science Foundation and the Board of Regents Support Fund award number OIA-2437963.

\bibliographystyle{abbrv}
\bibliography{References}

@article {Beltran-Cladek,
    AUTHOR = {Beltran, David and Cladek, Laura},
     TITLE = {Sparse bounds for pseudodifferential operators},
   JOURNAL = {J. Anal. Math.},
  FJOURNAL = {Journal d'Analyse Math\'ematique},
    VOLUME = {140},
      YEAR = {2020},
    NUMBER = {1},
     PAGES = {89--116},
      ISSN = {0021-7670,1565-8538},
   MRCLASS = {42B25 (35R11 35S05)},
  MRNUMBER = {4094458},
MRREVIEWER = {Yaoming\ Niu},
       DOI = {10.1007/s11854-020-0083-x},
       URL = {https://doi.org/10.1007/s11854-020-0083-x},
}

@article {Miller_Pseudo_differential_1982,
    AUTHOR = {Miller, Nicholas},
     TITLE = {Weighted {S}obolev spaces and pseudodifferential operators
              with smooth symbols},
   JOURNAL = {Trans. Amer. Math. Soc.},
  FJOURNAL = {Transactions of the American Mathematical Society},
    VOLUME = {269},
      YEAR = {1982},
    NUMBER = {1},
     PAGES = {91--109},
      ISSN = {0002-9947,1088-6850},
   MRCLASS = {47G05 (35S05 46E35)},
  MRNUMBER = {637030},
MRREVIEWER = {B.\ Muckenhoupt},
       DOI = {10.2307/1998595},
       URL = {https://doi.org/10.2307/1998595},
}

@book {Stein-book,
    AUTHOR = {Stein, Elias M.},
     TITLE = {Harmonic analysis: real-variable methods, orthogonality, and
              oscillatory integrals},
    SERIES = {Princeton Mathematical Series},
    VOLUME = {43},
      NOTE = {With the assistance of Timothy S. Murphy,
              Monographs in Harmonic Analysis, III},
 PUBLISHER = {Princeton University Press, Princeton, NJ},
      YEAR = {1993},
     PAGES = {xiv+695},
      ISBN = {0-691-03216-5},
   MRCLASS = {42-02 (35Sxx 43-02 47G30)},
  MRNUMBER = {1232192},
MRREVIEWER = {Michael\ Cowling},
}

@article {Calderon_Vaillancourt_Pseudo_differential_1971,
    AUTHOR = {Calder\'on, Alberto-P. and Vaillancourt, R\'emi},
     TITLE = {On the boundedness of pseudo-differential operators},
   JOURNAL = {J. Math. Soc. Japan},
  FJOURNAL = {Journal of the Mathematical Society of Japan},
    VOLUME = {23},
      YEAR = {1971},
     PAGES = {374--378},
      ISSN = {0025-5645,1881-1167},
   MRCLASS = {47.70 (35.00)},
  MRNUMBER = {284872},
MRREVIEWER = {F.\ Cardoso},
       DOI = {10.2969/jmsj/02320374},
       URL = {https://doi.org/10.2969/jmsj/02320374},
}

@article {Duong_Yan_Zhang_Schrodinger_operator_2014,
    AUTHOR = {Duong, Xuan Thinh and Yan, Lixin and Zhang, Chao},
     TITLE = {On characterization of {P}oisson integrals of {S}chr\"{o}dinger
              operators with {BMO} traces},
   JOURNAL = {J. Funct. Anal.},
  FJOURNAL = {Journal of Functional Analysis},
    VOLUME = {266},
      YEAR = {2014},
    NUMBER = {4},
     PAGES = {2053--2085},
      ISSN = {0022-1236},
   MRCLASS = {42B37 (35J10 35L15 42B30)},
  MRNUMBER = {3150151},
MRREVIEWER = {Yasuo Komori-Furuya},
       DOI = {10.1016/j.jfa.2013.09.008},
       URL = {https://doi.org/10.1016/j.jfa.2013.09.008},
}

@article {Bagchi_Basak_Garg_Ghosh_Grushin_2023,
    AUTHOR = {Bagchi, Sayan and Basak, Riju and Garg, Rahul and Ghosh,
              Abhishek},
     TITLE = {Sparse bounds for pseudo-multipliers associated to {G}rushin
              operators, {I}},
   JOURNAL = {J. Fourier Anal. Appl.},
  FJOURNAL = {The Journal of Fourier Analysis and Applications},
    VOLUME = {29},
      YEAR = {2023},
    NUMBER = {3},
     PAGES = {Paper No. 27, 38},
      ISSN = {1069-5869,1531-5851},
   MRCLASS = {58J40 (42B25 43A85)},
  MRNUMBER = {4585471},
MRREVIEWER = {Jianwei-Urbain\ Yang},
       DOI = {10.1007/s00041-023-10010-w},
       URL = {https://doi.org/10.1007/s00041-023-10010-w},
}

@article {BBGG-2,
    AUTHOR = {Bagchi, Sayan and Basak, Riju and Garg, Rahul and Ghosh,
              Abhishek},
     TITLE = {Sparse bounds for pseudo-multipliers associated to {G}rushin
              operators, {II}},
   JOURNAL = {J. Geom. Anal.},
  FJOURNAL = {Journal of Geometric Analysis},
    VOLUME = {34},
      YEAR = {2024},
    NUMBER = {2},
     PAGES = {Paper No. 34, 49},
      ISSN = {1050-6926,1559-002X},
   MRCLASS = {58J40 (42B25 43A85)},
  MRNUMBER = {4675196},
       DOI = {10.1007/s12220-023-01473-w},
       URL = {https://doi.org/10.1007/s12220-023-01473-w},
}

@article {Tang_Wang_Zhu_Weighted_Schrodinger_2019,
    AUTHOR = {Tang, L. and Wang, J. and Zhu, H.},
     TITLE = {Weighted norm inequalities for area functions related to
              {S}chr\"{o}dinger operators},
   JOURNAL = {Izv. Nats. Akad. Nauk Armenii Mat.},
  FJOURNAL = {Natsional\cprime naya Akademiya Nauk Armenii. Izvestiya. Matematika},
    VOLUME = {54},
      YEAR = {2019},
    NUMBER = {1},
     PAGES = {40--59},
      ISSN = {0002-3043},
   MRCLASS = {42B25 (35J10)},
  MRNUMBER = {3935461},
}

@article {Tang_Weighted_Schrodinger_Forum_Math_2015,
    AUTHOR = {Tang, Lin},
     TITLE = {Weighted norm inequalities for {S}chr\"{o}dinger type operators},
   JOURNAL = {Forum Math.},
  FJOURNAL = {Forum Mathematicum},
    VOLUME = {27},
      YEAR = {2015},
    NUMBER = {4},
     PAGES = {2491--2532},
      ISSN = {0933-7741},
   MRCLASS = {42B25 (35J10 42B20)},
  MRNUMBER = {3365805},
MRREVIEWER = {Petru A. Cojuhari},
       DOI = {10.1515/forum-2013-0070},
       URL = {https://doi.org/10.1515/forum-2013-0070},
}

@article {Guo_Zhou_Compactness_Pseudo_differential_2019,
    AUTHOR = {Guo, QingDong and Zhou, Jiang},
     TITLE = {Compactness of commutators of pseudo-differential operators
              with smooth symbols on weighted {L}ebesgue spaces},
   JOURNAL = {J. Pseudo-Differ. Oper. Appl.},
  FJOURNAL = {Journal of Pseudo-Differential Operators and Applications},
    VOLUME = {10},
      YEAR = {2019},
    NUMBER = {3},
     PAGES = {557--569},
      ISSN = {1662-9981,1662-999X},
   MRCLASS = {47G30 (42B35 47B47)},
  MRNUMBER = {3990259},
MRREVIEWER = {Maria\ A.\ Ragusa},
       DOI = {10.1007/s11868-019-00303-4},
       URL = {https://doi.org/10.1007/s11868-019-00303-4},
}

@article {Tang_Pseudo_Commutators_2012,
    AUTHOR = {Tang, Lin},
     TITLE = {Weighted norm inequalities for pseudo-differential operators
              with smooth symbols and their commutators},
   JOURNAL = {J. Funct. Anal.},
  FJOURNAL = {Journal of Functional Analysis},
    VOLUME = {262},
      YEAR = {2012},
    NUMBER = {4},
     PAGES = {1603--1629},
      ISSN = {0022-1236,1096-0783},
   MRCLASS = {47B33 (35S05 42B25)},
  MRNUMBER = {2873852},
MRREVIEWER = {Jos\'e\ Bonet},
       DOI = {10.1016/j.jfa.2011.11.016},
       URL = {https://doi.org/10.1016/j.jfa.2011.11.016},
}

@article {Shen_Schrodinger_operator_certain_potential_1995,
    AUTHOR = {Shen, Zhong Wei},
     TITLE = {{$L^p$} estimates for {S}chr\"odinger operators with certain
              potentials},
   JOURNAL = {Ann. Inst. Fourier (Grenoble)},
  FJOURNAL = {Universit\'e{} de Grenoble. Annales de l'Institut Fourier},
    VOLUME = {45},
      YEAR = {1995},
    NUMBER = {2},
     PAGES = {513--546},
      ISSN = {0373-0956,1777-5310},
   MRCLASS = {35J10 (42B20)},
  MRNUMBER = {1343560},
MRREVIEWER = {V.\ S.\ Rabinovich},
       DOI = {10.5802/aif.1463},
       URL = {https://doi.org/10.5802/aif.1463},
}

@article {Duong_Ouhabaz_Sikora_Sharp_Multiplier_2002,
    AUTHOR = {Thinh Duong, Xuan and Ouhabaz, El Maati and Sikora, Adam},
     TITLE = {Plancherel-type estimates and sharp spectral multipliers},
   JOURNAL = {J. Funct. Anal.},
  FJOURNAL = {Journal of Functional Analysis},
    VOLUME = {196},
      YEAR = {2002},
    NUMBER = {2},
     PAGES = {443--485},
      ISSN = {0022-1236},
   MRCLASS = {43A85 (33C45 35P20 42B15 47B25)},
  MRNUMBER = {1943098},
MRREVIEWER = {Dachun Yang},
       DOI = {10.1016/S0022-1236(02)00009-5},
       URL = {https://doi.org/10.1016/S0022-1236(02)00009-5},
}

@book {Taylor_Pseudo_differential,
    AUTHOR = {Taylor, Michael E.},
     TITLE = {Pseudodifferential operators},
    SERIES = {Princeton Mathematical Series},
    VOLUME = {No. 34},
 PUBLISHER = {Princeton University Press, Princeton, NJ},
      YEAR = {1981},
     PAGES = {xi+452},
      ISBN = {0-691-08282-0},
   MRCLASS = {35Sxx (35-02 47G05 58G15 78A05)},
  MRNUMBER = {618463},
MRREVIEWER = {Vesselin\ M.\ Petkov},
}

@book {Ruzhansky_Turunen_Book,
    AUTHOR = {Ruzhansky, Michael and Turunen, Ville},
     TITLE = {Pseudo-differential operators and symmetries},
    SERIES = {Pseudo-Differential Operators. Theory and Applications},
    VOLUME = {2},
      NOTE = {Background analysis and advanced topics},
 PUBLISHER = {Birkh\"auser Verlag, Basel},
      YEAR = {2010},
     PAGES = {xiv+709},
      ISBN = {978-3-7643-8513-2},
   MRCLASS = {35-02 (35S05 43A77 43A80 43A85 47G30 58J40)},
  MRNUMBER = {2567604},
MRREVIEWER = {Fabio\ Nicola},
       DOI = {10.1007/978-3-7643-8514-9},
       URL = {https://doi.org/10.1007/978-3-7643-8514-9},
}

@article {Laptev_Fourier_Integral_Op,
    AUTHOR = {Laptev, A. A.},
     TITLE = {Spectral asymptotics of a class of {F}ourier integral
              operators},
   JOURNAL = {Trudy Moskov. Mat. Obshch.},
  FJOURNAL = {Trudy Moskovskogo Matematicheskogo Obshchestva},
    VOLUME = {43},
      YEAR = {1981},
     PAGES = {92--115},
      ISSN = {0134-8663},
   MRCLASS = {35P20 (35S05 47G05 58G15)},
  MRNUMBER = {651330},
}

@article {Michalowski_Rule_Staubach_Pseudo_differential,
    AUTHOR = {Michalowski, Nicholas and Rule, David J. and Staubach,
              Wolfgang},
     TITLE = {Weighted {$L^p$} boundedness of pseudodifferential operators
              and applications},
   JOURNAL = {Canad. Math. Bull.},
  FJOURNAL = {Canadian Mathematical Bulletin. Bulletin Canadien de
              Math\'ematiques},
    VOLUME = {55},
      YEAR = {2012},
    NUMBER = {3},
     PAGES = {555--570},
      ISSN = {0008-4395,1496-4287},
   MRCLASS = {42B20 (35S05 47G30)},
  MRNUMBER = {2957271},
MRREVIEWER = {Masaharu\ Kobayashi},
       DOI = {10.4153/CMB-2011-122-7},
       URL = {https://doi.org/10.4153/CMB-2011-122-7},
}

@article {Alvarez_Bagby_Kurtz_Weighted_estimates,
    AUTHOR = {\'Alvarez, Josefina and Bagby, Richard J. and Kurtz, Douglas
              S. and P\'erez, Carlos},
     TITLE = {Weighted estimates for commutators of linear operators},
   JOURNAL = {Studia Math.},
  FJOURNAL = {Studia Mathematica},
    VOLUME = {104},
      YEAR = {1993},
    NUMBER = {2},
     PAGES = {195--209},
      ISSN = {0039-3223,1730-6337},
   MRCLASS = {47B38 (35S05 42B20 42B25 46E40 47B37 47G10 47G30)},
  MRNUMBER = {1211818},
MRREVIEWER = {Lubo\v s\ Pick},
       DOI = {10.4064/sm-104-2-195-209},
       URL = {https://doi.org/10.4064/sm-104-2-195-209},
}

@article {Bernicot_Frey_Pseudo_semigroup_2014,
    AUTHOR = {Bernicot, Fr\'ed\'eric and Frey, Dorothee},
     TITLE = {Pseudodifferential operators associated with a semigroup of
              operators},
   JOURNAL = {J. Fourier Anal. Appl.},
  FJOURNAL = {The Journal of Fourier Analysis and Applications},
    VOLUME = {20},
      YEAR = {2014},
    NUMBER = {1},
     PAGES = {91--118},
      ISSN = {1069-5869,1531-5851},
   MRCLASS = {35S05 (35H20 35R01 47F05)},
  MRNUMBER = {3180890},
       DOI = {10.1007/s00041-013-9309-y},
       URL = {https://doi.org/10.1007/s00041-013-9309-y},
}

@article {Bahouri_Fermanian_Gallagher_PseudodifferentialHeisenberg,
    AUTHOR = {Bahouri, Hajer and Fermanian-Kammerer, Clotilde and Gallagher,
              Isabelle},
     TITLE = {Phase-space analysis and pseudodifferential calculus on the
              {H}eisenberg group},
   JOURNAL = {Ast\'{e}risque},
  FJOURNAL = {Ast\'{e}risque},
    NUMBER = {342},
      YEAR = {2012},
     PAGES = {vi+127},
      ISSN = {0303-1179},
      ISBN = {978-2-85629-334-8},
   MRCLASS = {43A60 (35R03 35S05 58J42)},
  MRNUMBER = {2952066},
MRREVIEWER = {Franz Luef},
}

@book {Fischer_Ruzhansky_QuantizationBook,
   AUTHOR = {Fischer, Veronique and Ruzhansky, Michael},
     TITLE = {Quantization on nilpotent {L}ie groups},
    SERIES = {Progress in Mathematics},
    VOLUME = {314},
 PUBLISHER = {Birkh\"{a}user/Springer, [Cham]},
      YEAR = {2016},
     PAGES = {xiii+557},
      ISBN = {978-3-319-29557-2; 978-3-319-29558-9},
   MRCLASS = {22E25 (22E30 35R03 35S05 43A80 46L05)},
  MRNUMBER = {3469687},
MRREVIEWER = {Antoni Wawrzy\'{n}czyk},
       DOI = {10.1007/978-3-319-29558-9},
       URL = {https://doi.org/10.1007/978-3-319-29558-9},
}

@article {Cardona_Delgado_Ruzhansky_JGA2021,
AUTHOR = {Cardona, Duv\'{a}n and Delgado, Julio and Ruzhansky, Michael},
     TITLE = {{$L^p$}-{B}ounds for {P}seudo-differential {O}perators on
              {G}raded {L}ie {G}roups},
   JOURNAL = {J. Geom. Anal.},
  FJOURNAL = {Journal of Geometric Analysis},
    VOLUME = {31},
      YEAR = {2021},
    NUMBER = {12},
     PAGES = {11603--11647},
      ISSN = {1050-6926},
   MRCLASS = {22E30 (58J40)},
  MRNUMBER = {4322546},
       DOI = {10.1007/s12220-021-00694-1},
       URL = {https://doi.org/10.1007/s12220-021-00694-1},
}

@article {BagchiThangaveluHermitePseudo,
    AUTHOR = {Bagchi, Sayan and Thangavelu, Sundaram},
     TITLE = {On {H}ermite pseudo-multipliers},
   JOURNAL = {J. Funct. Anal.},
  FJOURNAL = {Journal of Functional Analysis},
    VOLUME = {268},
      YEAR = {2015},
    NUMBER = {1},
     PAGES = {140--170},
      ISSN = {0022-1236},
   MRCLASS = {43A80 (33C45 35J10 42B20 42B25 42B35)},
  MRNUMBER = {3280055},
MRREVIEWER = {Li Zhong Peng},
       DOI = {10.1016/j.jfa.2014.09.020},
       URL = {https://doi.org/10.1016/j.jfa.2014.09.020},
}

@article {EppersonHermitePseudo,
    AUTHOR = {Epperson, Jay},
     TITLE = {Hermite multipliers and pseudo-multipliers},
   JOURNAL = {Proc. Amer. Math. Soc.},
  FJOURNAL = {Proceedings of the American Mathematical Society},
    VOLUME = {124},
      YEAR = {1996},
    NUMBER = {7},
     PAGES = {2061--2068},
      ISSN = {0002-9939},
   MRCLASS = {42B15 (42B20)},
  MRNUMBER = {1343690},
MRREVIEWER = {Andreas Seeger},
       DOI = {10.1090/S0002-9939-96-03486-7},
       URL = {https://doi.org/10.1090/S0002-9939-96-03486-7},
}

@article {Ly_L2_boundedness_JFA,
    AUTHOR = {Ly, Fu Ken},
     TITLE = {On the {$L^2$} boundedness of pseudo-multipliers for {H}ermite
              expansions},
   JOURNAL = {J. Funct. Anal.},
  FJOURNAL = {Journal of Functional Analysis},
    VOLUME = {286},
      YEAR = {2024},
    NUMBER = {2},
     PAGES = {Paper No. 110220, 27},
      ISSN = {0022-1236,1096-0783},
   MRCLASS = {42C15 (33C45 35S05 42B20)},
  MRNUMBER = {4661641},
       DOI = {10.1016/j.jfa.2023.110220},
       URL = {https://doi.org/10.1016/j.jfa.2023.110220},
}

@article {Bui_Hermite_Pseudo_2020,
    AUTHOR = {Bui, The Anh},
     TITLE = {Hermite pseudo-multipliers on new {B}esov and
              {T}riebel-{L}izorkin spaces},
   JOURNAL = {J. Approx. Theory},
  FJOURNAL = {Journal of Approximation Theory},
    VOLUME = {252},
      YEAR = {2020},
     PAGES = {105348, 16},
      ISSN = {0021-9045,1096-0430},
   MRCLASS = {42B35 (42B20)},
  MRNUMBER = {4045217},
MRREVIEWER = {Bae\ Jun\ Park},
       DOI = {10.1016/j.jat.2019.105348},
       URL = {https://doi.org/10.1016/j.jat.2019.105348},
}

@article {Bagchi_Garg_L2_boundedness_JFA,
    AUTHOR = {Bagchi, Sayan and Garg, Rahul},
     TITLE = {On {$L^2$}-boundedness of pseudo-multipliers associated to the
              {G}rushin operator},
   JOURNAL = {J. Funct. Anal.},
  FJOURNAL = {Journal of Functional Analysis},
    VOLUME = {286},
      YEAR = {2024},
    NUMBER = {3},
     PAGES = {Paper No. 110252, 74},
      ISSN = {0022-1236,1096-0783},
   MRCLASS = {58J40 (42B15 43A85)},
  MRNUMBER = {4669590},
       DOI = {10.1016/j.jfa.2023.110252},
       URL = {https://doi.org/10.1016/j.jfa.2023.110252},
}

@article {Bagchi_Basak_Garg_Ghosh_Operator_valued_2024,
    AUTHOR = {Bagchi, Sayan and Basak, Riju and Garg, Rahul and Ghosh,
              Abhishek},
     TITLE = {On some operator-valued {F}ourier pseudo-multipliers
              associated to {G}rushin operators},
   JOURNAL = {J. Math. Anal. Appl.},
  FJOURNAL = {Journal of Mathematical Analysis and Applications},
    VOLUME = {539},
      YEAR = {2024},
    NUMBER = {1},
     PAGES = {Paper No. 128498, 34},
      ISSN = {0022-247X,1096-0813},
   MRCLASS = {58J40 (42B15 42B25 43A85)},
  MRNUMBER = {4749406},
       DOI = {10.1016/j.jmaa.2024.128498},
       URL = {https://doi.org/10.1016/j.jmaa.2024.128498},
}

@article {Bongioanni_Harboure_Salinas_Schrodinger_JFFA_2011,
    AUTHOR = {Bongioanni, B. and Harboure, E. and Salinas, O.},
     TITLE = {Commutators of {R}iesz transforms related to {S}chr\"odinger
              operators},
   JOURNAL = {J. Fourier Anal. Appl.},
  FJOURNAL = {The Journal of Fourier Analysis and Applications},
    VOLUME = {17},
      YEAR = {2011},
    NUMBER = {1},
     PAGES = {115--134},
      ISSN = {1069-5869,1531-5851},
   MRCLASS = {42B35 (35J10)},
  MRNUMBER = {2765594},
MRREVIEWER = {Guoen\ Hu},
       DOI = {10.1007/s00041-010-9133-6},
       URL = {https://doi.org/10.1007/s00041-010-9133-6},
}

@book {rao_Ren_Theory_Orlich_1991,
    AUTHOR = {Rao, M. M. and Ren, Z. D.},
     TITLE = {Theory of {O}rlicz spaces},
    SERIES = {Monographs and Textbooks in Pure and Applied Mathematics},
    VOLUME = {146},
 PUBLISHER = {Marcel Dekker, Inc., New York},
      YEAR = {1991},
     PAGES = {xii+449},
      ISBN = {0-8247-8478-2},
   MRCLASS = {46E30 (46-02 47-02)},
  MRNUMBER = {1113700},
MRREVIEWER = {Ting\ Fu\ Wang},
}

@article{Lrner_Nazarov_Dyadic_Calculus_2019,
 author = {Lerner, Andrei K. and Nazarov, Fedor},
 title = {Intuitive dyadic calculus: the basics},
 fjournal = {Expositiones Mathematicae},
 journal = {Expo. Math.},
 issn = {0723-0869},
 volume = {37},
 number = {3},
 pages = {225--265},
 year = {2019},
 language = {English},
 doi = {10.1016/j.exmath.2018.01.001},
 zbMATH = {7127659},
 Zbl = {1440.42062}
}

@article{Bui_Bui_Duong_square_function_new_weight_2022,
 author = {Bui, The Anh and Bui, The Quan and Duong, Xuan Thinh},
 title = {Quantitative estimates for square functions with new class of weights},
 fjournal = {Potential Analysis},
 journal = {Potential Anal.},
 issn = {0926-2601},
 volume = {57},
 number = {4},
 pages = {545--569},
 year = {2022},
 language = {English},
 doi = {10.1007/s11118-021-09927-y},
 zbMATH = {7622260},
 Zbl = {1509.42021}
}

@article{Lerner_Ombrosi_pointwise_sparse_domination_2020,
 author = {Lerner, Andrei K. and Ombrosi, Sheldy},
 title = {Some remarks on the pointwise sparse domination},
 fjournal = {The Journal of Geometric Analysis},
 journal = {J. Geom. Anal.},
 issn = {1050-6926},
 volume = {30},
 number = {1},
 pages = {1011--1027},
 year = {2020},
 language = {English},
 doi = {10.1007/s12220-019-00172-9},
 zbMATH = {7177340},
 Zbl = {1434.42020}
}

@article{Lerner_Ombrosi_Rivera_Pointwise_Commutators_2017,
 author = {Lerner, Andrei K. and Ombrosi, Sheldy and Rivera-R{\'{\i}}os, Israel P.},
 title = {On pointwise and weighted estimates for commutators of {Calder{\'o}n}-{Zygmund} operators},
 fjournal = {Advances in Mathematics},
 journal = {Adv. Math.},
 issn = {0001-8708},
 volume = {319},
 pages = {153--181},
 year = {2017},
 language = {English},
 doi = {10.1016/j.aim.2017.08.022},
 zbMATH = {6776222},
 Zbl = {1379.42007}
}

@article{Wen-Wu_Xue_Sparse_rho_variation_2022,
 author = {Wen, Yongming and Wu, Huoxiong and Xue, Qingying},
 title = {Sparse domination and weighted inequalities for the {{\(\rho\)}}-variation of singular integrals and commutators},
 fjournal = {The Journal of Geometric Analysis},
 journal = {J. Geom. Anal.},
 issn = {1050-6926},
 volume = {32},
 number = {12},
 pages = {30},
 note = {Id/No 297},
 year = {2022},
 language = {English},
 doi = {10.1007/s12220-022-01059-y},
 zbMATH = {7595996},
 Zbl = {1510.42026}
}

@article{Dziubanski_Zienkiewicz_Hardy_Schrodinger_1999,
 author = {Dziuba{\'n}ski, Jacek and Zienkiewicz, Jacek},
 title = {Hardy space {{\(H^1\)}} associated to {Schr{\"o}dinger} operator with potential satisfying reverse {H{\"o}lder} inequality},
 fjournal = {Revista Matem{\'a}tica Iberoamericana},
 journal = {Rev. Mat. Iberoam.},
 issn = {0213-2230},
 volume = {15},
 number = {2},
 pages = {279--296},
 year = {1999},
 language = {English},
 doi = {10.4171/RMI/257},
 url = {https://eudml.org/doc/39574},
 zbMATH = {1374930},
 Zbl = {0959.47028}
}

@article {Wang_some_remarks_sparse_multilinear_2021,
    AUTHOR = {Wang, Zhidan},
     TITLE = {Some remarks on the sparse dominations for commutators of
              multi(sub)linear operator},
   JOURNAL = {J. Inequal. Appl.},
  FJOURNAL = {Journal of Inequalities and Applications},
      YEAR = {2021},
     PAGES = {Paper No. 51, 11},
      ISSN = {1029-242X},
   MRCLASS = {42B20 (42B25)},
  MRNUMBER = {4232822},
MRREVIEWER = {Jian\ Tan},
       DOI = {10.1186/s13660-021-02585-z},
       URL = {https://doi.org/10.1186/s13660-021-02585-z},
}

@article {Wen_Wu_Xue_pseudo_commutator_2020,
    AUTHOR = {Wen, Yongming and Wu, Huoxiong and Xue, Qingying},
     TITLE = {A note on multilinear pseudo-differential operators and
              iterated commutators},
   JOURNAL = {Bull. Korean Math. Soc.},
  FJOURNAL = {Bulletin of the Korean Mathematical Society},
    VOLUME = {57},
      YEAR = {2020},
    NUMBER = {4},
     PAGES = {851--864},
      ISSN = {1015-8634,2234-3016},
   MRCLASS = {47G30 (42B20 42B25 46E30)},
  MRNUMBER = {4130115},
       DOI = {10.4134/BKMS.b190532},
       URL = {https://doi.org/10.4134/BKMS.b190532},
}

@article {Bui_Bui_Duong_Singular_JGA_2021,
    AUTHOR = {Bui, The Anh and Bui, The Quan and Duong, Xuan Thinh},
     TITLE = {Quantitative weighted estimates for some singular integrals
              related to critical functions},
   JOURNAL = {J. Geom. Anal.},
  FJOURNAL = {Journal of Geometric Analysis},
    VOLUME = {31},
      YEAR = {2021},
    NUMBER = {10},
     PAGES = {10215--10245},
      ISSN = {1050-6926,1559-002X},
   MRCLASS = {42B20 (42B25)},
  MRNUMBER = {4303917},
MRREVIEWER = {Ronghui\ Liu},
       DOI = {10.1007/s12220-021-00641-0},
       URL = {https://doi.org/10.1007/s12220-021-00641-0},
}

@incollection {Perez_Singular_Integrals_Weights_2015,
    AUTHOR = {P\'erez, Carlos},
     TITLE = {Singular integrals and weights},
 BOOKTITLE = {Harmonic and geometric analysis},
    SERIES = {Adv. Courses Math. CRM Barcelona},
     PAGES = {91--143},
 PUBLISHER = {Birkh\"auser/Springer Basel AG, Basel},
      YEAR = {2015},
      ISBN = {978-3-0348-0407-3; 978-3-0348-0408-0},
   MRCLASS = {42B25 (42B20)},
  MRNUMBER = {3364669},
MRREVIEWER = {Albert\ Mas},
}

@article {Cao_Yabuta_Multilinear_Littlewood_Paley_2019,
    AUTHOR = {Cao, Mingming and Yabuta, K\^oz\^o},
     TITLE = {The multilinear {L}ittlewood-{P}aley operators with minimal
              regularity conditions},
   JOURNAL = {J. Fourier Anal. Appl.},
  FJOURNAL = {The Journal of Fourier Analysis and Applications},
    VOLUME = {25},
      YEAR = {2019},
    NUMBER = {3},
     PAGES = {1203--1247},
      ISSN = {1069-5869,1531-5851},
   MRCLASS = {42B25 (42B20)},
  MRNUMBER = {3953502},
MRREVIEWER = {Israel\ P.\ Rivera-R\'ios},
       DOI = {10.1007/s00041-018-9613-7},
       URL = {https://doi.org/10.1007/s00041-018-9613-7},
}

@article {Mihlin_Fourier_Multiplier,
    AUTHOR = {Mihlin, S. G.},
     TITLE = {On the multipliers of {F}ourier integrals},
   JOURNAL = {Dokl. Akad. Nauk SSSR (N.S.)},
  FJOURNAL = {Dokl. Akad. Nauk SSSR (N.S.)},
    VOLUME = {109},
      YEAR = {1956},
     PAGES = {701--703},
   MRCLASS = {42.2X},
  MRNUMBER = {80799},
MRREVIEWER = {A.\ Zygmund},
}

@article {Hormander_Translation_invariant,
    AUTHOR = {H\"ormander, Lars},
     TITLE = {Estimates for translation invariant operators in {$L\sp{p}$}\
              spaces},
   JOURNAL = {Acta Math.},
  FJOURNAL = {Acta Mathematica},
    VOLUME = {104},
      YEAR = {1960},
     PAGES = {93--140},
      ISSN = {0001-5962,1871-2509},
   MRCLASS = {46.00 (42.00)},
  MRNUMBER = {121655},
MRREVIEWER = {I.\ I.\ Hirschman, Jr.},
       DOI = {10.1007/BF02547187},
       URL = {https://doi.org/10.1007/BF02547187},
}

@book {Grafakos_Classical_Fourier,
    AUTHOR = {Grafakos, Loukas},
     TITLE = {Classical {F}ourier analysis},
    SERIES = {Graduate Texts in Mathematics},
    VOLUME = {249},
   EDITION = {Second},
 PUBLISHER = {Springer, New York},
      YEAR = {2008},
     PAGES = {xvi+489},
      ISBN = {978-0-387-09431-1},
   MRCLASS = {42-01 (42Bxx)},
  MRNUMBER = {2445437},
MRREVIEWER = {Andreas\ Seeger},
}

@book {Stein_Weiss_Fourier_Analysis,
    AUTHOR = {Stein, Elias M. and Weiss, Guido},
     TITLE = {Introduction to {F}ourier analysis on {E}uclidean spaces},
    SERIES = {Princeton Mathematical Series},
    VOLUME = {No. 32},
 PUBLISHER = {Princeton University Press, Princeton, NJ},
      YEAR = {1971},
     PAGES = {x+297},
   MRCLASS = {42A92 (31B99 32A99 46F99 47G05)},
  MRNUMBER = {304972},
MRREVIEWER = {Edwin\ Hewitt},
}

@book {Hormander_Pseudo_Differential_operator,
    AUTHOR = {H\"ormander, Lars},
     TITLE = {The analysis of linear partial differential operators. {III}},
    SERIES = {Grundlehren der mathematischen Wissenschaften [Fundamental
              Principles of Mathematical Sciences]},
    VOLUME = {274},
      NOTE = {Pseudodifferential operators},
 PUBLISHER = {Springer-Verlag, Berlin},
      YEAR = {1985},
     PAGES = {viii+525},
      ISBN = {3-540-13828-5},
   MRCLASS = {35-02 (35Sxx 47G05 58G15)},
  MRNUMBER = {781536},
MRREVIEWER = {Min\ You\ Qi},
}

@article {Hebisch_Multiplier_Schrodinger,
    AUTHOR = {Hebisch, Waldemar},
     TITLE = {A multiplier theorem for {S}chr\"odinger operators},
   JOURNAL = {Colloq. Math.},
  FJOURNAL = {Colloquium Mathematicum},
    VOLUME = {60/61},
      YEAR = {1990},
    NUMBER = {2},
     PAGES = {659--664},
      ISSN = {0010-1354,1730-6302},
   MRCLASS = {47F05 (35J10 46E35)},
  MRNUMBER = {1096404},
MRREVIEWER = {Hitoshi\ Kitada},
       DOI = {10.4064/cm-60-61-2-659-664},
       URL = {https://doi.org/10.4064/cm-60-61-2-659-664},
}

@article {Mauceri_Weyl_Transform_JFA,
    AUTHOR = {Mauceri, Giancarlo},
     TITLE = {The {W}eyl transform and bounded operators on {$L\sp{p}({\bf
              R}\sp{n})$}},
   JOURNAL = {J. Functional Analysis},
  FJOURNAL = {Journal of Functional Analysis},
    VOLUME = {39},
      YEAR = {1980},
    NUMBER = {3},
     PAGES = {408--429},
      ISSN = {0022-1236},
   MRCLASS = {44A15 (33A65 42B15 58G15)},
  MRNUMBER = {600625},
MRREVIEWER = {Jos\'e\ L.\ Rubio de Francia},
       DOI = {10.1016/0022-1236(80)90035-X},
       URL = {https://doi.org/10.1016/0022-1236(80)90035-X},
}

@article {Thangavelu_Multiplier_Hermite,
    AUTHOR = {Thangavelu, S.},
     TITLE = {Multipliers for {H}ermite expansions},
   JOURNAL = {Rev. Mat. Iberoamericana},
  FJOURNAL = {Revista Matem\'atica Iberoamericana},
    VOLUME = {3},
      YEAR = {1987},
    NUMBER = {1},
     PAGES = {1--24},
      ISSN = {0213-2230},
   MRCLASS = {42C10 (35J10)},
  MRNUMBER = {1008442},
MRREVIEWER = {Aline\ Bonami},
       DOI = {10.4171/RMI/43},
       URL = {https://doi.org/10.4171/RMI/43},
}

@article {Coifman_Rochberg_Weiss_Commutators,
    AUTHOR = {Coifman, R. R. and Rochberg, R. and Weiss, Guido},
     TITLE = {Factorization theorems for {H}ardy spaces in several
              variables},
   JOURNAL = {Ann. of Math. (2)},
  FJOURNAL = {Annals of Mathematics. Second Series},
    VOLUME = {103},
      YEAR = {1976},
    NUMBER = {3},
     PAGES = {611--635},
      ISSN = {0003-486X},
   MRCLASS = {42A40},
  MRNUMBER = {412721},
MRREVIEWER = {D.\ Sarason},
       DOI = {10.2307/1970954},
       URL = {https://doi.org/10.2307/1970954},
}

@article {Calderon_Commutators,
    AUTHOR = {Calder\'on, A.-P.},
     TITLE = {Commutators of singular integral operators},
   JOURNAL = {Proc. Nat. Acad. Sci. U.S.A.},
  FJOURNAL = {Proceedings of the National Academy of Sciences of the United
              States of America},
    VOLUME = {53},
      YEAR = {1965},
     PAGES = {1092--1099},
      ISSN = {0027-8424},
   MRCLASS = {47.70},
  MRNUMBER = {177312},
MRREVIEWER = {R.\ T.\ Seeley},
       DOI = {10.1073/pnas.53.5.1092},
       URL = {https://doi.org/10.1073/pnas.53.5.1092},
}

@article {Ly_Hermite_Weighted,
    AUTHOR = {Ly, Fu Ken},
     TITLE = {Weighted estimates for {H}ermite pseudo-multipliers with rough
              symbols},
   JOURNAL = {J. Approx. Theory},
  FJOURNAL = {Journal of Approximation Theory},
    VOLUME = {300},
      YEAR = {2024},
     PAGES = {Paper No. 106043, 9},
      ISSN = {0021-9045,1096-0430},
   MRCLASS = {42C15 (33C45 35S05 47G30)},
  MRNUMBER = {4732155},
       DOI = {10.1016/j.jat.2024.106043},
       URL = {https://doi.org/10.1016/j.jat.2024.106043},
}

@article {Uchiyama_Compactness_Commutators,
    AUTHOR = {Uchiyama, Akihito},
     TITLE = {On the compactness of operators of {H}ankel type},
   JOURNAL = {Tohoku Math. J. (2)},
  FJOURNAL = {The Tohoku Mathematical Journal. Second Series},
    VOLUME = {30},
      YEAR = {1978},
    NUMBER = {1},
     PAGES = {163--171},
      ISSN = {0040-8735,2186-585X},
   MRCLASS = {47B37 (44A25)},
  MRNUMBER = {467384},
MRREVIEWER = {Richard\ Rochberg},
       DOI = {10.2748/tmj/1178230105},
       URL = {https://doi.org/10.2748/tmj/1178230105},
}

@article {Duong_Li_Mo_Wu_Yang_compactness_Bessel_2018,
    AUTHOR = {Duong, Xuan Thinh and Li, Ji and Mao, Suzhen and Wu, Huoxiong
              and Yang, Dongyong},
     TITLE = {Compactness of {R}iesz transform commutator associated with
              {B}essel operators},
   JOURNAL = {J. Anal. Math.},
  FJOURNAL = {Journal d'Analyse Math\'ematique},
    VOLUME = {135},
      YEAR = {2018},
    NUMBER = {2},
     PAGES = {639--673},
      ISSN = {0021-7670,1565-8538},
   MRCLASS = {42B10 (47F05)},
  MRNUMBER = {3829612},
       DOI = {10.1007/s11854-018-0048-5},
       URL = {https://doi.org/10.1007/s11854-018-0048-5},
}

@article {Tao_Yang_Yuan_Zhang_compact_ball_Banach_2023,
    AUTHOR = {Tao, Jin and Yang, Dachun and Yuan, Wen and Zhang, Yangyang},
     TITLE = {Compactness characterizations of commutators on ball {B}anach
              function spaces},
   JOURNAL = {Potential Anal.},
  FJOURNAL = {Potential Analysis. An International Journal Devoted to the
              Interactions between Potential Theory, Probability Theory,
              Geometry and Functional Analysis},
    VOLUME = {58},
      YEAR = {2023},
    NUMBER = {4},
     PAGES = {645--679},
      ISSN = {0926-2601,1572-929X},
   MRCLASS = {47B47 (42B20 42B25 42B30 42B35 46E30)},
  MRNUMBER = {4568877},
       DOI = {10.1007/s11118-021-09953-w},
       URL = {https://doi.org/10.1007/s11118-021-09953-w},
}

@article {Chen_Duong_Li_Wu_Compactness_Riesz_Transform_Stratified_group_2019,
    AUTHOR = {Chen, Peng and Duong, Xuan Thinh and Li, Ji and Wu, Qingyan},
     TITLE = {Compactness of {R}iesz transform commutator on stratified
              {L}ie groups},
   JOURNAL = {J. Funct. Anal.},
  FJOURNAL = {Journal of Functional Analysis},
    VOLUME = {277},
      YEAR = {2019},
    NUMBER = {6},
     PAGES = {1639--1676},
      ISSN = {0022-1236,1096-0783},
   MRCLASS = {43A15 (22E30 42B20 43A17 43A80)},
  MRNUMBER = {3985516},
MRREVIEWER = {E.\ K.\ Narayanan},
       DOI = {10.1016/j.jfa.2019.05.008},
       URL = {https://doi.org/10.1016/j.jfa.2019.05.008},
}

@article {Xue_Yabuta_Yan_Kolmogorov_Riesz_theorem,
    AUTHOR = {Xue, Qingying and Yabuta, K\^oz\^o{} and Yan, Jingquan},
     TITLE = {Weighted {F}r\'echet-{K}olmogorov theorem and compactness of
              vector-valued multilinear operators},
   JOURNAL = {J. Geom. Anal.},
  FJOURNAL = {Journal of Geometric Analysis},
    VOLUME = {31},
      YEAR = {2021},
    NUMBER = {10},
     PAGES = {9891--9914},
      ISSN = {1050-6926,1559-002X},
   MRCLASS = {42B20 (42B35 47B07 47G99)},
  MRNUMBER = {4303944},
       DOI = {10.1007/s12220-021-00630-3},
       URL = {https://doi.org/10.1007/s12220-021-00630-3},
}

@article {Bongioanni_Harboure_Salinas_Schrodinger_first,
    AUTHOR = {Bongioanni, B. and Harboure, E. and Salinas, O.},
     TITLE = {Classes of weights related to {S}chr\"odinger operators},
   JOURNAL = {J. Math. Anal. Appl.},
  FJOURNAL = {Journal of Mathematical Analysis and Applications},
    VOLUME = {373},
      YEAR = {2011},
    NUMBER = {2},
     PAGES = {563--579},
      ISSN = {0022-247X,1096-0813},
   MRCLASS = {35J10 (42B25)},
  MRNUMBER = {2720705},
MRREVIEWER = {Michael\ A.\ Perelmuter},
       DOI = {10.1016/j.jmaa.2010.08.008},
       URL = {https://doi.org/10.1016/j.jmaa.2010.08.008},
}

@article {Bongioanni_Harboure_Salinas_Schrodinger_second,
    AUTHOR = {Bongioanni, B. and Harboure, E. and Salinas, O.},
     TITLE = {Weighted inequalities for commutators of
              {S}chr\"odinger-{R}iesz transforms},
   JOURNAL = {J. Math. Anal. Appl.},
  FJOURNAL = {Journal of Mathematical Analysis and Applications},
    VOLUME = {392},
      YEAR = {2012},
    NUMBER = {1},
     PAGES = {6--22},
      ISSN = {0022-247X,1096-0813},
   MRCLASS = {35J10 (35B45 47B47)},
  MRNUMBER = {2914933},
MRREVIEWER = {Shengbing\ Deng},
       DOI = {10.1016/j.jmaa.2012.02.008},
       URL = {https://doi.org/10.1016/j.jmaa.2012.02.008},
}

@article {Bui_Duong_Ly_Hermite_Non_smooth_2026,
    AUTHOR = {Bui, The Anh and Duong, Xuan Thinh and Ly, Fu Ken},
     TITLE = {On {H}ermite pseudo-multipliers with non-smooth kernels},
   JOURNAL = {J. Funct. Anal.},
  FJOURNAL = {Journal of Functional Analysis},
    VOLUME = {290},
      YEAR = {2026},
    NUMBER = {7},
     PAGES = {Paper No. 111323, 47},
      ISSN = {0022-1236,1096-0783},
   MRCLASS = {42B35 (33C45 35S05 42B20)},
  MRNUMBER = {5007889},
       DOI = {10.1016/j.jfa.2025.111323},
       URL = {https://doi.org/10.1016/j.jfa.2025.111323},
}

\end{document}